\documentclass[reqno,10pt]{amsart}
\numberwithin{equation}{section}
\usepackage{amsmath}
\usepackage{amsthm}
\usepackage{color}
\usepackage{pdfsync}
\usepackage[colorlinks,linkcolor={black},citecolor={blue},urlcolor={black}]{hyperref}
\usepackage{amssymb,mathrsfs,enumerate,paralist}
\usepackage[normalem]{ulem}
\usepackage{a4wide}

\newcommand{\R}{\Rz}
\newcommand{\N}{\Nz}
\newcommand{\Rz}{\mathbb{R}}
\newcommand{\Nz}{\mathbb{N}}
\newcommand{\up}{\uparrow}

\newcommand{\disp}{\displaystyle}

\newcommand{\epsi}{\varepsilon}
\newcommand{\eps}{\varepsilon}
\newcommand{\U}{X}
\newcommand{\ls}{|\partial \phi|}       
\newcommand{\rls}{|\partial^- \phi|}    
\newcommand{\rsls}{|\partial^- \phi|}  
\newcommand{\weak}{\buildrel \sigma \over \to} 
\newcommand{\weakto}{\rightharpoonup} 

\newcommand{\loc}{\text{\rm loc}}
\newcommand{\forae}{\text{for a.a. }}

\newcommand{\lsvi}{|\partial \Ve|}
\newcommand{\llsvi}{|\tilde{\partial} \Ve|}

\newcommand{\down}{\downarrow}

\newcommand{\limitingSubdifferential}{\partial_{\ell}}
\newcommand{\lmsbd}{\limitingSubdifferential}

\newcommand{\AC}{\mathrm{AC}}

\newcommand{\Banach}{{\mathcal B}}

\newcommand{\dualoperator}

\def\calD{{\mathcal D}}  
 \def\calH{{\mathcal H}} \def\calI{{\mathcal I}}
  
\def\calM{{\mathcal M}}  
  
  \def\calU{{\mathcal U}}

 \def\rme{{\mathrm e}}

 \def\rmw{{\mathrm w}} 
 
  \def\rmC{{\mathrm C}}
\def\rmD{{\mathrm D}} \def\rmE{{\mathrm E}}

\def\dd{\;\!\mathrm{d}} 

\newcommand{\teta}{\vartheta}

\newcommand{\nchi}{{\raise.2ex\hbox{$\chi$}}}

\newcommand{\MMM}{{\mathcal M}}

\definecolor{dcyan}{rgb}{0,0.4,0.9}
\definecolor{dred}{rgb}{0.6,0,0.2}
\definecolor{ddcyan}{rgb}{0,0.1,0.9}
\definecolor{ddmagenta}{rgb}{0.8,0,0.8}
\definecolor{orange}{rgb}{0.6,0.2,0}

\newcommand{\piecewiseConstant}[2]{\overline{#1}_{\kern-1pt#2}}

\newcommand{\ini}{{\bar{u}}}

\newcommand{\inie}{{\bar{u}_\eps}}
\newcommand{\foraa}{\text{for a.a.}}
\newcommand{\umin}{\ue}

\newcommand{\vmin}{v_\eps}

 \def\trait #1 #2 #3 {\vrule width #1pt height #2pt depth #3pt}

 \def\fin{\hfill
         \trait .3 5 0
         \trait 5 .3 0
         \kern-5pt
         \trait 5 5 -4.7
         \trait 0.3 5 0
 \medskip}
\newtheorem{theorem}{Theorem}[section]
\newtheorem{remark}[theorem]{Remark}
\newtheorem{corollary}[theorem]{Corollary}
\newtheorem{definition}[theorem]{Definition}

\newtheorem{property}[theorem]{Property}
\newtheorem{proposition}[theorem]{Proposition}
\newtheorem{problem}[theorem]{Problem}
\newtheorem{lemma}[theorem]{Lemma}
\newtheorem{notation}[theorem]{Notation}

\newcommand{\upeq}{\vert u_\eps '\vert^2}

\newcommand{\Ve}{V_\eps}
\newcommand{\ue}{u_\eps}
\newcommand{\uek}{u_{\eps_k}}

\newcommand{\Ven}{V_{\eps_n}}
\newcommand{\dphi}{\sfd_\phi}

\newcommand{\wslo}{|\partial_{\rmw}\phi|}
\newcommand{\rwslo}{|\partial_{\mathrm{w}}^{-}\phi|}

\newcommand{\weaksigma}{\buildrel \sigma \over \to}

\newcommand{\glee}{G_\eps}

\newcommand{\II}{[0,\infty)}
\newcommand{\ACini}[2]{\mathscr C_{#2}(#1)} 
\newcommand{\Leb}[1]{\mathscr L^{#1}}

\newcommand{\sfa}{\mathsf a}
\newcommand{\sfA}{\mathsf A}
\newcommand{\sfQ}{\mathsf Q}
\newcommand{\sfb}{\mathsf b}
\newcommand{\sfB}{\mathsf B}
\newcommand{\inftyT}{T}
\newcommand{\sfd}{\mathsf d}
\newcommand{\SFD}{\mathrm D}
\newcommand{\frf}{\mathfrak f}
\newcommand{\frg}{\mathfrak g}
\newenvironment{cor}{\color{red}}{\color{black}}

\newenvironment{perme}{\color{ddcyan}}{\color{black}}

\newcommand{\UUU}{\color{black}}
\newcommand{\GGG}{\color{black}}
\newcommand{\RRR}{\color{black}}

\newcommand{\EEE}{\color{black}}

\newcommand{\RIC}{\color{black}}
\newcommand{\OOO}{\color{black}}

\newcommand{\beper}{\begin{perme}}
\newcommand{\eper}{\end{perme}}
\newcommand{\beroc}{\begin{cor}}
\newcommand{\ecor}{\end{cor}}
\newcommand{\Gelinf}{\mathscr{G}^-}

\begin{document}

\title[\UUU WED principle for gradient flows in  metric spaces ]{\UUU Weighted Energy-Dissipation principle for \EEE \\ gradient flows in  metric spaces}

\pagestyle{myheadings}

\author{Riccarda Rossi}
  \address{DIMI,  Universit\`a di Brescia, via Branze 38, I--25100 Brescia, Italy}
\email{\tt riccarda.rossi@unibs.it}
\urladdr{http://riccarda-rossi.unibs.it}

\author{Giuseppe Savar\'e}
\address{Dipartimento di Matematica, Universit\`a di Pavia, via Ferrata 1, I--27100 Pavia, Italy}
\email{giuseppe.savare@unipv.it}
\urladdr{http://www.imati.cnr.it/savare/}

\author{Antonio Segatti}
\address{Dipartimento di Matematica ``F. Casorati'', via Ferrata 1, I--27100 Pavia, Italy}
\email{antonio.segatti@unipv.it}
\urladdr{http://www-dimat.unipv.it/segatti/ita/res.php}

\author{Ulisse Stefanelli}
\address{\UUU Faculty of Mathematics, University of Vienna,
  Oskar-Morgenstern-Platz 1, A--1090 Vienna, Austria and \EEE Istituto di Matematica Applicata e Tecnologie Informatiche {E. Magenes} --  CNR, via Ferrata 1, I--27100 Pavia, Italy}
\email{ulisse.stefanelli@univie.ac.at}
\urladdr{http://www.mate.univie.ac.at/~stefanelli}

\date{Januray 5, 2018}

\thanks{
G.S.\ has been partially supported by MIUR via PRIN 2015
project 
``Calculus of Variations'', by
Cariplo foundation and Regione Lombardia via project \emph{Variational evolution problems and optimal transport}, and by IMATI-CNR.
 R.\ R.\ and A.\ S.\     acknowledge support from
 GNAMPA (Gruppo Nazionale per l'Analisi Matematica, la Probabilit\`a e le loro Applicazioni)  of  INdAM (Istituto Nazionale di Alta Matematica).
\UUU U.S. acknowledges the support by the Vienna Science and Technology Fund (WWTF)
through Project MA14-009 and by the Austrian Science Fund (FWF)
projects F\,65, I\,2375,  and  P\,27052.}.

 \subjclass[2000]{}

\begin{abstract}
This paper develops the so-called Weighted Energy-Dissipation (WED)
variational approach for the analysis of gradient flows in 
\EEE metric
spaces. This focuses on the minimization of the parameter-dependent
global-in-time functional \RIC of trajectories \GGG
\begin{displaymath}
  \calI_\eps[u] = \int_0^{\infty} \rme^{-t/\eps}\left( \frac12 |u'|^2(t) + \frac1{\eps}\phi(u(t)) \right) \dd t,
\end{displaymath}
 \UUU featuring the weighted sum of energetic and
dissipative terms. As the parameter \OOO $\eps$ is sent to~$0$, \UUU the minimizers $u_\eps$
of such functionals converge, up to subsequences, to curves of maximal
slope \GGG driven by the functional $\phi$. \EEE 
This delivers a new and \GGG general \EEE variational approximation procedure, hence
a  new existence proof, \RIC for metric gradient flows. \EEE 
In addition, it provides a novel perspective
towards relaxation. 
\EEE
\end{abstract}

\maketitle



\section{Introduction}
The study of gradient flows  has attracted remarkable attention since the late '60s, starting from the pioneering works
\cite{Komura67, Crandall-Pazy69, Crandall-Liggett71, Brezis71, Brezis73}, where existence and approximation results
have been established, in the Hilbert framework, for 
\emph{convex} \GGG or $\lambda$-convex  
\EEE driving energy functionals. 
The extension of the existence theory, still in Hilbert spaces, to
\emph{highly nonconvex} (as dominated concave perturbations of convex) energies, see
\cite{Rossi-Savare06}, hinges on the \emph{variational approach}
to gradient flow evolution which in fact dates back to the seminal work by 
\textsc{E.\ De Giorgi} and coworkers
 \cite{DeGiorgiMarinoTosques80, Marino-Saccon-Tosques89, DeGiorgi93} on the theory of \emph{Minimizing Movements} and \emph{Curves of Maximal Slope}.
 This approach is indeed at the heart of the 
  theory of gradient flows in metric spaces 
 \cite{Ambrosio95, Ambrosio-Gigli-Savare08}.   \RIC In turn, this theory  provides the basis \EEE
  for the interpretation \`a la \textsc{Otto} \cite{Jordan-Kinderlehrer-Otto98, Otto01}  of a wide class of 
 evolution equations and systems as gradient flows in the Wasserstein spaces of probability measures, in close connection with the theory of 
 Optimal Transport \cite{Villani09}.
 \par
 The focus of this paper
 is on \emph{gradient flows in metric spaces}, but the motivation stems from the 
 Hilbert theory. \UUU By detailing  the results announced in
 \cite{cras,segatti}, we \EEE
extend to the metric framework
 the analysis carried out, in the context of a Hilbert space $H$, in \cite{MieSte11}. Therein, the
(Cauchy problem for the)  gradient flow
  \begin{equation}
 \label{gflow-hilbert}
 \begin{cases}
 u'(t) + \partial \phi(u(t)) \ni 0 \quad \text{in } H, \qquad \foraa\, t \in (0,T),
 \\
 u(0)= x 
 \end{cases}
 \end{equation}
   driven by a (proper) lower semicontinuous and  $\lambda$-convex energy functional $\phi: H \to (-\infty,\infty]$,
   with $\partial\phi: H \rightrightarrows H$ the subdifferential of $\phi$ in the sense of convex analysis,
   was studied from a novel perspective. Namely, the authors
   considered the   functional of trajectories \UUU $u \in  H^1 (0,T;H)
   \to (-\infty,\infty]$ \EEE defined by
   \begin{equation}
   \label{wed-hilbert}
   \calI_{\eps,T}[u] : = \int_0^T \rme^{-t/\eps}\left( \frac12 |u'(t)|^2 + \frac1{\eps}\phi(u(t)) \right) \dd t,
 \end{equation}
   which features both the \emph{energy} and the (quadratic) \emph{dissipation} terms, with an exponential weight, and is thus referred to as
   \emph{WED (Weighted Energy-Dissipation) functional}. In  \cite{MieSte11} it was
    shown that:
   \emph{(i)}
   for every $\eps>0$ there exists a unique curve $u_\eps$ minimizing
   \UUU the latter functional \EEE over all trajectories starting from a given  datum $x$; 
   \emph{(ii)} the minimizers $u_\eps$ converge as $\eps\down 0$ to the unique solution of \eqref{gflow-hilbert}. 
   \par
\UUU This convergence result resides on the observation that
minimizers $u_\eps$ of $\calI_{\eps,T}$ solve the Euler-Lagrange equation
\begin{equation}
\label{Eu-La-later}
-\eps u''_\eps(t) +  u_\eps'(t) + \partial \phi(u_\eps(t)) \ni 0
\quad \text{in } H, \qquad \foraa \ t \in (0,T). 
\end{equation}
This is nothing but an elliptic-in-time regularization of the gradient
flow problem \eqref{gflow-hilbert}, which is formally recovered by
taking the limit $\eps\to 0$. \EEE
   Besides providing an alternative method for proving existence results for \eqref{gflow-hilbert}, 
   constructing solutions as (limits of) minimizers of functionals on entire trajectories paves the way to 
    \emph{relaxation}. 
    Indeed, the convergence result in   \cite{MieSte11}   
is extended to sequences of \emph{approximate minimizers}, which shows that, in the case $\phi$ is neither $\lambda$-convex nor 
lower semicontinuous, one may consider the relaxation of  $\calI_\eps$ (provided it  is itself the WED functional for a suitable 
lower semicontinuous and $\lambda$-convex energy). 

\UUU
The idea of regularizing evolution problems by singular elliptic
perturbations in time has been
pioneered by \textsc{Lions} \cite{Lions:63a,Lions:63,Lions:65} and used
by  \textsc{Kohn }\& \textsc{Nirenberg} \cite{Kohn65} and \textsc{Olein\u{i}k}
\cite{Oleinik} as a tool to inspect regularity. An account of these
techniques in the linear case can be found in the  book by \textsc{Lions} \& \textsc{Magenes} \cite{LionsMagenes72}.

The variational view to such elliptic regularization via WED functionals has to be traced
back at least to
  \textsc{Ilmanen} \cite{Ilmanen94}, whose proof of existence and partial
  regularity of the  Brakke mean-curvature flow of
varifolds is based on this variational technique. An occurrence of WED
functionals in the proof of existence of periodic solutions to
gradient flows is due to \textsc{Hirano} \cite{Hirano94},  and this variational approach is mentioned in the classical textbook by \textsc{Evans}
\cite[Problem 3, p.~487]{Evans98}. 

The WED approach to gradient flows has been initiated by \textsc{Conti} \&
\textsc{Ortiz} \cite{Conti-Ortiz08}, who presented two examples of
relaxation related with micro-structure evolution. As mentioned, the
corresponding theoretical analysis  in \cite{MieSte11}  and one can
find an early application to the case of mean-curvature evolution
of Cartesian surfaces is in \cite{spst}. An extension of the abstract
theory to nonpotential perturbations of gradient flows is by
\textsc{Melchionna} \cite{Melchionna0}, while $\lambda$-convex energies are
treated in \cite{akst5}. 
\textsc{B\"ogelein, Duzaar, \&
Marcellini}~\cite{BoDuMa14, BDMS17} recently used this variational approach to
prove the existence of variational
solutions to the equation 
\[
u_t -\nabla \cdot f(x,u,\nabla u) + \partial_u f(x,u,
\nabla u)=0
\]
where the field $f$ is convex in
$(u, \nabla u)$,  \RIC see also Sec.\  \ref{ss:appl.1}. \UUU

Doubly nonlinear evolution equations have also been tackled by WED
methods. In the case of rate-independent processes, the abstract
theory is developed by \textsc{Mielke \& Ortiz}
\cite{MieOrt08}, see also the subsequent \cite{MieSte11}, and  an
application to crack-front propagation in brittle materials  has been
presented by \textsc{Larsen, Ortiz, \& Richardson}
\cite{Larsen-et-al09}. 
The rate-dependent case is in turn  \RIC addressed \UUU by the series of
contributions 
\cite{akme,akst,AkaSte11, AkaSte14,akmest}. The reader is additionally referred to
 \textsc{Liero \& Melchionna} \cite{Liero-Melchionna} for
a stability result via $\Gamma$-convergence and to
\textsc{Melchionna} \cite{Melchionna} for an application to the study of
qualitative properties of solutions.

The WED variational approach can be applied to certain classes of
hyperbolic problems as well. Indeed, \textsc{De Giorgi} conjectured this
possibility in the setting of semilinear waves
\cite{DeGio96}. Such conjecture has been positively checked  in \cite{Ste11} for the finite-time case
and by \textsc{Serra \& Tilli} \cite{SerTil12} for the original,
infinite-time case. Extensions to mixed
hyperbolic-parabolic semilinear equations \cite{LieSte13b}, to different
classes of nonlinear energies \cite{LieSte13,Serra-Tilli14}, and to
nonhomogeneous equations \cite{TTilli17} are also available.


The WED approach fits into the general class of global-in-time
variational methods for evolution equations, see  \cite{MieSte11} for
some survey. Among the many options, we shall minimally mention the
celebrated \textsc{Br\'ezis-Ekeland-Nayroles} principle
\cite{BreEke76,BreEke76b, Nay76}, its generalization in the frame of
self-dual Lagrangian theory \cite{Gho09}, and its extensions to
doubly nonlinear \cite{Stefanelli08,Vis11}, maximal-monotone \cite{Vis08,Vis13}, and
pseudo-monotone flows \cite{Vis15}.

\EEE

\par
In this paper we 
 aim  to 
 extend the result from  \cite{MieSte11} to gradient flows set in a
\[
 \text{complete  metric space } (X,\sfd). 
\]
We will  thus  prove, for a reasonably wide class of driving energy functionals $\phi$,  
\GGG
that WED minimizers 
$u_\eps$ of the functional 
\begin{equation}
\label{wed-intro}
\calI_\eps[u] : = \int_0^{\infty} \rme^{-t/\eps}\left( \frac12 |u'|^2(t) + \frac1{\eps}\phi(u(t)) \right) \dd t,
\end{equation}
among a natural class of absolutely continuous curves $u:[0,\infty)\to
X$ satisfying a
given initial condition $u(0)=x$,
converge (up to subsequences) 
to \emph{curves of maximal slope} for
$\phi$, characterized by 
\begin{equation}
\label{cms}
-(\phi\circ u)'(t)  = \frac12 |u'|^2(t) +\frac12  |\partial\phi|^2(u(t)) = |u'|^2(t) = |\partial\phi|^2(u(t))  \qquad \foraa \,  t \in (0,\infty)\,.
\end{equation}
In \eqref{cms},  the \emph{metric derivative} $|u'|(t)$ has to be understood as the metric surrogate of 
the norm $|u'(t)|$, and it indeed replaces the latter in the
associated WED functional \eqref{wed-intro};
$|\partial\phi|$ denotes the \emph{metric slope} of the energy $\phi$;
\eqref{cms} is the differential characterization 
of Curves of Maximal Slope when $|\partial\phi|$ 
is  a strong upper gradient.

\GGG In this way, we show that the WED approach shares the same features
of the well-known
Minimizing Movement
scheme \cite{Ambrosio-Gigli-Savare08}, which relies on a recursive minimization of 
a functional combining distance and energy.
Notice that in a metric space an underlying linear
structure is missing, as well as the Euler equation
\eqref{Eu-La-later}.
Moreover, the dissipation term provided by the metric velocity 
$|u'|^2$ is not required to be a quadratic form on a linear tangent space, so that 
even in a linear framework (e.g.~in a Banach space) the resulting
evolution equation is doubly nonlinear.

One of the basic feature of the WED approach is that it does not
directly involve the distance $\sfd$ but it relies on the notion of 
length (and quadratic action) of a curve. We thus hope that 
our metric strategy could be extended to more general cases of length
structures, 
where the length of a curve may be strongly affected by the geometry
of the sublevels of the functional $\phi$. \EEE

 \par  
 In what follows, we briefly recapitulate the challenges attached to this analysis, and  the main ideas underlying the proof of our main result. This discussion will \RIC also \EEE  make apparent how big the leap is
 between the Hilbert and the metric theory.  
\subsubsection*{\bf The analysis in Hilbert vs.\ metric spaces.}
The starting point in the proof of \cite[Thm.\ 1.1]{MieSte11} 
\GGG in a
Hilbert space $H$
\EEE is the observation that, 
by  the the direct method of calculus of variations 
\GGG and the $\lambda$-convexity of $\phi$, \EEE
the WED functional $\calI_{\eps,T}$ \eqref{wed-hilbert}
admits a unique 
minimizer $u_\eps:[0,T]\to H$ among all trajectories starting from a given initial datum $x \in H$. 
A suitable smoothing argument allows the authors to show that $u_\eps$
 fulfills the  Euler-Lagrange equation \UUU \eqref{Eu-La-later}, \EEE  
supplemented with the initial condition  
$ u_\eps(0)= x $ 
and the  additional Neumann boundary condition
 $ \UUU \eps \EEE
 u_\eps'(T) =0 $ at the final time $T$. 
In fact, \eqref{Eu-La-later}
is the elliptic (in time) regularization of the original gradient flow \eqref{gflow-hilbert}. Its role is twofold: first of all,  from  \eqref{Eu-La-later} it is possible to deduce the 
key estimate
\[
\eps \| u_\eps''\|_{L^2 (0,T;H)} + \eps^{1/2}  \| u_\eps'\|_{L^\infty (0,T;H)} +  \| u_\eps'\|_{L^2 (0,T;H)} + \|\xi_\eps\|_{L^2 (0,T;H)} \leq C,
\]
with $\xi_\eps$ a selection in $\partial\phi(u_\eps)$ satisfying \eqref{Eu-La-later}. Secondly, it is in \eqref{Eu-La-later}  that, exploiting the above estimates,
 it is possible to pass to the limit as $\eps\down0$, proving the convergence of  the curves $(u_\eps)_\eps$ to the (unique) solution of 
 \eqref{gflow-hilbert}. 
 
 The arguments in \cite{MieSte11} clearly rely on \GGG two structural
 properties available in Hilbert spaces: the 
 \emph{linear setting} and the \emph{quadratic norm}. \EEE
 It is far from obvious how to replicate them in the metric context, where
 the gradient flow equation \eqref{gflow-hilbert} is formulated by means of the notion of 
\emph{Curve of Maximal Slope} \eqref{cms}. 
\par
In our metric setup,
\GGG once that the existence of minimizers for \eqref{wed-intro} (among all curves starting from a given initial datum  
$x\in X$) has been proved, 
a nontrivial challenge is 
to provide new metric insights, taking the place of the Euler-Lagrange equation \eqref{Eu-La-later}. 

A first piece of information can be obtained by taking \emph{inner variations}
with respect to time  (namely, perturbations by time rescalings) of a
minimizer $u_\eps$  for $\calI_\eps$, which lead
(cf.\ Proposition \ref{euler-lagrange}) to the metric inner variation equation
\begin{equation}
\label{metric-inner-var-intro}
\frac{\dd}{\dd
t}\left(\phi(u_\eps(t)) -\frac{\eps}2 |u_\eps'|^2(t) \right) =-|u_\eps'|^2(t)   \qquad \text{in } \mathcal{D}'(0,\infty).
\end{equation}
However, \GGG even in a finite dimensional framework 
it is clear that equation  \eqref{metric-inner-var-intro} is much
weaker than the former \eqref{Eu-La-later} and does not contain enough
information to characterize the evolution.
 \par
 We will  borrow the second crucial idea from the \EEE
 theory of  \emph{Optimal control} and \emph{Hamilton-Jacobi equations}, cf.\ e.g.\ \cite{bardi-CapuzzoDolcetta97}. 
In this direction, the key step will be to  work with  the \emph{value functional} associated with the minimum problem 
for  $\calI_\eps$, namely
\begin{equation}
\label{value-functional-intro}
V_\eps(x): = \min_{u \in \mathrm{AC}([0,\infty);X), \, u(0)=x } \calI_\eps[u]\,.
\end{equation}
\GGG
\subsubsection*{\bf Variational properties of the value function and
  its gradient flow.}
From the WED point of view, the value function $V_\eps$ plays a crucial
role, which can be compared to the importance of the Yosida approximation
\begin{equation}
  \label{eq:75}
  \phi_\eps(x):=\min_{y\in X}\left(\frac1{2\eps}\sfd^2(x,y)+\phi(y)\right)
\end{equation}
in the Minimizing Movement approach. 
In order to illustrate the main ideas in a simpler situation, \EEE
 in the following lines we will keep to 
 the finite-dimensional framework $X=\R^n$ and consider a  \emph{smooth}
 energy $\phi:\R^n \to \R$. In this setting, a classical result from the theory of Optimal Control (cf., e.g., \cite[Chap.\ III, Prop.\ 2.5]{bardi-CapuzzoDolcetta97}), 
 ensures that the value function for the \emph{infinite-horizon}
 minimum problem complies with the \emph{Dynamic Programming Principle}. Namely, there holds
 \begin{equation}
 \label{dpp-intro}
 V_\eps(x) = \min_{u \in \mathrm{AC}([0,\infty);\R^n), \, u(0)=x } \left(   \int_0^{T} \rme^{-t/\eps}\left( \frac12 |u'(t)|^2 + \frac1{\eps}\phi(u(t)) \right) \dd t
 + V_\eps(u(T)) \rme^{-T/\eps} \right) 
 \end{equation}
 for all $T>0,$
 and every minimizer $u_\eps $ for  \eqref{wed-intro} is also a  minimizer for the minimum problem 
 \eqref{dpp-intro}, whence 
 \begin{equation}
 \label{dpp-intro-ue}
 V_\eps(x) = \int_0^{T} \rme^{-t/\eps}\left( \frac12 |u_\eps'(t)|^2 + \frac1{\eps}\phi(u_\eps(t)) \right) \dd t
 + V_\eps(u_\eps(T))) \rme^{-T/\eps}  \quad \text{for all }T>0.
 \end{equation}
 We recall the interpretation of formula \eqref{dpp-intro} provided in 
  \cite{bardi-CapuzzoDolcetta97}, viz.\ that, to achieve the minimum cost it is necessary and sufficient to:
  \begin{compactenum}
  \item let the system evolve in an arbitrary finite interval $[0,T] $, along an arbitrary trajectory $u$
  \item pay the corresponding cost, i.e.\ $ \int_0^{T} \rme^{-t/\eps}\left( \frac12 |u'(t)|^2 + \tfrac1{\eps}\phi(u(t)) \right) \dd t$
  \item pay what remains to pay in a optimal way, i.e.  $V_\eps(u(T))) \rme^{-T/\eps}$
  \item minimize over all possible trajectories.
  \end{compactenum} 
  \par
  As we will see now,  the Dynamic Programming Principle \eqref{dpp-intro}
 is the milestone of our analysis.
 Indeed, from \eqref{dpp-intro-ue} one deduces  that for every $0 \leq s \leq t$
 \[
 \begin{aligned}
  \int_s^{t} \rme^{-r/\eps}\left( \frac12 |u_\eps'(r)|^2 + \frac1{\eps}\phi(u_\eps(r)) \right) \dd r &  =   V_\eps(u_\eps(s))) \rme^{-s/\eps} -  V_\eps(u_\eps(t))) \rme^{-t/\eps}  
  \\
   & 
   = - \int_s^t \frac{\dd}{\dd r} \left( V_\eps(u_\eps(r))) \rme^{-r/\eps} \right) \dd r  
   \\
   & = -  \int_s^t \frac{\dd}{\dd r}  (V_\eps(u_\eps(r)) )  \rme^{-r/\eps}  \dd r  + \int_s^t  \frac1{\eps} \rme^{-r/\eps}V_\eps(u_\eps(r)) \dd r.
   \end{aligned}
 \]
 Rearranging terms and using the Lebesgue Theorem we then conclude the \emph{Fundamental identity}
 \begin{equation}
 \label{fund-id-intro}
 -\frac{\dd}{\dd t} V_\eps(u_\eps(t)) = \frac12 |u_\eps'(t)|^2 + \frac1{\eps}\phi(u_\eps(t)) - \frac1{\eps} V_\eps(u_\eps(t)) \qquad \foraa\, t \in (0,\infty)\,.
  \end{equation}
 In fact, 
 \eqref{fund-id-intro} can be combined with another consequence of the Dynamic Programming Principle (cf.\ 
 \cite[Chap.\ III, Thm.\ 2.12]{bardi-CapuzzoDolcetta97}), i.e..\ that the value function $V_\eps$ fulfills the \emph{Hamilton-Jacobi equation}
 \begin{equation}
 \label{HJ-intr}
 \frac1{\eps} V_\eps(x) + \calH(x, \rmD V_\eps(x))=0 \qquad \text{in } \R^n\,,
 \end{equation}
 where 
  the Hamiltonian  \RIC $\calH:  \R^n \times \R^n \to \R$ \EEE  is defined by
 \[
 \calH(x,p): = \sup_{v \in R^n} \left( - x \cdot v - \frac{|v|^2}{2} - \frac1{\eps} \phi(x) \right) =  \frac12|p|^2 -  \frac1{\eps} \phi(x)\,.
 \]
 Hence, \eqref{HJ-intr} yields 
 \begin{equation}
 \label{HJ_Veps}
 \frac12 \left| \rmD  V_\eps(x)\right|^2 = \frac1\eps \phi(x) - \frac1{\eps} V_\eps(x)  \qquad \text{in } \R^n\,.
 \end{equation}
 \par
 Combining \eqref{fund-id-intro}
 and 
 \eqref{HJ_Veps}
we thus arrive at the crucial relation
  \begin{equation}
 \label{Veps-flow}
 -\frac{\dd}{\dd t} V_\eps(u_\eps(t)) = \frac12 |u_\eps'(t)|^2 +  \frac12 \left| \rmD  V_\eps(u_\eps(t))\right|^2 \qquad \foraa\, t \in (0,\infty)\,,
  \end{equation}
 i.e.\ we conclude that \emph{any WED minimizer $u_\eps$ fulfills the gradient flow equation driven by  the value function $V_\eps$}. 
 Since the latter can be thought as an approximation of the energy
 functional $\phi$ \GGG as $\eps\downarrow0$ (as in the case of the
 Yosida regularization \eqref{eq:75}), \EEE
 this argument suggests that a possible way to prove the
 convergence of the WED minimizers to a solution of 
 gradient flow for $\phi$, is to pass to the limit \emph{directly in \eqref{Veps-flow}}.
 We will follow this idea in the metric setting. 
\subsubsection*{\bf Main results.} In the metric framework
we will suppose that  the energy functional $\phi$
complies with the \emph{lower semicontinuity-coercivity-compactness}  (LSCC) conditions, by now standard in the variational approach to 
metric gradient flows, cf.\ \cite[Chap.\ II]{Ambrosio-Gigli-Savare08}, namely
\begin{compactitem}
\item[\textbf{Lower Semicontinuity:}]
 $\phi$ is sequentially
lower semicontinuous; 
\item[\textbf{{Compactness:}}]
 Every $\sfd$-bounded set contained in a sublevel of $\phi$ is
relatively
sequentially compact; 
\item[\textbf{Coercivity:}]
There exists $u_* \in X$ and constants $\sfA,\sfB\ge 0$ such that
$ \phi(u) \geq - \sfB \, \sfd^2(u,u_*) -\sfA$ for all $u \in X$. 
 \end{compactitem}
 \RIC In fact, \EEE  throughout the paper we will work with a generalized
 version of the above conditions, featuring an  interplay between the
 topology induced by the  metric $\sfd$ and a second topology
 $\sigma$, cf.\ \UUU LSCC \EEE Property \ref{basic-ass} ahead. 
\par
 Our first  result, \underline{\bf Theorem \ref{exist-minimizers}},
 ensures that, under the LSCC \UUU  Property \ref{basic-ass}  \EEE there exists    a minimizer for the WED functional 
 $\calI_\eps$ among all trajectories  starting from a given datum $x\in \rmD(\phi)$. Its proof  relies on an 
 \emph{integral compactness criterion}, Theorem \ref{thm:main-compactness} ahead, 
  which  
 establishes suitable compactness properties for 
 any sequence $(u_n)_n \subset \AC_{\mathrm{loc}} ([0,\infty);X)$ such that 
 \begin{equation}
 \label{estimates-4-minimizers}
 \sup_n \int_J |u_n'|^2(t) \dd t  \leq C, \qquad 
 \sup_n \int_J \phi(u_n(t)) \dd t \leq C\,,
 \end{equation}
 for every compact interval $J\subset [0,\infty)$.
Observe that \eqref{estimates-4-minimizers} are indeed  the estimates  that can be
deduced from $\sup_{n\in \N}\calI_\eps [u_n]\leq C$. 
\par
It can be shown that the Dynamic Programming Principle also holds 
for the metric value function \eqref{value-functional-intro}. Then, the calculations 
leading to 
\eqref{fund-id-intro}  carry over to the metric setting, allowing us to conclude the metric analogue
 of 
 \eqref{fund-id-intro} for any WED minimizer $u_\eps$, cf.\ Proposition \ref{prop:intermediate}. Namely, there holds
 \begin{equation}
 \label{fund-metric}
 -\frac{\dd}{\dd t} V_\eps(u_\eps(t)) = \frac12 |u_\eps'|^2(t)+ \frac1{\eps}\phi(u_\eps(t)) - \frac1{\eps} V_\eps(u_\eps(t)) \qquad \foraa\, t \in (0,\infty)\,.
  \end{equation} 
Relation \eqref{fund-metric} is a cornerstone in the proof of \textbf{our main result, \underline{Theorem \ref{th:4.1}}}, here recalled in a slightly simplified form, and without specifying the topology involved in the definition of the lower semicontinuous relaxation $|\partial^-\phi|$ of  the local slope $|\partial\phi |$:  
  \begin{theorem}
  \label{thm:intro}
 Assume the LSCC \UUU Property \ref{basic-ass} \EEE
 and that the \emph{relaxed slope} $|\partial^-\phi|$ is an upper gradient for $\phi$. 
  Let  $ \ini \in \mathrm{D}(\phi)$ and $(u_\eps)_\eps$ be  a family of curves minimizing $\calI_\eps$ among all trajectories starting from $\ini$.
 \par
  Then, for any vanishing sequence $(\eps_k)_k$ the curves $(u_{\eps_k})_k$ pointwise converge on $[0,\infty)$, up to a subsequence, to a curve of maximal slope for $\phi$ (with respect to $|\partial^-\phi|$).
  \end{theorem}
  Let us highlight that the convergence of WED minimizers holds under the very same conditions ensuring the existence of curves of maximal slope for $\phi$, cf.\ \cite{Ambrosio-Gigli-Savare08}. 
  \par
 We now briefly comment on the main ideas underlying the proof of \UUU
 Theorem \ref{thm:intro}: \EEE It is in \eqref{fund-metric} that we shall pass to the limit, as $k\to\infty$, to show that any limit curve $u$ of the sequence  $(u_{\eps_k})_k$ is a curve of maximal slope.  The basic ingredients for taking the limit in an integrated version of  \eqref{fund-metric}
will be  the lower estimates
\[
  \GGG \liminf_{k\to\infty} V_{\eps_k}(u_{\eps_k} \EEE (t)) \geq \phi(u(t)) \qquad \text{for all } t \in [0,\infty),
\]
and 
\[
\liminf_{k\to\infty}  \int_0^t  \frac1{\eps_k} \left( \phi(u_{\eps_k}(s)) {-} V_{\eps_k}(u_{\eps_k}(s)) \right)  \dd s \geq   \int_0^t   \frac12 \ls^2(u(s)) \dd s \qquad \text{for all } t \in [0,\infty).
\]
\par
Finally, observe that relation   \eqref{fund-metric} does contain the information that  any WED minimizer $u_\eps$ 
  is a curve of maximal slope for the value function $\Ve$: this is shown in \underline{\bf Theorem \ref{thm:upper-gradient}}, revealing the upper gradient properties of the quantity $G_\eps = \sqrt{2 \frac{\phi-\Ve}{\eps}}$, see also Thm.\ \ref{cont-wrt-dphi} in the Appendix. What is more, if in addition $\phi$ is \emph{$\lambda$-geodesically convex},
  it can be shown (\underline{\bf Theorem \ref{prop:HJ}})  that  the following \emph{Hamilton-Jacobi identity} holds
  \begin{equation}
  \label{HJ-ID-INTRO}
\frac{1}{\eps} \phi(u) -\frac1\eps \Ve(u)   = |\tilde\partial \Ve|^2(u) \quad \text{for all  } u\in \mathrm{D}(\phi)
\end{equation}
 (with $ |\tilde\partial \Ve|$ a slightly modified version of the local slope of $\Ve$). Hence, \eqref{fund-metric} reads 
  \[
   -\frac{\dd}{\dd t} V_\eps(u_\eps(t)) = \frac12 |u_\eps'|^2(t)+ \frac12  |\tilde\partial \Ve|^2(u_\eps(t)) \qquad \foraa\, t \in (0,\infty)\,,
  \]
  and the analogy with \eqref{Veps-flow} in the Banach framework is complete.
  
  \GGG We will show that \eqref{HJ-ID-INTRO} holds pointwise; 
  it would be interesting to study 
  its formulation in other contexts, e.g.
  in connection with the recently developed theory of \emph{viscosity solutions} to Hamilton-Jacobi equations in metric spaces, cf.\ 
  \cite{Ambrosio-Feng2014, GanSwie15, GiHaNa15}. 
  Notice however that when $(X,\sfd)$ is not locally compact 
  (as it mostly happens for infinite-dimensional dynamics), the  LSCC
  assumptions prevent  the continuity of the driving
  energy $\phi$ and  of the value function $V_\eps$. This gives rise to  technical issues in the
  viscosity approach.
  \begin{remark}
  \upshape
    All the results of the present paper could be easily extended to 
    more general dissipation terms, induced by $p$-powers,
    $1<p<\infty$, or by
    superlinear convex functions $\psi:[0,\infty)\to[0,\infty)$ as in
    \cite[Sect.~2.4]{Rossi-Mielke-Savare08}; they correspond to WED
    functionals of the form
    \begin{displaymath}
      \calI^\psi_\eps[u] : = \int_0^{\infty} \rme^{-t/\eps}\left( \psi
        \big(|u'|(t)\big) + \frac1{\eps}\phi(u(t)) \right) \dd t
\end{displaymath}
and to the metric gradient flow
\begin{displaymath}
  -(\phi\circ u)'(t)  = \psi(|u'|(t)) +\psi^*(|\partial\phi|(u(t)))  \qquad \foraa \,  t \in (0,\infty)\,.
\end{displaymath}
However, in order to keep the presentation simpler, we will only focus on
the case $p=2$, $\psi(v)=\frac 12 v^2$.
  \end{remark}
\EEE
\subsubsection*{\bf Plan of the paper.}
 In \underline{Section \ref{s:3}}, after recalling some basic notions on metric gradient flows, we fix the metric-topological setup of our results, precisely state our assumptions on the energy functional $\phi$, and prove some preliminary results, among which the compactness criterion in Thm.\ \ref{thm:main-compactness}. 
\par
\underline{Section \ref{s:4}} is devoted to  the minimization of the WED functional $\calI_\eps$: the existence of minimizers is shown in Thm.\ \ref{exist-minimizers}, and the metric inner variation equation established in Prop.\ \ref{euler-lagrange}. 
\par
A thorough analysis of the properties of the value function \RIC $V_\eps$ \EEE  is carried out throughout \underline{Section \ref{s:added}}, where in particular we prove that any WED minimizer is a curve  of maximal slope for $\Ve$.
\par
In \underline{Section \ref{sez:passlim}} we finally pass to the limit as $\eps\down 0$ in the gradient flow equation for $\Ve$, and conclude the proof of Theorem \ref{th:4.1}. 
\par
Under the additional $\lambda$-geodesic convexity of $\phi$, in 
\underline{Section \ref{s:aprio}} we prove finer results on WED minimizers $u_\eps$: in particular, we show that, for every fixed $\eps>0$, the mapping $t\mapsto \phi(u_\eps(t))$ enjoys continuity, monotonicity, and convexity  properties akin to those holding for the  function $t\mapsto\phi(u(t))$ \RIC whenever $u$  is  \EEE a curve of maximal slope for $\phi$. 
\par
Keeping the $\lambda$-convexity assumption, in \underline{Section \ref{sez:hamilton-jabobi}} we establish the Hamilton-Jacobi identity \eqref{HJ-ID-INTRO} in Thm.\ \ref{prop:HJ}.
\par
\underline{Section \ref{s:appl}} shows some applications of our
results to \RIC gradient 
flows of nonconvex functonals in Hilbert and Banach
spaces (Sec.\ \ref{ss:appl.1}), and to  a class of  \UUU curves of maximal slope in Wasserstein
spaces of probability measures (Sec.\ \ref{ss:appl.2}). \EEE
\par
Finally, in \underline{Appendix A} we introduce
 a  Finsler-type 
metric on $X$, induced by $\phi$, which will provide further insight
into the properties of $\Ve$. 

\begin{notation}[General notation]
\upshape
Throughout the paper,  we shall use the symbols
$c,\,c',\, C,\,C'$, etc., whose meaning may vary even within the same   line,   to denote various positive constants depending only on
known quantities. Furthermore, the symbols $I_i$,  $i = 0, 1,... $,
will be used as place-holders for several integral terms (or sums of integral terms) involved in
the various estimates: we warn the reader that we will not be
self-consistent with the numbering, so that, for instance, the
symbol $I_1$ will occur several times with different meanings.
\end{notation}

\section{Setup and preliminary results}
\label{s:3}
In this section, 
in order to make the paper as self-contained as possible, we first
collect  some basic definitions and results from the
theory of gradient flows in metric spaces  in Sec.\ \ref{ss:2.1},   referring to \cite{Ambrosio-Gigli-Savare08} 
for all details,  as well as some auxiliary results on the reparameterization of curves  (Sec.\ \ref{ss:2.2}). 
In Sec.\ \ref{subsec:topological} 
we then  state 
the basic lower semicontinuity/coercivity assumptions on the energy $\phi$ under which we shall prove the main  results in this paper,
 also relying on the compactness criterion provided by Theorem \ref{thm:main-compactness}  (Sec.\ \ref{ss:2.4}).  We conclude by fixing some results on the exponential measure  
$$ \UUU \mu_\eps:=
  \frac1\epsi {\mathrm e^{-t/\eps}}\Leb 1, \EEE$$
and the induced weighted Sobolev spaces, that will turn out to be useful in order to study the properties
   of the WED functional $\calI_\eps$ \eqref{wed-intro},  cf.\ Sec.\  \ref{ss:2.6}. 
\par
 Throughout   the paper 
we will assume that 
\[
\text{
$(X,\sfd)$ is a complete metric space.}
\]

\subsection{Recaps on gradient flows in metric spaces}
\label{ss:2.1}
\paragraph{\bf Absolutely continuous curves and metric
derivative.} Let 
$I$ be an interval of $\R$.
We say that a
curve $u: 
 I \to X$ 
belongs to $\AC^p(I;X)$,  $p\in
[1,\infty]$,  if there exists $m\in L^p (I)$ such that
\begin{equation}\label{metric_dev}
\sfd(u(s),u(t))\leq \int_s^t m(r)\dd r \quad \text{for all  $s,t\in I,\ s\leq
t.$}
\end{equation}
For $p=1$, we simply write $\AC(I;X)$ and speak of {\it
absolutely continuous curves}. 
The case $p=\infty$ corresponds to Lipschitz curves.
As usual, $\AC^p_{\rm loc}(I;X)$ will denote the set of 
curves $u:I\to X$ which belong to $\AC^p(J;X)$ for every compact
interval
$J\subset I$.

It  can be proved (see e.g.\ \cite[Sec.\ 1.1]{Ambrosio-Gigli-Savare08}) that for all $u
\in\AC^p(I;X) $, the limit~$$|u'|(t) = \lim_{s\to t}\frac{\sfd(u(s),u(t))}{|t-s|}$$
exists for a.a. $t\in I$. We will refer to it as the {\it metric
derivative} of $u$ at $t$. The map
 $t \mapsto|u'|(t)$  belongs to $ L^p(I) $
and it is minimal within the class of functions $m\in L^p(I)$
fulfilling \eqref{metric_dev}.
\par
 A distinguished class of Lipschitz curves in $\AC^\infty([0,1];X)$ is
provided
by \emph{minimal, constant-speed, geodesics} (for short, geodesics): they are
curves $\gamma:[0,1]\to X$ satisfying
\begin{equation}
  \label{eq:2}
  \sfd(\gamma(s),\gamma(t))=|t-s| \,\sfd(\gamma(0),\gamma(1))\quad
  \text{for every }s,t\in [0,1];
\end{equation}
in particular a geodesic $\gamma$ satisfies
$|\gamma'|(t)\equiv \sfd(\gamma(0),\gamma(1))$ 
for every $t\in (0,1).$
\medskip 

\paragraph{\bf Local  (descending) 
 slope 
 and moderated upper gradient of a l.s.c.~functional.}
Let
 \[
 \begin{gathered}
 \phi : X \to (-\infty,\infty] \quad \text{be lower semicontinuous
with proper domain}\\  D( \phi ):=\{u \in X \ : \  \phi (u) <
\infty\}
\neq \emptyset.
\end{gathered}
\]
The {\it local (descending) slope}
(see \cite{DeGiorgi-Marino-Tosques80,Ambrosio-Gigli-Savare08}) of $  \phi  $ at $ u \in D( \phi ) $
is defined by
$$
|\partial  \phi |(u) := \limsup_{v \to u} \frac{( \phi (u) -
 \phi (v))^+}{\sfd(u,v)}.$$
  As already mentioned in the Introduction, we will also consider on the space $X$ a second 
 (Hausdorff) topology $\sigma$, suitably compatible with that induced by the metric $\sfd$ (cf.\ \eqref{basic-assu}
 ahead).  Accordingly, we will 
   work with 
  the   sequentially  $\sigma$-lower semicontinuous envelope
of the local slope $\ls$
along $d$-bounded sequences with bounded energy, namely
 the
  {\it relaxed slope}
\begin{gather}
 \rsls(u) :=\inf\big\{\liminf_{n\to \infty}\ls(u_n)\ :  \ u_n \weaksigma u,  \  \sup_n (d(u_n,u),\,\phi(u_n)) <\infty \big\}
 \qquad \text{for $u \in D(\phi)$}.
 \label{slope1}
\end{gather}


We recall (cf.\   \cite{Heinonen-Koskela98,Cheeger00}, \cite[Def.~1.2.1]{Ambrosio-Gigli-Savare08}) that 
 a function
 $ g : \U \to [0,\infty]$ is a  {\it (strong) upper gradient} for
  the functional $  \phi  $ if, for every curve
  $ u \in \AC(I;\U)$
  the function $ g \circ u $ is Borel and there holds
\begin{equation}
\label{e:2.3} | \phi (u(t)) -  \phi (u(s))| \leq \int_s^t g(u(r))
|u'|(r)\dd r \quad \text{for all $ 0< s\leq t$}. 
\end{equation}

We now 
now introduce a slightly weaker notion of \emph{upper gradient}
which is well adapted to gradient flows of 
functionals that can assume the 
value $\infty$.
Let  $p\in [1,\infty]$: We say that a function
 $ g : \U \to [0,\infty]$ is an 
  {\it $L^p$-moderated upper gradient} 
for
  the functional $  \phi  $  if for every curve 
$ u \in \AC(I;\U)$
  such that,  in addition, 
  \begin{equation}
    \label{eq:45}
    \phi\circ u\in L^p(I; |u'|\Leb 1),\quad 
    g\circ u\in L^1(I;|u'|\Leb 1)
  \end{equation}
  we still  have \eqref{e:2.3}. 
Observe that, whenever $g$  is a strong, or a $L^p$-moderated, upper gradient, we have that 
$  \phi  \circ u  \in
 \AC(I;\R)$ 
and
$$
 |( \phi  \circ u)'(t)| \leq g(u(t)) |u'|(t) \quad
\forae t \in I.
$$
 \paragraph{\bf Curves of Maximal Slope.}
 Let $g :\U \to [0,\infty]$ be an
 {\em $L^\infty$-moderated upper gradient} 
  for $\phi$ 
 and let $I=\II$. 
 We recall (see~\cite[Def.~1.3.2,~p.32]{Ambrosio-Gigli-Savare08}, following~\cite{DeGiorgi-Marino-Tosques80,Ambrosio95}), that
 a  curve  $ u \in \AC^2_\loc(\II;\U)$ is said to be a {\it curve of
maximal slope} for the functional $  \phi  $ with respect to 
$ g$ if $\phi\circ u$ is locally bounded and
 \begin{equation}
 \label{e:differential-equality}
   -( \phi \circ u)'(t) = |u'|^2 (t)= g^{2}(u(t))
   \qquad \forae t \in (0,
   \infty).
\end{equation}
In particular, $ \phi  \circ u$ is locally absolutely continuous in
$[0,\infty)$, $g \circ u \in L^{2}_\loc([0,\infty))$, and the
energy identity
  \begin{equation}
  \label{abstra-enide}
  \frac1{2} \int_s^t |u'|^2 (r)\dd r + \frac1{2}\int_s^t g^{2}(u(r))\dd r +  \phi (u(t)) =
   \phi (u(s))  \qquad \text{for all $0 \leq s\leq t,$}
\end{equation} directly follows.
 It is interesting that curves of maximal slope
w.r.t.~$g$ can be characterized by an integral condition:
in fact, if a curve $u\in \AC^2_\loc(\II;\U)$ with
$u(0)\in \SFD(\phi)$ satisfies
  \begin{equation}
    \label{eq:1}
    \frac1{2} \int_0^t |u'|^2 (r)\dd r + \frac1{2}\int_0^t
    g^{2}(u(r))\dd r +  
    \phi (u(t)) \le
   \phi (u(0))  \qquad \text{for all $t\ge 0$}
\end{equation}
then $u$ fulfills \eqref{e:differential-equality}. 
Notice that,  for any reasonable definition of gradient flow local boundedness of $\phi\circ u$
is not a restrictive a priori assumption:
this justifies the restriction to the
class of moderated upper gradients in the above  definition of curve of maximal slope. 
\paragraph{\bf  Geodesically $\lambda$-convex functionals.}
A remarkable case in which  the local slope is a (strong) upper
gradient occurs when (cf.~\cite[Thm.~2.4.9]{Ambrosio-Gigli-Savare08}) the functional
$ \phi $ is geodesically $\lambda$-convex for some $\lambda \in
\R$, i.e.
\begin{align}
&\text{for all $v_0,\, v_1 \in \SFD( \phi )$ there exists a
  geodesic $\gamma:[0,1] \to \U$}\nonumber
\quad
\text{with }\gamma(0)=v_0, \ \gamma(1)=v_1, \ \   \text{and }
\\
& \phi (\gamma(t)) \leq (1-t)  \phi (v_0) + t  \phi (v_1)
-\frac{\lambda}{2}t (1-t) \sfd^2(v_0,v_1) \ \ \text{for all
$0 \leq t \leq 1$.} \label{def:l-geod-convex}
\end{align}
%
The following result   subsumes \cite[Cor.~2.4.10, Lemma~2.4.13,
Thm.~2.4.9]{Ambrosio-Gigli-Savare08}. 
\begin{proposition}
\label{prop:l-p-geod} Let $ \phi : \U \to (-\infty,\infty]$ be
$\sfd$-lower semicontinuous and  $\lambda$-geo\-de\-si\-cal\-ly convex
for some $\lambda \in \R$. Then, the local slope $|\partial
 \phi |$ is lower semicontinuous and admits the representation
\begin{equation}
\label{repr-loc-slope} |\partial  \phi |(u)= \sup_{v \neq u} \left(
\frac{ \phi (u) -  \phi (v)}{\sfd(u,v)} + \frac{\lambda}2
\sfd(u,v)\right)^+ \quad \text{for all $u \in \SFD( \phi )$.}
\end{equation}
Furthermore,  $|\partial  \phi |$ is  a (strong) upper gradient,
 and  therefore an $L^p$-moderated upper gradient for every
$p\in [1,\infty]$.\
\end{proposition}
\subsection{Length and energy reparameterization}
\label{ss:2.2}
We recall here a few standard results on reparameterization of
curves,  from which we  will also  derive  a useful criterion 
for 
  a functional  to be an
upper gradient. Lemma  \ref{le:reparametrization} below shall be also used in the proof of Lemma \ref{le:basic} ahead. 
\begin{lemma}[Length and Energy reparameterization]
  \label{le:reparametrization}
  Let $\frg:X\to [0,\infty)$ be a
  Borel
  function and
  let $\vartheta\in \AC([a,b];X)$ be a curve with
  $\int_a^b |\teta'|(t)\,\dd t=L.$ 
  The reparameterized curve
  $\tilde\vartheta:[0,L]\to X$
  \begin{equation}
    \label{eq:22}
    \tilde\vartheta(r):=\vartheta(\kappa(r)) \quad \text{with } 
    \kappa(r):=\inf\Big\{t\in [a,b]:\int_a^t
    |\teta'|(\UUU s \EEE)\,\dd \UUU s \EEE=r\Big\} \text{ for all } 
    r\in [0,L],
  \end{equation}
  is $1$-Lipschitz, $\kappa:[0,L]\to[0,1]$ is continuous,
  nondecreasing and surjective
  (so that $\tilde\vartheta$ has the
  same support as $\vartheta$ with the same initial and final points),
  and
  \begin{equation}
    \label{eq:23}
    \begin{aligned}
    &
    |\tilde\vartheta'|(r)=1\quad
    \text{for $\Leb 1$-a.e.~$r\in [0,L]$,}\quad
\\
&
    \int_{\kappa(r_0)}^{\kappa(r_1)} \frg(\teta(t))|\teta'|(t)\,\dd t=
    \int_{r_0}^{r_1} \frg(\tilde\teta(r))\,\dd r
    \quad \text{for all } 0\le r_0<r_1\le L.
    \end{aligned}
  \end{equation}
  Similarly, if $S:=
  \int_a^b \frac1{\frg(\vartheta(t))}|\vartheta'|(t)\,\dd t<\infty$,
  the reparametrized curve $\vartheta_\frg:[0,S]\to X$
  \begin{equation}
    \label{eq:22bis}
    \vartheta_\frg(s):=\vartheta(\kappa_\frg(s)),\quad
    \kappa_\frg(s):=\inf\Big\{t\in [a,b]:\int_a^t
    \frac1{\frg(\vartheta(t))}\,|\teta'|(r)\,\dd r=s\Big\},\quad
    s\in [0,S],
  \end{equation}
  belongs to $\AC([0,S];X)$, has the same support as
  $\vartheta$ with the same initial and final points, and satisfies
  \begin{equation}
    \label{eq:47bis}
    \frg(\vartheta_\frg(s))=
    \UUU|\vartheta_\frg'|\EEE (s)
    \quad \text{a.e.~in }[0,S].
  \end{equation}
  In particular, if $\int_a^b\frg(\vartheta(t))|\vartheta'|(t)\dd
  t<\infty$
  we have $\vartheta_\frg\in \AC^2([0,S];X)$ and
  \begin{equation}
    \label{eq:47}
    \int_0^S \frg^2(\vartheta_\frg(s))\,\dd s=
    \int_0^S  \UUU |\vartheta_\frg'|^2 (s) \EEE\,\dd s=
    \int_0^S \frg(\vartheta_\frg(s))\,|\vartheta_\frg'|(s)\,\dd s
    =\int_0^1 \frg(\vartheta(t))\,|\vartheta'|(t)\,\dd t.
  \end{equation}
\end{lemma}
\begin{proof}
  The first part of the statement
  is a direct application of the $1$-dimensional
  area formula, see e.g.\ \cite[Rem.\ 4.2.2]{Ambrosio-Tilli00}.

  In order to prove the second part,
   based on the construction \eqref{eq:22} of the curve $\tilde\vartheta$, 
   let us define $\mathsf{s}: [0,L]\to [0,\infty)$ by 
\begin{equation}
\label{reparam-s}
\mathsf{s}(r) : = \int_0^r \frac{1}{\frg(\tilde\vartheta(z)) } \dd z
=\int_0^{\kappa(r)}\frac{1}{\frg(\vartheta(t)) }|\vartheta'|(t) \dd t
\,,\quad
S:=\mathsf{s} (L).
\end{equation}
Since $\frg\circ \tilde\vartheta<\infty$ a.e.~in $(0,L)$  by \eqref{eq:23},  we have
$\mathsf{s}'(r) =  \frac{1}{\frg(\tilde\teta(r)) }  >0$
for almost all $r \in (0,L)$,
so that the map $\mathsf{s}$ is strictly increasing hence invertible,
with inverse $\mathsf{r} : [0,S] \to [0,L]$ satisfying
$\kappa_\frg(s)=\kappa(\mathsf r(s))$ and
$\vartheta_\frg(s) =
\vartheta(\kappa_\frg(s)) = \vartheta (\kappa(\mathsf r(s)))  = 
\tilde\vartheta(\mathsf r(s))$.
Since
\[
| \teta_\frg'|(s) = \frac{|\tilde\teta'| (\mathsf{r}(s))
}{\mathsf{s}'(\mathsf{r}(s)) } = \frg (\tilde\teta(\mathsf{r}(s)))
= \frg(\teta_\frg(s))
\qquad \text{for a.a.~} s \in (0,S),
\]
we get \eqref{eq:47bis}; a further application of
the change of variable formula yields \eqref{eq:47}.
\end{proof}
As a first application of the previous Lemma we have the following
criterion  that will be applied in Section \ref{ss:4.3} to prove
that a certain quantity is a moderated upper gradient for the value
function associated with the WED functional, cf.\ Theorem \ref{thm:upper-gradient} ahead. 
\begin{corollary}[A criterion for upper gradients]
  \label{cor:criterium}
  Let $\phi:X\to (-\infty,\infty]$ be a l.s.c.~functional
  and let $g:X\to [0,\infty]$ be a Borel map such that
  for every curve $\vartheta\in \AC^2([a,b];X)$
  with $\phi\circ \vartheta\in L^{2p}(a,b)$ there holds
  \begin{equation}
    \label{eq:50}
    |\phi(\vartheta(b))-\phi(\vartheta(a))|\le
    \frac 12\int_a^b \Big(|\vartheta'|^2(t)+g(\vartheta(t))\Big)\,\dd t.
  \end{equation}
  Then $g$ is a $L^p$-moderated upper gradient for $\phi$.
\end{corollary}
\begin{proof}
  In order to prove that  $g$
   is an upper gradient, it 
  is not restrictive to check condition \eqref{e:2.3}
  in the case when $I=[0,1]$ and
  $s=0$ $t=1$. We also assume $p<\infty$,  as the
  modifications in the case $p=\infty$ are obvious.

  Let us fix a curve $\vartheta\in \AC([0,1];X)$ such that
  $\phi\circ \vartheta\in L^p(0,1;|\vartheta'|\Leb 1)$ and
  $g\circ \vartheta\in L^1(0,1;|\vartheta'|\Leb 1)$.
  Setting $g_\eps(x):=g(x)+ \eps(1\lor |\phi(x)|)^p$,
  we can apply the second part of
  Lemma \ref{le:reparametrization} to find
  a reparametrized
  curve    $\vartheta_\eps :=\teta_{g_\eps}  \in
  \AC^2([0,S_\eps];X)$ 
  corresponding to  $g_\eps$ 
  such that
  \begin{displaymath}
    \int_0^{S_\eps} g_\eps^2(\vartheta_\eps(s))\,\dd s=
    \int_0^{S_\eps} \UUU |\vartheta_\eps'|^2(s) \EEE\,\dd s=
    \int_0^{S_\eps} g_\eps(\vartheta_\eps(s))\,|\vartheta_\eps'|(s)\,\dd s
    =\int_0^1 g_\eps(\vartheta(t))\,|\vartheta'|(t)\,\dd t.
  \end{displaymath}
  Applying \eqref{eq:50} to the curve $\vartheta_\eps$
  (notice that $\phi\circ \vartheta_\eps
  \in L^{2p}(0,S_\eps)$  since $|\phi(\teta_\eps)|^p \leq g_\eps(\teta_\eps)$),  we obtain
  \begin{align*}
    |\phi(\vartheta(1))-&\phi(\vartheta(0))|=
    |\phi(\vartheta_\eps(S_\eps))-\phi(\vartheta_\eps(0))|\le
    \frac 12
    \int_0^{S_\eps} \Big(|\vartheta_\eps'|^2(s)+g(\vartheta_\eps(s))\Big)\,\dd s
    \\&\le
    \frac 12
    \int_0^{S_\eps}
    \Big(|\vartheta_\eps'|^2(s)+
    g_\eps(\vartheta_\eps(s))\Big)\,\dd s
    =
    \int_0^1
    g_\eps(\vartheta(t))|\vartheta'|(t)\dd t\\&
    =
    \int_0^1
    g(\vartheta(t))|\vartheta'|(t)\dd t+
    \eps \int_0^1(1\lor\phi(\vartheta(t)))^p|\vartheta'|(t)\dd t
  \end{align*}
  Passing to the limit as $\eps\downarrow0$ we conclude  the desired estimate \eqref{e:2.3}. 
\end{proof}

\subsection{\bf A general metric-topological
framework for gradient flows.}
\label{subsec:topological}
Throughout the paper, we will always assume that
 $(X,\sfd)$
is a complete metric space
endowed with
an auxiliary
 Hausdorff topology $\sigma$ on $X$,
that satisfies the following
compatibility conditions:
\begin{subequations}
\label{basic-assu}
\begin{enumerate}[(MT1)]
\item
  $d$ is sequentially
      lower semicontinuous w.r.t.~the product topology
      induced by $\sigma$ on $X\times X$:
\begin{equation}
  \label{weaker-top-1}
  (u_n,v_n) \weaksigma (u,v) \quad \Rightarrow \quad
\liminf_{n \to \infty} \sfd(u_n,v_n) \geq \sfd(u,v);
\end{equation}
\item
  for every $\sigma$-open set $U$ and every $x\in U$
  \begin{equation}
  \text{there
  exist a $\sigma$-open neighborhood $V$ of $x$
  and $\delta>0$ such that \ 
  $\sfd(y,V)<\delta\ \Rightarrow\ y\in U$.}\label{eq:11}
\end{equation}
\end{enumerate}
\end{subequations}
We call $(X,\sfd,\sigma)$ a compatible metric-topological space.
Notice that  (MT2) in particular    shows that $\sigma$ is weaker than
the topology induced by the distance $\sfd$. The possibility to work
with two possibly different topologies allows for a wider
applicability of the theory, as the following examples show.
\begin{remark}[Examples of compatibile metric-topological structures]
  \upshape
  \
  \begin{enumerate}[(E1)]\item
    The above condition is obviously satisfied in the simplest case
    in which $\sigma$ coincides with the topology induced by $\sfd$.
    \item Another
    interesting example is provided by a topology $\sigma$ induced by
    another distance $\sfd'$ satisfying $\sfd'(x,y)\le C \sfd(x,y)$ for every
    $x,y\in X$: this is the typical situation when $X$ is a Banach
    space continuously imbedded in a Banach space $Y$ and $\sfd,\sfd'$ are
    the distances induced by the norms of $X$ and $Y$ respectively.
    \item
      In the previous example,
      $Y$ could also be a Fr\'echet space: consider e.g.\ the case
    when $X=L^p(\Omega)$, $\Omega$ open subset of some Euclidean
    space $\R^m$, and $\sigma$ is the topology
    of $L^p_\loc(\Omega)$, induced by the distance
    \begin{displaymath}
      \sfd'(f,g):=\sum_{n=1}^\infty 2^{-n}\Big(\|f-g\|_{L^p(K_n)}\land
      1\Big),
    \end{displaymath}
    where $K_n\Subset \Omega$
    is a nondecreasing sequence of compact subsets invading $\Omega$.
  \item As a further example, one can consider the
    weak topology in a Banach space $X$, when $\sfd$ is
    the distance induced by its norm. This example
    also highlights that it is interesting  to deal with
    possibly non-metrizable topologies.
  \item Transport distances provide another important example:
    we can consider $\mathcal X=\mathscr P_p(X)$ endowed
    with the $p$-Wasserstein distance $W_p$, $p\in [1,\infty)$, and the
    topology $\sigma$ of weak convergence of probability measures,
     cf.\ e.g.\ \cite[Chap.\ 7]{Ambrosio-Gigli-Savare08}. 
  \end{enumerate}
\end{remark}
\EEE
%
  \begin{subequations}
\begin{property}[Standard lower
  semicontinuity-coercivity-compactness (LSCC) assumptions]
  \label{basic-ass}
We  \UUU say \EEE  that the proper
functional $\phi : X \to (-\infty,\infty] $
satisfies the \emph{standard assumptions} if
the following properties hold:
%
\begin{compactitem}
\item[\textbf{Lower Semicontinuity:}]
 $\phi$ is $\sigma$-sequentially
lower semicontinuous on $\sfd$-bounded sets:
\begin{equation}
\label{basic-ass-1}
\sup_{n,m} \sfd(u_n,u_m) <\infty, \ u_n \weaksigma u  \ \Rightarrow \
\liminf_{n \to \infty} \phi(u_n) \geq \phi(u);
\end{equation}
\item[\textbf{{Compactness:}}]
 Every $\sfd$-bounded set contained in a sublevel of $\phi$ is
relatively
$\sigma$-sequentially compact:  
\begin{equation}
\label{basic-ass-3}
\begin{aligned}
&
\text{if } (u_n)_n \subset X \text{ with } \sup_n \phi(u_n)<\infty, \
 \sup_{n,m} \sfd(u_n,u_m) <\infty,
 \\
 &
 \text{then $(u_n)_n$ admits a $\sigma$-convergent subsequence.}
 \end{aligned}
 \end{equation}
\item[\textbf{Coercivity:}]
There exists $u_* \in X$ and constants $\sfA,\sfB\ge 0$ such that
\begin{equation}
  \label{basic-ass-2bis}
  \phi(u) \geq - \sfB \, \sfd^2(u,v) -\sfQ(v),\quad
  \text{where}\quad
  \sfQ(v):= \sfB\sfd^2(v,u_*)+\sfA\quad
  \text{for every }u,v\in X.
\end{equation}
 \end{compactitem}
\end{property}
\end{subequations}
 Notice that if $\phi$ satisfies
%
\begin{equation}
\label{basic-ass-2} 
\phi(u) \geq - \sfa-\sfb\,  \sfd^2(u,u_*)\quad \text{for every }u\in
X,
\end{equation}
for some $\sfa,\sfb\ge0$ then   \eqref{basic-ass-2bis} 
holds with $\sfA:=\sfa$ and $\sfB:=2\sfb$. 
\par
 The simplest situation
in which Property \ref{basic-ass} holds is provided by a
functional $\phi$ whose sublevels $\{v\in X:\phi(v)\le c\}$ are
compact in $(X,\sfd)$; in this case we can choose $\sigma$ to be  the
topology induced by $\sfd$.
\subsection{An integral compactness criterion.}
\label{ss:2.4}
In this section we adapt to our setting a 
compactness result 
 for sequences of absolutely continuous curves drawn 
from  
\cite[Rmk.\ Extension 1, Thm.\ 4.12]{Rossi-Savare03}.  First and foremost, we shall apply it
 to show with Theorem \ref{exist-minimizers} the existence of minimizers for  the WED functional, relying on  
 a pointwise equicontinuity
estimate. 
Since we have it  at our disposal,  we can provide a simpler and more direct proof of Thm.\  \ref{thm:main-compactness} than that in \cite{Rossi-Savare03}. 
\begin{theorem}
  \label{thm:main-compactness}
  Let $I$ be an interval of $\R$ and 
  let us assume that $\phi:X\to (-\infty,\infty]$ satisfies the 
  standard LSCC \UUU Property \EEE  \ref{basic-ass}.  
  If $(u_n)_n\subset \AC^2_\loc(I;X)$ is a sequence satisfying
  \begin{equation}
    \label{eq:7}
    \sup_n\int_J |u_n'|^2(t)\,\dd t<\infty,\quad
    \sup_n\int_J \phi(u_n(t))\,\dd t<\infty
    \quad 
    \text{for every compact interval }J\subset I,
  \end{equation}
  and $\big(u_n(t_0)\big)_n$ is bounded for some $t_0\in I$, 
  then there exists a limit function $u\in \AC^2_\loc(I;X)$ 
  and a subsequence $k\mapsto n_k$ such that 
  \begin{equation}
    \label{eq:12}
    \text{$u_{n_k}(t)\weaksigma u(t)$ for every $t\in I$ as $k\up\infty$},
  \end{equation}
  \begin{equation}
    \label{eq:3}
    |u_{n_k}'|\weakto v\quad \text{in $L^2_\loc(I)$ as
      $k\up\infty$},\quad
    v\ge |u'|\quad\text{$\Leb 1$-a.e.~in $I$},
  \end{equation}
  \begin{equation}
    \label{eq:14}
    \liminf_{k\to\infty}\int_J\phi(u_{n_k}(t))\,\zeta(t)\,\dd t\ge
    \int_J\phi(u(t))\zeta(t)\,\dd t
  \end{equation}
  {for every nonnegative 
      $\zeta\in \rmC(I)$ and every compact interval }$J\subset I.$
\end{theorem}
\begin{proof}
  By a standard diagonal argument, it is not restrictive to assume 
  that $I=[a,b]$ is a compact interval.
  Since the sequence $(u_n(t_0))_n$ is uniformly bounded,
  we can assume that there exists a constant $C>0$ 
  such that
  \begin{equation}
    \label{eq:8}
    \int_a^b |u_n'|^2(t)\,\dd t\le C,\quad
    \sup_{t\in [a,b]}\sfd(u_n(t),u_*)\le C,\quad
    \int_a^b \psi(u_n(t))\,\dd t\le C
    \qquad\text{for every }n\in \N,
  \end{equation}
  where 
  \begin{equation}
    \label{eq:6}
    \psi(w):=\phi(w)+\sfB\sfd^2(w,u_*)+\sfA
    \ge 0,\qquad
    w\in X,
  \end{equation}
  and $u_*,\sfA,\sfB$ are given by  \eqref{basic-ass-2bis}.
  The first integral bound of \eqref{eq:8} also yields
  \begin{equation}
    \label{eq:9}
    \sfd(u_n(t),u_n(s))\le \int_s^t |u_n'|(r)\,\dd r\le 
    \sqrt {C|t-s|}\quad \text{for every}\quad 
    s,t\in I,\ s\le t,\ n\in \N.
  \end{equation}
  Up to extracting a suitable subsequence, it is not
  restrictive to assume that 
  $|u'_{n}|\weakto v$ as $n\to\infty$  in $L^2(a,b)$.
  
  Let $J_m$, $m\in \N$, be a countable basis of open sets
  in $(a,b)$. 
  We want to find a family of sequences indexed by $m\in \N$,
  that we represent by a map $(m,k)\mapsto n(m,k)\in \N$, 
  and points $t_m\in J_m$ with $t_m\neq t_{m'}$ if $m\neq m'$, such 
  that 
  \begin{itemize}\item 
    for every $m \in \N$ the sequence $k\mapsto n(m+1,k)$ is an increasing
    subsequence of $k\mapsto n(m,k)$ 
    \item $k\mapsto u_{n(m,k)}(t_m)$ is
    converging to some $w_m\in X$ w.r.t.~$\sigma$ as $k\to\infty$.
  \end{itemize}
  We argue by induction w.r.t.~$m$. When $m=0$ we simply 
  set $n(0,k):=k$.
  Assuming that the sequence
  $k\mapsto n(m,k)$ is given for some $m\in \N$,
  Fatou's lemma yields
  \begin{displaymath}
    \int_a^b \liminf_{k\to\infty}\psi(u_{n(m,k)}(t))\dd t\le C
  \end{displaymath}
  so that $\liminf_{k\to\infty}\psi(u_{n(m,k)}(t)) <\infty$ 
  for $\Leb 1$-a.e.~$t\in [a,b]$.
  In particular, since $\Leb 1(J_{m+1})>0$,
  there exists a point $t_{m+1}\in J_{m+1}\setminus \{t_1,\cdots,t_m\}$ 
  such that $\liminf_{k\to\infty}\psi(u_{n(m,k)}(t_{m+1})) <\infty$
  and therefore by \eqref{basic-ass-3} we can find a 
  subsequence $k\mapsto n(m+1,k)$ of 
  $k\mapsto n(m,k)$ and a limit point $w_{m+1}\in X$ 
  such that $u_{n(m+1,k)}(t_{m+1})\weaksigma w_{m+1}$.

  By  a Cantor diagonal argument, 
  we conclude that the sequence $k\mapsto n_k:= n(k,k)$ 
  satisfies
  $u_{n_k}(t_m)\weaksigma w_m$ for every $m\in \N$. 
  Since $J_m$ is a countable basis of open intervals in $(a,b)$,
  the set $D=\{t_m:m\in \N\}$ is countable and dense in $[a,b]$:
  we can then define a function
  $u:D\to X$ by setting $u(t_m):=w_m$.
\par
  Now we can argue as in 
  \cite[Prop.~3.3.1]{Ambrosio-Gigli-Savare08} to conclude, by a careful use 
  of the compatibility conditions \eqref{weaker-top-1} and
  \eqref{eq:11} 
  between $\sfd$ and $\sigma$. In fact, 
  passing to the limit in \eqref{eq:9} thanks to \eqref{weaker-top-1} 
  we get
  \begin{equation}
    \label{eq:13}
    \sfd(u(t),u(s))\le \int_s^t v(r)\,\dd r\le 
    \sqrt {C|t-s|}\quad \text{for every}\quad 
    s,t\in D,\ s\le t.
  \end{equation}
  By \eqref{eq:13} and the completeness of $X$, we can extend $u$ 
  to a curve (still denoted by $u$) defined on $I$
  and still satisfying estimate \eqref{eq:13} for every $s,t\in I$. In
  particular
  $u\in \AC^2(I;X)$ and $|u'|\le v$,  so that \eqref{eq:3} is proved. 
\par
  In order to prove convergence \eqref{eq:12}, we pick an arbitrary point
  $t\in I$ and a $\sigma$-neighborhood $U$ of $u(t)$. 
  Let then $\delta>0$ and $V$ be  as in 
  the compatibility assumption  (MT2)   (with $x=u(t)$).
  Since $u$ is 
  $\sfd$-continuous (and therefore also $\sigma$-continuous)
  we can then find a point $s\in D$ such that 
  $C|t-s|<\delta^2$ and $u(s)\in V$. 
  Since $u_{n_k}(s)\weaksigma u(s)$ as $k\up\infty$ 
  we can also find $\bar k$ sufficiently big such that 
  $u_{n_k}(s)\in V$ for every $k\ge \bar k$. 
  Since $\sfd(u_{n_k}(s),u_{n_k}(t))\le \sqrt{C|t-s|}<\delta$
  for every $k\in \N$ by \eqref{eq:9}, 
  we deduce by \eqref{eq:11} that $u_{n_k}(t)\in U$ 
  for every $k\ge \bar k$.
  
   The lower estimate  \eqref{eq:14} then follows by Fatou's Lemma and
  the fact that $\phi\circ u_n$ is uniformly bounded
  from below by a constant. \UUU The latter boundedness \EEE 
  ensues from  the $\sfd$-uniform 
  boundedness of $u_n$ given by \eqref{eq:8},  combined with 
 \eqref{basic-ass-2bis}. 
\end{proof}
  We can refine \UUU the pointwise convergence \EEE in \eqref{eq:12} by showing that 
  the sequence $(u_{n_k})_k$ is in fact converging in the compact-open
  topology induced by $\sigma$. When $\sigma$ is metrizable,
  this implies the locally uniform convergence of $(u_{n_k})_k$. In fact,   
  this is a general property of any 
  sequence of $\sfd$-equicontinuous functions 
  that pointwise converge in the $\sigma$-topology.
  \begin{lemma}
    \label{le:compact-open}
    Let $(u_k)_{k\in \N}\subset \rmC(I;X)$ be a sequence 
    of locally $\sfd$-equicontinuous functions pointwise converging 
    to $u$ \UUU pointwise \EEE in the $\sigma$-topology as $k\up\infty$.
    Then $(u_k)_k$ converges to $u$ in the compact-open topology 
    induced by $\sigma$.
  \end{lemma}
  \begin{proof}
    Let us consider an arbitrary open
    neighborhood $\calU$ of $u$ in the compact-open topology: this
    corresponds to a finite collection of \UUU compact  \EEE sets $K_m\subset I$
    and corresponding $\sigma$-open sets $U_m\subset X$ such that
    $u(K_m)\subset U_m$, $m\in M:=\{1,2,\cdots, \bar m\}$.  For every
    $t\in K_m$ let $V(t)$ be  a $\sigma$-open neighborhood of $u(t)$ and
    $\delta(t)>0$ satisfying \eqref{eq:11} for $x=u(t)$ and $U=U_m$.
    
    We then find $\eta(t)>0$ with 
    \begin{displaymath}
      \sfd(u_k(r),u_k(s))\le \delta(t)/2\quad \text{for every $r,s\in
        \cup_m K_m$ 
        with $|s-r|\le \eta(t)$ 
    and $k\in \N$},
    \end{displaymath}
    and we set
  $$ W(t):= u^{-1}(V(t))\cap B(t,\eta(t))\cap K_m,\quad \text{where}\
  B(t,\eta):=\{s\in I: |s-t|<\eta\}.$$ 
  Since $\{W(t):t\in K_m\}$ is an
  open covering of $K_m$, we can find a finite   subcovering  $\{W(t):t\in
  J_m\}$ corresponding to a finite set $J_m=\{t_{m,1},\cdots,t_{m,\bar
    h(m)}\}$ of points in $K_m$.  We define
  $\delta_{m,h}:=\delta(t_{m,h})$,
  $\eta_{m,h}:=\eta(t_{m,h})$
  and consider the new collection of
  compact sets $K_{m,h}:=\overline{W(t_{m,h})}\subset I$ and points
  $t_{m,h}\in K_{m,h}$ indexed by integers in $N:=\big\{(m,h)\in
  \N\times \N: m\le \bar m,\ h\le \bar h(m)\big\}$ with the property
  that
  \begin{equation}
    \label{eq:15}
    \begin{gathered}
    \bigcup_{1\le h\le \bar h(m)}K_{m,h}= K_m,\quad
    K_{m,h}\subset B(t_{m,h},\eta_{m,h}),\quad
    \\
    u(K_{m,h})\subset V(t_{m,h})\quad 
    \text{for every } (m,h)\in N.
    \end{gathered}
  \end{equation}
  The neighborhood $\calU$ can then be represented as the set of
  $\sigma$-continuous curves $w:I\to X$ with $w(K_{m,h})\subset U_m$
  for every $(m,h)\in N$.

  Arguing as in the proof of the Theorem \ref{thm:main-compactness},
  we can find $\bar k$ sufficiently big such that $u_k(t_{m,h})\in
  V(t_{m,h})$ for every $k\ge \bar k$ and $(m,h)\in N$.  Since
  $K_{m,h}\subset B(t_{m,h},\eta_{m,h})$, the
  equicontinuity estimate \eqref{eq:9} and \eqref{eq:11} yield
  $u_{k}(K_{m,h})\subset U_m$, thus $u_{k}\in \calU$ for every
  $k\ge\bar k$,  which concludes the proof of the convergence of $(u_k)_k$. 
\end{proof}
\subsection{The exponential measure and weighted Sobolev spaces}
\label{ss:2.6}
In this section we quickly recall a few basic properties
of the Sobolev spaces $W^{1,2}(0,\infty;\mu_\eps)$ induced by the
probability measure
\begin{equation}
  \label{eq:33}
  \mu_\eps:=
  \frac{\mathrm{e}^{-t/\eps}}\eps\Leb 1 \ \ \text{i.e.} \quad \ \ 
  \int_0^\infty \zeta(t)\,\dd\mu_\eps(t):=\int_0^\infty
  \zeta(t)\frac{\rme^{-t/\eps}}\eps\,\dd t\,.
\end{equation}
We say that $w\in W^{1,p}(0,\infty;\mu_\eps)$, $p\in [1,\infty)$, if 
$w\in W^{1,p}_\loc((0,\infty))$ and
\begin{equation}
  \label{eq:21}
  \int_0^\infty \Big(|w(t)|^p+|w'(t)|^p\Big)\,\dd\mu_\eps(t)<\infty.
\end{equation}
Denoting by $\tilde v$ the continuous representative of the function $v$, 
we easily check that $\tilde v\in \AC^p_\loc([0,\infty);\R)$ 
and for 
$v,w\in W^{1,2}(0,\infty;\mu_\eps)$ the following  integration by parts formula holds
\begin{equation}
  \label{eq:34}
  \int_a^b \eps vw'\,\dd\mu_\eps=\int_a^b (-\eps 
  v'+v)w\,\dd\mu_\eps+
  \rme^{-b/\eps} \tilde v(b)\tilde w(b)-\rme^{-a/\eps}\tilde
  v(a)\tilde w(a)\ \text{for all } 0\le a<b<\infty.
\end{equation}
\begin{lemma}
  \label{W11}
  If $w\in \AC_{\mathrm{loc}}([0,\infty);\R)$ with $\int_0^\infty
  |w'(t)|\dd\mu_\eps(t)<\infty$ then $w\in W^{1,1}(0,\infty;\mu_\eps)$
  and
  \begin{align}
    \label{eq:55}
    w(0)+\eps\int_0^T w'(t)\dd\mu_\eps(t)&=w(T)\rme^{-T/\eps}+\int_0^T
    w(t)\dd\mu_\eps(t)\quad \text{for every }T>0,\\
    \label{eq:55bis}
    w(0)+\eps\int_0^\infty w'(t)\dd\mu_\eps(t)&=\int_0^\infty w(t)\dd\mu_\eps(t).
  \end{align}
  In particular, if $u\in W^{1,q}(0,\infty;\mu_\eps)$ and $v\in
  W^{1,p}(0,\infty;\mu_\eps)$ for a couple of conjugate exponents $p,q\in
  [1,\infty]$,  then
\begin{equation}
  \label{eq:35}
    \tilde
    u(0)\tilde v(0)+
  \int_0^\infty \eps u v'\,\dd\mu_\eps=\int_0^\infty (-\eps 
  u'+u)v\,\dd\mu_\eps.
\end{equation}
\end{lemma}
\begin{proof}
Formula 
  \eqref{eq:55} follows from \eqref{eq:34} for $v\equiv 1$.
  Setting $W(t):=\int_0^t
  |w'(r)|\dd r$,
  \eqref{eq:55} yields for every $T>0$
  \begin{equation}
    \label{eq:56}
    \eps\int_0^T |w'(t)|\dd\mu_\eps(t)=\eps \int_0^T W'(t)\dd \mu_\eps(r)=\int_0^T W(r)\dd
    \mu_\eps(r)+\mathrm{e}^{-T/\eps}W(T).
  \end{equation}
  Passing to the limit as $T\up\infty$ we get $W\in
  L^1(0,\infty;\mu_\eps)$
  and, since 
  \begin{equation}
  \label{new-w-added}
  |w(t)|\le |w(0)|+W(t),
  \end{equation}
   we deduce that $w\in
  L^1(0,\infty;\mu_\eps)$.
  Since $\rme^{-t/\eps}W(t)$ has finite integral, its limit set as
  $t\to\infty$ should contain $0$. Therefore, from 
   \eqref{new-w-added}
 we gather
  $\lim_{t\to\infty} \rme^{-t/\eps}W(t) =\lim_{t\to\infty} \rme^{-t/\eps}w(t)= 0$.
  Passing to the limit as $T\up\infty$ in \eqref{eq:55} we get \eqref{eq:55bis}.

Finally,   \eqref{eq:35}  follows by applying  \eqref{eq:55bis} to $w:=uv$.
\end{proof}
Starting from \eqref{eq:35} it is easy to check that 
a function $w\in L^1_\loc(0,\infty)$ belongs to
$W^{1,1}_\loc(0,\infty)$
if and only if 
there exists $g\in L^1_\loc(0,\infty)$ such that 
\begin{equation}
  \label{eq:36}
  \int_0^\infty w(-\eps\xi'+\xi)\,\dd\mu_\eps=\int_0^\infty \eps\,
  g\,\xi\,\dd\mu_\eps
  \quad
  \text{for every $\xi\in \rmC^\infty_c(0,\infty)$,}
\end{equation}
and in this case $w'=g$ in the distributional sense.

In Lemma \ref{l:hynek} below
we compute the sharp constant for the Poincar\'{e}
inequality for real functions in $W^{1,2}(0,\infty;\mu_\eps)$ 
that vanish at $0$: it
will play a crucial role in the next section  in order to derive suitable bounds on infimizing sequences for the WED functional.  
\begin{lemma}
\label{l:hynek}
For every function $w\in \AC^2_\loc([0,\infty);\R)$ 
with $w(0)=0$,
every $\eps>0$ and every $T\in (0,\infty]$ we have
\begin{equation}
\label{e:spectral}
\int_0^{\inftyT} \big|w'(t) \big|^2  \dd \mu_\eps(t) \ge 
\frac1{4\eps^2}\,{ \int_0^{\inftyT} 
\big|w(t)\big|^2 \dd \mu_\eps(t) }.
\end{equation}
In particular, if 
$ \lambda\in (-\infty,1/4\eps^2)$
and $(w_n)_n\subset \AC^2_\loc([0,\infty);\R)$ is a sequence 
satisfying $w_n(0)=0$ and 
\begin{equation}
  \label{eq:28}
 \sup_{n\in \N} \int_0^{\infty} 
  \Big(\big|w_n'(t)
  \big|^2  -\lambda |w_n(t)|^2\Big)\dd \mu_\eps(t) \le
  C<\infty\quad\text{for every }n\in \N,
\end{equation}
then there exists an increasing subsequence $k\mapsto n_k$ 
such that,  as $k\to\infty$,   $(w_{n_k})_k$ converges to $w\in \AC^2_\loc([0,\infty);\R)$
locally uniformly,
$w_{n_k}'\to w'$ weakly in $L^2(0,\infty;\mu_\eps)$, and 
for every $\eta\in (-\infty,1/4\eps^2]$
\begin{equation}
  \label{eq:17}
  \liminf_{k\to\infty}\int_0^{\infty} 
  \Big(\big|w_{n_k}'(t)
  \big|^2  - \eta\, |w_{n_k}(t)|^2\Big)\dd \mu_\eps(t) \ge 
  \int_0^{\infty} 
  \Big(\big|w'(t)
  \big|^2  -\eta\, |w(t)|^2\Big)\dd \mu_\eps(t).
\end{equation}
\end{lemma}
\begin{proof}
Let us first prove \eqref{e:spectral}. For every $\alpha,\beta\ge 0$ we have
\begin{align}
  \notag
  \int_0^T \big|\alpha w'-\beta w\big|^2\dd\mu_\eps&=
  \int_0^T \Big(\alpha^2 \big(w'\big)^2+\beta^2 w^2\Big)\dd\mu_\eps-
  \alpha\beta\int_0^T (w^2)'\dd\mu_\eps\\
  &=
  \label{eq:58}
  \alpha^2\int_0^T \big|w'\big|^2\dd\mu_\eps+\Big(\beta^2  -\frac {\alpha\beta}\eps\Big)
  \int_0^T w^2\dd\mu_\eps-
  \frac{\alpha\beta}{\eps}  \rme^{-T/\eps} w^2(T),
\end{align}
where the second  equality follows from applying  \eqref{eq:55} to $w^2$. 
Choosing $\alpha=1,\ \beta=\frac1{2\eps}$ we get 
\eqref{e:spectral} for finite $T>0$; passing to the limit as $T\up
\infty$  in \eqref{eq:58} and arguing as in the previous Lemma we obtain
 \begin{equation}
  \label{eq:59}
  \int_0^\infty \big|w'\big|^2\dd\mu_\eps=
  (1-  \alpha^2 ) \int_0^\infty \big|w'\big|^2\dd\mu_\eps+
  \Big(\frac {\alpha\beta}\eps-\beta^2 \Big)
  \int_0^\infty w^2\dd\mu_\eps+\int_0^\infty \big|\alpha w'-\beta w\big|^2\dd\mu_\eps.
\end{equation}
Choosing  $\alpha<1$ and $\beta>0$  with 
$\alpha\beta/\eps-\beta^2>\lambda$, 
 \eqref{eq:28} yields that $(w_n)_n$ is uniformly bounded in $W^{1,2}(0,\infty;\mu_\eps)$.
By standard weak compactness we obtain a subsequence $(w_{n_k})_k$ weakly converging
to some limit $w$
in $W^{1,2}(0,\infty;\mu_\eps)$, so that 
$w_{n_k}\weakto w$ and $w_{n_k}'\weakto w'$ weakly in
$L^2(0,\infty;\mu_\eps)$ as $k\to\infty$.  The lower estimate
\eqref{eq:17} then follows from \eqref{eq:59} by choosing $\alpha=1$ and $\beta/\eps-\beta^2=\eta$. 
%
\end{proof}
\begin{remark}
  \upshape
  The optimality of the constant $\frac1{4\eps^2}$ on
  the right-hand side of 
  the inequality \eqref{e:spectral} 
  can be easily checked by
  considering 
  the sequence 
  $w_n(t)=(1\land (n-|t-n|)\lor 0)\rme^{-t/2\eps}$.
\end{remark}

\section{The WED functional, its minimization,  and the main
  convergence result}
\label{s:4}
Let us introduce the functional
$\ell_\eps:
X \times [0,\infty) \to \R$
\begin{equation}
\label{ell_eps}
\ell_\eps(u,v):=  \frac\eps 2 v^2 + 
\phi(u).
\end{equation}
In this section we will investigate the following 
variational problem.
\begin{problem}[The $\eps$-WED variational problem]
  \label{prob:main}
  Given $\eps>0$ and $\bar u\in X$, minimize
  the \emph{weighted energy-dissipation functional}
\begin{equation}
\label{metric-wed}
\mathcal{I}_\eps[u]:= \int_0^{ \infty
 }\ell_\eps(u(t),|u'|(t))\,\dd\mu_\eps(t)=
\int_0^{ \infty } 
\left(\frac\eps2 |u'|^2(t) + \phi(u(t)) \right) \dd \mu_\eps(t),
\end{equation}
over all trajectories $u$ in
\begin{equation}
\label{ac-ini}
\ACini{\ini}\eps :=    \Big\{
u \in  \AC_\mathrm{loc}^2([0,\infty);X)\, : \ u(0)=\ini,\quad
\int_0^\infty|u'|^2(t)\,\dd \mu_\eps(t)<\infty \Big\}.
\end{equation}
We will denote by $ \MMM_\epsi(\ini)$ 
the (possibly empty) set of minimizers of \eqref{metric-wed} in 
$\ACini\ini\eps$.
\end{problem}
Even though we will mainly focus on the WED formulation in $(0,\infty)$,
it will also be useful to consider a localized version of Problem
\ref{prob:main}:
we fix a time $T>0$ and we simply restrict the functional $\calI_\eps$
to curves which are constant in $[T,\infty)$; we thus introduce
\begin{equation}
\label{ac-iniT}
  \ACini{\ini}{\eps,T} :=    \Big\{
  u \in  \AC_\mathrm{loc}^2([0,\infty);X)\, : \ u(0)=\ini,\quad
  u(t)\equiv u(T)\quad\text{in $[T,\infty)$}
  \Big\},
\end{equation}
which is a closed subset of $\ACini\ini\eps$ and could also be
identified with $\AC^2([0,T];X)$; we have the obvious inclusions
\begin{equation}
  \label{eq:57}
  \ACini\ini{\eps,T_1}\subset \ACini\ini{\eps,T_2}\subset
  \ACini\ini{\eps}=\ACini\ini{\eps,\infty}\quad
  \text{whenever}\quad
  0<T_1<T_2<\infty.
\end{equation}
Notice that 
\begin{equation}
  \label{eq:68}
  \mathcal{I}_\eps[u]= \int_0^{ T}
\ell_\eps(u(t),|u'|(t))\,\dd\mu_\eps(t)+
\rme^{-T/\eps}\phi(u(T))\quad\text{if }u\in \ACini\ini{\eps,T}.
\end{equation}
We will denote by  $\MMM_{\eps,T}(\ini)$ the set of minimizers of
$\calI_\eps$ in $\ACini\ini{\eps,T}$.
%
%
%
\subsection{Well-posedness and existence of minimizers
of Problem \ref{prob:main}}
\label{ss:3.1}
First of all, in the metric-topological framework of Section 
\ref{subsec:topological}, 
for $\eps>0$ sufficiently small (depending on 
the constant $\sfB$ in  \eqref{basic-ass-2bis}),  
we
address the well-posedness of Problem \ref{prob:main} 
and the existence of minimizers 
by assuming that Problem \ref{prob:main} is \emph{feasible}, i.e.\ that there 
exists a curve $u\in \ACini\ini\eps$ such that 
$\calI_\eps[u]<\infty$. 
This is always the case when $\ini\in \SFD(\phi)$: in fact, 
\begin{equation}
  \label{eq:4}
  \text{the constant curve $u\in \ACini\ini\eps$, defined by}
  \quad u(t)\equiv \ini\quad \text{$t\ge 0$,}
  \quad
  \text{satisfies}\quad
  \calI_\eps[u]\le \phi(\ini).
\end{equation}
\EEE
\begin{theorem}
  \label{exist-minimizers}
   Let us suppose that $\phi$ satisfies 
  the standard \UUU LSCC Property \EEE \ref{basic-ass} and
  that 
  \begin{equation}
    \label{eq:27}
     \frac 1{16 \eps}\ge\sfB. 
  \end{equation}
  Then
  the integral in \eqref{metric-wed} 
  is well defined (possibly taking the value $\infty$)
  for every $\bar u\in X$ and every $u\in \ACini\ini\eps$.
  
  Moreover, 
if   Problem \ref{prob:main} is feasible
  (in particular when $\ini\in \SFD(\phi)$)
  then
  it admits at least one solution.
  In this case all the sets $\calM_{\eps,T}(\ini)$, $T\in (0,\infty]$, are compact
  in $\ACini\ini\eps$ endowed with 
  the compact-open topology induced by $\sigma$.
\end{theorem}
We divide the proof of Theorem \ref{exist-minimizers}
in a few steps, starting from an immediate application 
of Lemma \ref{l:hynek}. 
Notice that it is sufficient to consider the case $T=\infty$.
\begin{lemma}
  \label{cor:hynek-metric}
  Let $u\in \AC_\loc^2([0,\infty);X)$, 
  $L(t):=\int_0^t |u'|(r)\,\dd r$,
  $\eps>0$, and $\inftyT\in (0,\infty]$. Then
  \begin{equation}
\label{e:spectral-metric}
\int_0^{ \inftyT } 
 |u'|^2(t)  \dd \mu_\eps(t) \geq \frac1{4\eps^2}  
 \int_0^{ \inftyT } 
 L^2(t)  \dd \mu_\eps(t)\ge
 \frac1{4\eps^2}  
 \int_0^{ \inftyT } 
 \sfd^2(u(t),u(0)) \dd \mu_\eps(t).
\end{equation}
In particular, for $\eps>0$ sufficiently small (cf.\ \eqref{eq:27}),  
the integral defining $\mathcal I_\eps[u]$
in \eqref{metric-wed} is well defined for every
$u\in \ACini\ini\eps$.
\end{lemma}
As a further consequence of  Lemma  \ref{cor:hynek-metric}
we provide
 separate estimates for  $ \int_0^{ \infty }
|u'|^2  \dd  \mu_\eps  $ and $ \int_0^{ \infty }
\big(\phi(u)\UUU \big)^+ \EEE \dd  \mu_\eps  $ for any $u\in
\ACini\ini\eps$ such that
$\calI_\eps[u]<\infty$  (recall that $(x)^+ : = x \lor 0$).  Observe that this in fact  requires
absorbing the term $-  \int_0^{ \infty } 
\sfd^2(u(t), u(0)) \dd \mu_\eps(t) $,  which bounds $ \int_0^{ \infty }
\phi(u(t))  \dd \mu_\eps( t)  $ from below (cf.\
\eqref{basic-ass-2bis}), into $ \int_0^{ \infty } \frac \eps 2 |u'|^2
\dd \mu_\eps $.  It is at this level that \eqref{e:spectral-metric} comes into
play.
\begin{lemma}
\label{l:wed-summability}
If $\phi$ satisfies the standard \UUU LSCC Property \EEE 
\ref{basic-ass}
and \eqref{eq:27} holds, 
then for every $u\in \AC^2_{\loc}([0,\infty);X)$  there holds
 \begin{equation}
 \label{coerc-I-e}
 \int_0^\infty 
 \Big(\frac \eps4 |u'|^2+\big(\phi(u(t))\UUU \big)^+ \EEE\Big)\,\dd \mu_\eps(t)\le
 \calI_\eps[u]+\sfQ(u(0)) \qquad \text{with }
 \sfQ(w):=\sfB\,\sfd^2(w,u_*)+\sfA.
 \end{equation}
\end{lemma}
\begin{proof}
Setting $L(t):=\int_0^t |u'|(r)\,\dd r$
we write $\mathcal I_\eps$ as 
\begin{equation}
  \label{eq:29}
  \mathcal I_\eps[u]=
  \frac\eps2\int_0^\infty \Big(|L'|^2-\frac 1{8\eps^2}L^2\Big)\,\dd\mu_\eps+
  \int_0^\infty \psi\,\dd\mu_\eps-\sfQ(u(0)).
\end{equation}
where
\begin{equation}
  \label{eq:30}
  \psi(t):=\phi(u(t))+\frac 1{16 \eps}L^2(t)+\sfQ(u(0)). 
\end{equation}
Since $\psi$ is nonnegative 
thanks to \eqref{basic-ass-2bis}
 and  \eqref{eq:27},  
 we have $\psi(t)\ge (\phi(u(t))\UUU )^+ \EEE $. 
On the other hand, we   have 
\begin{equation}
\label{eq:2911}
\begin{aligned}
  \frac\eps2  \int_0^\infty \Big(|L'|^2-\frac 1{8\eps^2}L^2\Big)\,\dd\mu_\eps & =
 \frac\eps4  \int_0^\infty \Big(|L'|^2-\frac 1{8\eps^2}L^2\Big)\,\dd\mu_\eps
+ \frac\eps4  \int_0^\infty \Big(|L'|^2-\frac 1{8\eps^2}L^2\Big)\,\dd\mu_\eps
\\
& 
\geq 
 \frac\eps4  \int_0^\infty \frac 1{8\eps^2}L^2 \,\dd\mu_\eps
 +\frac\eps4  \int_0^\infty |L'|^2 \,\dd\mu_\eps,
 \end{aligned}
\end{equation}
where the second estimate follows from
 \eqref{e:spectral-metric}. 
 Then, \eqref{coerc-I-e}
 follows. 
\end{proof}
\begin{corollary}[Lower semicontinuity and compactness
  of the functional $\calI_\eps$]
  \label{cor:lscI}
  Let $(u_n)_n  $  be a sequence in $\AC^2_\loc([0,\infty);X)$ such that
  \begin{equation}
    \label{eq:31}
    (u_n(0))_n\quad\text{is bounded},\quad
    \sup_{n\in \N} \mathcal I_\eps[u_n]\le C<\infty.
  \end{equation}
  Then there exists an increasing subsequence $k\mapsto {n_k}$ 
  and a limit function $u\in \AC^2_\loc([0,\infty);X)$ such that 
  the conclusions
  \eqref{eq:12}, \eqref{eq:3} and \eqref{eq:14} 
  of Theorem \ref{thm:main-compactness} hold with $I=[0,\infty)$,
  and moreover
  $\mathcal I_\eps[u]\le C$.
\end{corollary}
\begin{proof}
We can apply Theorem \ref{thm:main-compactness}
thanks to  estimate \eqref{coerc-I-e} combined with \eqref{eq:31}.
In order to prove that $\calI_\eps[u]\le C$ 
we use the splitting in \eqref{eq:29} by choosing 
\begin{displaymath}
  L_n(t):=\int_0^t |u_n'|(r)\,\dd r,\quad
  \psi_n(t):=\phi(u_n(t))+\frac 1{16\eps} L_n^2(t)+\sfQ,
  \quad
  \sfQ\ge \sup_n \sfQ(u_n(0)),
\end{displaymath}
and writing
\begin{equation}
  \label{eq:29bis}
  \mathcal I_\eps[u_n]=
  \frac\eps2\int_0^\infty \Big(|L_n'|^2-\frac 1{8\eps^2}L_n^2\Big)\,\dd\mu_\eps+
  \int_0^\infty \psi_n\,\dd\mu_\eps-\sfQ.
\end{equation}
Denoting by $L(t):=\int_0^tv(r)\,\dd r$ (where $v$ is defined by \eqref{eq:3}), we observe that 
$(L_{n_k})_k$ is pointwise converging to $L\in
\AC^2_\loc(0,\infty;\R)$ with $|L'|\ge |u'|$ 
$\Leb 1$-a.e.~and 
\begin{equation}
\psi(t):=\liminf_{k\to\infty}\psi_{n_k}(t)\ge 
\phi(u(t))+\frac 1{16\eps} L_n^2(t)+\sfQ.\label{eq:19}
\end{equation}
Combining \eqref{eq:17}, Fatou's Lemma (which applies since $\psi_n\ge0$),
and \eqref{eq:19}
we get
\begin{align*}
  C\ge
  \liminf_{k\to\infty} \calI_\eps[u_{n_k}]&\ge 
  \frac\eps2\int_0^\infty \Big(|L'|^2-\frac 1{8\eps^2}L^2\Big)\,\dd\mu_\eps+
  \int_0^\infty \psi\,\dd\mu_\eps-\sfQ\ge
  \calI_\eps[u].
\end{align*}
\end{proof}
The proof of Theorem \ref{exist-minimizers} now
follows  by a
simple application of the Direct method of Calculus of Variations.
\par
 We conclude this section by stating the \underline{\textbf{main result}} of the paper on the convergence of sequences of WED minimizers. Its proof is postponed to Section \ref{ss:6.2}.  
%
\begin{theorem}
\label{th:4.1}
Assume 
Property \ref{basic-ass}.
Let $(\inie)_\eps,
\,\ini \in  D  (\phi)$  fulfill
\begin{equation}
\label{converg-init-data} \inie \weaksigma \ini,
\quad \sup_\eps \sfd(\inie,\ini)<\infty,
 \quad \phi(\inie) \to
\phi(\ini) \ \text{ as $\eps \down 0$.}
\end{equation}
For every $\eps>0$, let $\ue \in   \MMM_\epsi(\ini_\eps) $.

 Then, for any sequence
$(\eps_k)_k $ with $\eps_k \down 0$, there exist a (not relabeled)
subsequence
and \GGG $u \in  \AC_\mathrm{loc}^2([0,\infty);X)$, with $u(0)=\ini$, \EEE  such that
\begin{equation}
\label{pointiwse-conv} \uek(t) \weaksigma u(t) \quad \forall\, t \in
[0,\infty),
\end{equation}
$u(0) = \ini$, and
$u$ fulfills
\begin{equation}
\label{lsc-gflow}
\int_0^t \left( \frac12 |u'|^2(s) + \frac12 \rls^2(u(s)) \right) \dd s + \phi(u(t)) \leq \phi(\ini) \qquad \text{for all } t \geq 0.
\end{equation}
Therefore, if in addition $\rls$ is a ($L^\infty$-moderated)  upper gradient for the functional $\phi$, 
 $u$ is a curve of maximal slope for $\phi$ w.r.t.\ $\rls$.
\end{theorem}
 As already mentioned in the Introduction, a crucial step in the proof of 
Thm.\ \ref{th:4.1} will be to show that WED minimizers are, in a suitable sense discussed at length in Sections \ref{s:added} and \ref{sez:hamilton-jabobi}, metric gradient flows for the value functional $V_\eps$ \eqref{value-functional-intro}. In turn, a key ingredient for this is the \emph{metric inner variation} equation \eqref{euler-lagrange-eq-infinite}, proved in Section \ref{ss:mive} below.  
\subsection{The metric inner variation equation}
\noindent
\label{ss:mive}
\EEE
By taking inner variations of a minimizer of the functional
$\calI_\eps$
we now derive 
a useful equation. 
\begin{proposition}
\label{euler-lagrange}
Let $T\in (0,\infty]$ and let $u$ be a minimizer of $\calI_\eps$ in $\MMM_{\eps,T}({\ini})
$. 
\EEE Then the map 
$t\mapsto \phi(u(t))-\frac \eps2|u'|^2(t)$ belongs to 
 $W^{1,1}(0,T)$ {\rm(}$W^{1,1}_\loc([0,\infty))$ when
$T=\infty${\rm)} 
\EEE
and it fulfills
\begin{equation}
\label{euler-lagrange-eq-infinite} \frac{\dd}{\dd
t}\left(\phi(u(t)) -\frac{\eps}2 |u'|^2(t) \right) =-|u'|^2(t)   \qquad \text{in } \mathcal{D}'(0,T).
\end{equation}
\end{proposition}
\begin{proof}
Following
\cite[Chap.\ III]{Giaquinta-Hildebrandt96},
we consider
perturbations of $u$ obtained by time rescalings, which we devise
by means of the family of smooth diffeomorphisms of $(0,\infty)$
\begin{equation}
  \label{eq:32}
  S_\tau(t):=t+\tau \xi(t),\quad \xi \in \mathrm{C}_\mathrm{c}^\infty (0,T).
\end{equation}
%
Observe that
%
for every $\tau \in \R$ the map $t\mapsto S_\tau(t)$ is
in $\mathrm{C}^\infty (\R)$ with 
smooth inverse $T_\tau=S_\tau^{-1}$ whenever $|\tau|\cdot \max_\R
|\xi'|<1$;
moreover $S_\tau(t)=t$ outside the compact support of $\xi$ in $(0,T)$
and $S_\tau((0,T))=(0,T).$
We then define
 \[
 u_\tau : [0,\infty) \to X \ \text{ by } \ u_\tau(s) := u(T_\tau(s))=u(S_\tau^{-1}(s)).
 \]
 Hence, $u(t) = u_\tau (S_\tau(t))$.
Notice that
\[
|u_\tau'|(s) = |u'|(T_\tau(s))T_\tau'(s)= \frac{|u'|(T_\tau(s))}{S_\tau'(T_\tau(s))}
\quad \foraa\,s \in
(0,\infty).
\]
We can compute $\calI_\eps[u_\tau]$ by applying
a standard change of variables
\begin{displaymath}
\begin{aligned}
  {\mathcal I} [u_\tau] &  = 
  \int_0^{\infty}  \rme^{-s/\eps} \left(\frac12 |u_\tau'|^2(s) +
    \frac1\eps \phi(u_\tau(s)) \right) \dd s
  =\int_0^{\infty}  \rme^{-s/\eps} \left(\frac12 
    \Big(\frac{|u'|(T_\tau(s))}{S_\tau'(T_\tau(s))}\Big)^2 +
    \frac1\eps \phi(u_\tau(s)) \right) \dd s\\  & = \int_0^{\infty}
  \rme^{-{S_\tau(t)}/\eps} \left(\frac12 \frac{|u'|^2(t)}{S_\tau'(t) } +
    \frac1\eps \phi(u(t)) S_\tau'(t)\right) \dd t
  \end{aligned}
\end{displaymath}
and we recover the  metric inner variation equation 
\eqref{euler-lagrange-eq-infinite} by taking the derivative of
$\calI_\eps[u_\tau]$ 
w.r.t.~$\tau$ at the minimum point $\tau=0$.
We have
\[
\begin{aligned}
   \frac{\dd}{\dd \tau} \calI_\eps[u_\tau]&=
   \int_0^{\infty} \rme^{-{S_\tau(t)}/\eps}\left(
     -\frac1\eps \frac{\partial}{\partial \tau}S_\tau(t)\right) \left(\frac12 \frac{|u'|^2(t)}{S_\tau'(t) } +
     \frac1\eps \phi(u(t)) S_\tau'(t)\right) \dd t\\  & 
    \quad + \int_0^{\infty} \rme^{-{S_\tau(t)}/\eps} \left(-\frac12 
      \frac{|u'|^2(t)}{(S_\tau'(t))^2}+\frac1\eps\phi(u(t))\right)
      \frac{\partial}{\partial \tau}S_\tau'(t)\,\dd t
\end{aligned}
\]
 Setting  $\tau=0$ and taking into account that
\[
S_0(t)=t,\quad
S_\tau'(t)=1+\tau\xi'(t), \quad 
 \frac{\partial}{\partial \tau}S_\tau(t)=\xi(t),\quad\frac{\partial}{\partial
  \tau}S'_\tau(t) =\xi'(t),
\]
we conclude that
\[
\begin{aligned}   0=\frac{\dd}{\dd \tau} \calI_\eps[u_\tau]
  \big|_{\tau=0}& = -\int_0^{\infty}
  \left(\frac12
    |u'|^2 + \frac1\eps \phi\circ u \right) \xi\dd \mu_\eps+
\int_0^{\infty}
\left(-\frac\eps2 |u'|^2 +
\phi\circ u \right)\xi' \dd \mu_\eps
\\&=
\int_0^{\infty} 
\left[
  -|u'|^2\xi + 
  \left(-\frac\eps2 |u'|^2 +
      \phi\circ u\right)\Big(\xi'-\frac1\eps \xi\Big) \right]\dd \mu_\eps.
\end{aligned}
\]
Since $\xi\in \rmC^\infty_c(0,T)$ is arbitrary,
an integration by parts as stated in \eqref{eq:36} yields
\eqref{euler-lagrange-eq-infinite}.
\end{proof}
\begin{corollary}
  \label{cor:boundary}
  Let $u_{\eps,T}\in \MMM_{\eps,T}(\ini)$, $T\in (0,\infty]$, and let us denote by 
  $\mathcal V_{\eps,T}$ the absolutely continuous representative of 
  $t \mapsto \phi( u_{\eps,T}(t)) -\frac \eps2 |u_{\eps,T}'|^2(t)$ in the interval $[0,T]$
  (we simply write $u_\eps$ and $\mathcal V_\eps$ when $T=\infty$).
  Then we have
  \begin{equation}
    \label{eq:66}
    \mathcal I_\eps[u_\eps]=\mathcal V_\eps(0)\quad \text{if }T=\infty,
  \end{equation}
  and
  \begin{equation}
    \label{eq:64}
    \mathcal I_\eps[u_{\eps,T}]=\mathcal V_{\eps,T}(0)+
    \rme^{-T/\eps}\Big(\phi(u_{\eps,T}(T))-\mathcal V_{\eps,T}(T)\Big)
    \quad\text{if }T<\infty.
  \end{equation}
\end{corollary}
\begin{proof}
In the case $T=\infty$ the inner variation equation
\eqref{euler-lagrange-eq-infinite} gives that the distributional derivative of $\mathcal{V}_\eps$ fulfills
$\frac{\dd}{\dd t} \mathcal{V}_\eps(t) = -\vert u_\eps'\vert^2(t)
\in\,L^1(0,\infty;\mu_\eps)$. Hence,
$\mathcal{V}_\eps\in W^{1,1}(0,\infty;\mu_\eps)$ 
so that the identity
\begin{displaymath}
  \ell(u_\eps(t),|u_\eps'|(t))=\mathcal V_\eps(t)-\eps\mathcal
  V_\eps'(t)\quad\text{a.e.~in }(0,\infty)
\end{displaymath}
yields, by the integration by parts formula \eqref{eq:34}, that
\begin{displaymath}
\mathcal I_\eps[u_\eps] =\int_0^\infty
\ell(u_\eps,|u_\eps'|)\dd\mu_\eps=
\int_{0}^{\infty}\Big(\mathcal V_\eps(t)-\eps\mathcal
V_\eps'(t)\Big) \dd \mu_\eps(t) 
=\mathcal{V}_\eps(0). 
\end{displaymath}
A similar argument leads  to \eqref{eq:64}.
\end{proof}
\section{\bf The value function and its properties}
\label{s:added}
As we mentioned in the Introduction,
 Problem  \ref{prob:main}  can be interpreted in the framework
of   optimal control theory,  as the simplest 
\emph{infinite-horizon} problem, cf. \cite[Chap.
III]{bardi-CapuzzoDolcetta97}. In this connection,
 the associated
 \emph{value function} 
 $\Ve: X \to [0,\infty]$
\begin{equation}
\label{the-value-function} \Ve(x):=
\inf_{u \in
 \ACini x\eps}\mathcal{I}_\eps[u]=
\inf_{u \in
 \ACini x\eps}\int_0^\infty \ell_\eps(u,|u'|)\,\dd\mu_\eps,
\quad x\in X,
 \end{equation}
 will play a crucial role. 
  As usual, we will always suppose
 that $\phi$ satisfies the standard \UUU LSCC Property \EEE \ref{basic-ass}
 and
 $\frac 1{16\eps}>\sfB$;  we also set 
$\rmD(V_\eps):=\big\{x\in X:\Ve(x)<\infty\big\}.$ 

It will also be useful to deal with the corresponding functional
associated with   the finite-horizon functional from   \eqref{eq:68}, namely: 
\begin{equation}
  \label{VeT} V_{\eps,T}(x):=
\inf_{u \in
 \ACini x{\eps,T}}\mathcal{I}_\eps[u]=
\inf_{u \in
 \ACini x{\eps,T}}\int_0^T \ell_\eps(u,|u'|)\,\dd\mu_\eps+\mathrm{e}^{-T/\eps}\phi(u(T)),
\quad x\in X,
 \end{equation}

 In this section, we first address some general properties of $\Ve$. Then, we use the Dynamic Programming Principle
 (cf.\ Proposition \ref{prop:DDP} below) to derive a fundamental equation satisfied by $\Ve$ evaluated along
  \emph{any} minimizer $u$ 
  for   \eqref{the-value-function}, cf.\ \eqref{intermediate-relation} below.  Then,
 with the aid of Theorem \ref{thm:upper-gradient} ahead,  we  will read from \eqref{intermediate-relation} that WED-minimizers are curves of maximal slope of the
 value function $\Ve$, in a suitable sense  (cf.\ Corollary  \ref{cor:WED-GF}). 

Our first result guarantees that the functional $\Ve$ is quadratically
bounded from below
and lower semicontinuous with respect to the $\sigma$-topology,
at least on bounded subsets of $X$.
\begin{lemma}
\label{lemma1} 
Let us suppose that $\phi$ satisfies the standard \UUU LSCC Property \EEE  \ref{basic-ass} and that   $\frac1{16 \eps}\ge \sfB$.  Then
the infimum in \eqref{the-value-function} is attained for every $x\in
\rmD(\Ve)$
and $V_\eps$ itself satisfies the standard \UUU LSCC Property \EEE  \ref{basic-ass};
in particular, 
$\Ve$ is sequential $\sigma$-lower semicontinuous on 
$\sfd$-bounded sets of $X$, and
\begin{equation}
  \label{stime-V}
 \phi(x)\ge \Ve(x)\ge -\sfQ(x)=-\sfA-\sfB\,\sfd^2(x,u_*)\quad \text{for every
 }x\in X.
\end{equation}
\end{lemma}
\begin{proof}
  The first two statements are immediate consequences of 
  Theorem \ref{exist-minimizers} and Corollary
  \ref{cor:lscI}. 
 Estimate \eqref{stime-V} follows from \eqref{eq:4},
  the representation  formula \eqref{eq:29} for $\calI_\eps$
   and the positivity of the first two integral terms in  \eqref{eq:29}, cf.\ also \eqref{eq:2911}. 
%
 \end{proof}

\subsection{\bf The \UUU Dynamic Programming Principle \EEE and its consequences}
\label{s:added-2}
Interpreting the WED minimum problem \eqref{the-value-function}
in the light of the theory of
 optimal control
 provides the following key result. 
\begin{proposition}[Dynamic Programming Principle]
  \label{prop:DDP} 
  If $\phi$ satisfies 
  the standard \UUU LSCC Property \EEE \ref{basic-ass}, 
  $\frac 1{16 \eps}\ge\sfB$, and $x\in \rmD(\Ve)$ then
\begin{equation}
\label{programmazione-dinamica} \Ve(x)= 
\min_{u \in \ACini x\eps}\left( \int_0^T \ell_\eps(u,|u'|) \dd \mu_\eps
  + \Ve(u(T)) \rme^{-T/\eps}\right)
\quad\text{for every }T>0.
\end{equation}
\EEE
In particular, every $u_\eps \in  \MMM_\epsi(x)  $ is a minimizer for the
minimum problem on the right-hand side of
\eqref{programmazione-dinamica}, it satisfies
\begin{equation}
\label{ddp-identity} \Ve(x)=  \int_0^T
\ell_\eps(u_\eps,|u_\eps'|) \dd \mu_\eps + \Ve(u_\eps(T))
\rme^{-T/\eps} \quad \text{for all $T>0$,}
\end{equation}
 and for every $T>0$ the curve $ w_{\eps,T}(t):=
\ue(t+T) $ fulfills
\begin{equation}
\label{desired-minimized} w_{\eps,T} \in   \MMM_\epsi(\ue(T)).
\end{equation}
\end{proposition}
\begin{proof}
Formula \eqref{programmazione-dinamica} can be proved arguing along the very same lines as in the proof
of \cite[Prop. 2.5,Chap. III]{bardi-CapuzzoDolcetta97}. 
%
\par
 We now prove \eqref{ddp-identity}. First of all we show that, 
if $u\in \ACini x\eps$ and $T>0$, then
\begin{equation}
  \label{eq:37}
  \Ve(x)\le
   \int_0^T \ell_\eps(u,|u'|) \dd \mu_\eps
  + \Ve(u(T)) \rme^{-T/\eps}.
\end{equation}
It is not restrictive to assume that $\Ve(u(T))<\infty$:
we can then choose $w\in \MMM_\eps(u(T))$ so that
\begin{equation}
  \label{eq:38}
  \Ve(u(T))=\int_0^\infty \ell_\eps(w,|w'|)\,\dd\mu_\eps=
  \rme^{T/\eps}\int_T^\infty \ell_\eps(w(t-T),|w'|(t-T))\,\dd\mu_\eps(t), 
\end{equation}
and we consider the new curve $v\in \ACini x\eps$ defined by
\begin{displaymath}
  v(t):=
  \begin{cases}
    u(t)&\text{if }t\in [0,T],\\
    w(t-T)&\text{if }t\ge T.
  \end{cases}
\end{displaymath}
By the very definition of the value function we have
\begin{displaymath}
  \Ve(x)\le \calI_\eps[v]=
  \int_0^T \ell_\eps(u,|u'|)\,\dd\mu_\eps+
  \int_T^\infty \ell_\eps(w(t-T),|w'|(t-T))\,\dd\mu_\eps(t)
\end{displaymath}
which yields \eqref{eq:37} thanks to \eqref{eq:38}.

On the other hand, choosing $u_\eps\in \MMM_\eps(x)$ and defining
$w_{\eps,T}(t):=u_\eps(t+T)$, since $w_{\eps,T}\in \ACini {u_\eps(T)}\eps$ we get
\begin{equation}
  \label{eq:39}
  \begin{aligned}
    \Ve(x)&=\int_0^T \ell_\eps(u_\eps,|u_\eps'|)\,\dd\mu_\eps+
    \rme^{-T/\eps}\int_0^\infty
    \ell_\eps(u(t+T),|u'|(t+T))\,\dd\mu_\eps(t)
    \\&=\int_0^T \ell_\eps(u_\eps,|u_\eps'|)\,\dd\mu_\eps+
    \rme^{-T/\eps}\,\calI[w_{\eps,T}]\ge 
    \int_0^T \ell_\eps(u_\eps,|u_\eps'|)\,\dd\mu_\eps+
    \rme^{-T/\eps}\,\Ve(u_\eps(T));
  \end{aligned}
\end{equation}
by \eqref{eq:37} 
the previous \eqref{eq:39} is in fact an equality,
which shows that $u_\eps$ satisfies \eqref{ddp-identity},
is a minimizer of \eqref{programmazione-dinamica} and
satisfies $\calI[w_{\eps,T}]=\Ve(u_\eps(T))$, which yields \eqref{desired-minimized}.
\end{proof}
  Relation
\eqref{ddp-identity} has a simple but important differential version,
which will be the
starting point for our asymptotic analysis when $\eps \down 0$.
In order to highlight its structure, 
we introduce the function
 \begin{equation}
 \label{proto-wed-slope}
   \glee(x):= 
  \begin{cases}
    \displaystyle 
    \sqrt{2\frac{\phi(x)-\Ve(x)}{\eps}}&\text{if }x\in D(V_\eps),\\
    \infty&\text{otherwise}
  \end{cases}
 \end{equation}
which, in the next sections,
 will be shown to   suitably approximate the (relaxed) slope of $\phi$. 
\begin{proposition}[Fundamental identity]
  \label{prop:intermediate} 
  Let us suppose that $\phi$ satisfies the standard \UUU LSCC Property
  \EEE  \ref{basic-ass} and
   $\frac1{16 \eps}\ge\sfB$. 
  If $x\in \rmD(\Ve)$ and $\ue \in \mathcal{M}_\eps(x)$,
  the map $t\mapsto \Ve(u_\eps(t))$ is absolutely continuous,
  and it 
   fulfills 
\begin{equation}
\label{intermediate-relation}
\begin{aligned}
 -\frac{\dd}{\dd t} \Ve(\ue(t)) & = \frac12 |\ue'|^2(t) +\frac1\eps
\phi(\ue(t))-\frac1\eps \Ve (\ue (t))
\\
 & = \frac12 |\ue'|^2(t) + \frac12 \glee^2(\ue(t))
 \quad \foraa\, t \in
(0,\infty),\end{aligned}
\end{equation}
  \begin{equation}
    \label{Vabs}
    \Ve(u_\eps(t))=\phi(u_\eps(t))-\frac \eps2|\ue'|^2(t)\,\,\,\,
    \quad \foraa\, t \in (0,\infty).
  \end{equation}
\end{proposition}
\begin{proof}
It follows 
from the \UUU Dynamic Programming Principle \EEE  \eqref{programmazione-dinamica}  that for any $\ue \in \mathcal{M}_\eps(\ini)$ and  for all $0 \leq s \leq t$ there holds
\begin{equation} 
\rme^{-s/\eps} \Ve(\ue(s))-
\rme^{-t/\eps} \Ve(\ue(t))  = \int_s^t
\Big(\frac 12|\ue'|(r)^2+\frac 1\eps\phi(\ue(r)) \Big) \rme^{-t/\eps}\,\dd r,\label{eq:40}
\end{equation}
which shows that 
the map $t\mapsto \rme^{-t/\eps} \Ve(\ue(t)) $ is absolutely continuous. The
Lebesgue Theorem then yields that 
\begin{displaymath}
  \rme^{-t/\eps} \Big(\frac 1\eps \Ve(\ue(t))-\frac{\dd}{\dd t} \Ve(\ue(t))\Big)=
  \rme^{-t/\eps} \Big(\frac 12|\ue'|(t)^2+\frac 1\eps\phi(\ue(t))\Big) \quad \text{for a.a.~$t\in (0,\infty)$},
\end{displaymath}
and therefore \eqref{intermediate-relation}.
\par
In order to get \eqref{Vabs} we
denote by $\mathcal{V}_\eps$ the 
absolutely continuous representative of the function
$t\mapsto\phi(u_\eps(t))-\frac{\eps}{2}\vert u_\eps'\vert^2(t)$ on
$(0,\infty)$. 
The inner variation equation
\eqref{euler-lagrange-eq-infinite} gives that the distributional derivative of $\mathcal{V}_\eps$ fulfills
$\frac{\dd}{\dd t} \mathcal{V}_\eps(t) = -\vert u_\eps'\vert^2(t)
\in\,L^1(0,\infty;\mu_\eps)$. Hence,
$\mathcal{V}_\eps\in W^{1,1}(0,\infty;\mu_\eps)$ 
so that the identity
\begin{displaymath}
  \ell(u_\eps(t),|u_\eps'|(t))=\mathcal V_\eps(t)-\eps\mathcal
  V_\eps'(t)\quad\text{a.e.~in }(0,\infty)
\end{displaymath}
yields by the integration by parts formula   \eqref{eq:36}  that 
\begin{displaymath}
\Ve(u_\eps(t)) =\int_0^\infty
\ell(u_\eps(t+\tau),|u_\eps'|(t+\tau))\dd\mu_\eps(\tau)=
\int_{0}^{\infty}\Big(\mathcal V_\eps(t+\tau)-\eps\mathcal
V_\eps'(t+\tau)\Big) \dd \mu_\eps(\tau) 
=\mathcal{V}_\eps(t) 
\end{displaymath}
for every $t\ge0$.
\end{proof}
\noindent As a consequence of \eqref{Vabs}, we deduce  some integral
estimates on $\phi(\ue)$ and $|\ue'|$ \emph{uniformly} with respect to
$\eps>0$.
\begin{corollary}
\label{lemma:est-phi-up}
Let us suppose that $\phi$ satisfies the standard \UUU LSCC Property
\EEE  \ref{basic-ass} and
 $\frac1{16\eps}>\sfB$.  
 Then,  every $u_\eps\in  \MMM_\epsi (\ini) $  fulfills the following energy identity
\begin{equation}
\label{enid}
\Ve(u_\eps(t)) + \int_s^t \vert u_\eps'\vert^2(r) \dd r =
\Ve(u_\eps(s)) \quad \text{for all } 0 \leq s \leq t <\infty.
\end{equation}
 Furthermore
the following estimates hold
\begin{subequations}
\label{est-phi-up}
\begin{align}
\label{est-phi-up-1}
&
\int_{0}^{T}\vert u_\eps'\vert^2(t) \dd t
\,\le\  2 \Big(V_\eps(\bar u)+\sfQ(\bar u)\Big)
\,\mathrm{e}^{2\sfB T}, 
\\
&
\label{est-phi-up-2}
 \int_{0}^{T}\phi(u_\eps(t)) \dd t
 \,\le\,T \phi(\bar u)+  \eps\Big(V_\eps(\bar u)+  \sfQ(\bar u)\Big)
\,\mathrm{e}^{2\sfB T}. 
\end{align}
\end{subequations}
 for all $T\ge 0$ and all $ \eps>0.$
\end{corollary}
\begin{proof}
Combining \eqref{Vabs} and  the metric inner variation equation \eqref{euler-lagrange-eq-infinite}
we obtain
\begin{equation}
\label{pointwise-intermediate-identity}
\frac{\dd}{\dd t}\Ve(u_\eps(t)) + \vert u_\eps'\vert^2(t) \,=\,0
\,\,\hbox{ for a.a. } t \in (0,\infty),
\end{equation}
yielding \eqref{enid}.

We now introduce the function
\begin{equation}
  \label{eq:41}
  W_\eps(x):=V_\eps(x)+2\sfQ(x),\quad
  \text{satisfying}\quad
  0\le \max\Big(V_\eps(x),\sfQ(x)\Big)\le W_\eps(x)\quad
  \text{for every }x\in X,
\end{equation}
 (with $\sfQ$ from \eqref{basic-ass-2bis}), 
and we set $w_\eps(t):=W_\eps(u_\eps(t))$. Then, $w_\eps$ is an absolutely
continuous function satisfying the differential inequality
\begin{equation}
  \label{eq:42}
  \frac\dd{\dd t}w_\eps\le -|u_\eps'|^2 +4\sfB
  \sfd(u_\eps,u_*) |u_\eps'|  \le 
  4\sfB^2 \sfd^2(u_\eps,u_*)\le 4\sfB \sfQ(u_\eps)\le 
    2\sfB w_\eps, 
\end{equation}
 where  we have used \eqref{pointwise-intermediate-identity} and,   for the second inequality, 
  the elementary estimate $xy\leq x^2+\tfrac14 y^2$, 
so that 
\begin{equation}
  \label{eq:43}
  \sfQ(u_\eps(t))\le w_\eps(t)\le  w_\eps(0)\mathrm{e}^{2\sfB t}=
  \Big(V_\eps(\bar u)+2\sfQ(\bar u)\Big) \mathrm{e}^{2\sfB t}. 
\end{equation}
We then have that
\begin{equation}
  \label{eq:44}
  \int_0^T |u_\eps'|^2\,\dd t= V_\eps(\bar u)-
  V_\eps(u_\eps(T))\le V_\eps(\bar u)+\sfQ(u_\eps(T))\le 
  2\Big(V_\eps(\bar u)+\sfQ(\bar u)\Big) \mathrm{e}^{2\sfB t}, 
\end{equation}
 where the first estimate follows from \eqref{stime-V}  and the last one from \eqref{eq:43}. This
 yields 
\eqref{est-phi-up-1}.

It follows from \eqref{stime-V} and \eqref{Vabs} that
\begin{equation*}
\label{phi-metric-der}
\phi(u_\eps(t))\,\le\,\phi(\bar u)+ \frac{\eps}{2}\vert u_\eps'\vert^2(t) \quad \text{for all $t \in [0,\infty).$}
\end{equation*}
Hence, a further integration over $(0,T)$ gives  \eqref{est-phi-up-2}.
\end{proof}
\subsection{Gradient flow of the value function}
\label{ss:4.3}
%
%
%
%
 In this section we will
show that any minimizer $u_\eps$ of the WED functional 
$\mathcal I_\eps$ of Problem \ref{prob:main}
is a curve of maximal slope for the
value function $V_\eps$ 
with respect to its $L^1$-moderated upper gradient 
 \[
  \glee(x):= 
  \begin{cases}
    \displaystyle 
    \sqrt{2\frac{\phi(x)-\Ve(x)}{\eps}}&\text{if }x\in D(V_\eps),\\
    \infty&\text{otherwise.}
  \end{cases}
\]
In the forthcoming 
Theorem \ref{thm:upper-gradient} we will 
show that $\glee$ is an $L^1$-moderated upper gradient   (in the sense specified in \eqref{eq:45}) 
of $V_\eps$. We also refer to Appendix A ahead for further results in this connection. 
 We first state a useful Lemma.
\begin{lemma}
  \label{le:comparisonV}
  Let $u\in \AC^2([0,T];X)$ with $\phi\circ u\in L^1(0,T)$.
  Then for every $0\le r<s<T$ we have
  \begin{equation}
    \label{eq:48}
    \mathrm{e}^{-r/\eps} V_\eps(u(r))-
     \mathrm{e}^{-s/\eps}V_\eps(u(s))
     \le 
     \int_r^s 
    \Big(\frac 12 |u'|^2+\frac 1\eps \phi(u)\Big)\mathrm{e}^{-t/\eps}
    \dd t,
  \end{equation}
    \begin{equation}
    \label{eq:48bis}
    \mathrm{e}^{s/\eps} V_\eps(u(s))-
     \mathrm{e}^{r/\eps}V_\eps(u(r))
     \le 
     \int_r^s 
    \Big(\frac 12 |u'|^2+\frac 1\eps \phi(u)\Big)\mathrm{e}^{t/\eps}
    \dd t.
  \end{equation}
\end{lemma}
\begin{proof}
  Let us fix $r<s\in [0,T]$, let us choose $w\in \mathcal
  M_\eps(u(s))$ and let us consider the curve
  \begin{displaymath}
    v(t):=
    \begin{cases}
      u(r+t)&\text{if }0\le t\le s-r,\\
      w(t-(s-r))&\text{if }t\ge s-r,
    \end{cases}
  \end{displaymath}
  so that 
  \begin{displaymath}
    V_\eps(u(r))\le \mathcal I_\eps[v]=
    \int_r^{s}\ell(u(t),|u'|(t))
    \frac{\mathrm{e}^{-(t-r)/\eps}}\eps\dd t+
    \mathrm{e}^{-(s-r)/\eps} V_\eps(u(s))
  \end{displaymath}
  Multiplying the previous inequality by $\mathrm{e}^{-r/\eps}$ 
  we get \eqref{eq:48}. Applying the same argument
  inverting the order of time we infer \eqref{eq:48bis}. 
\end{proof}
\begin{theorem}
\label{thm:upper-gradient}
 Under the standard \UUU LSCC Property \EEE
\ref{basic-ass},
for every $\eps>0$ and for every $u \in \AC^2 ([0,T];X)$
such that $V_\eps\circ u\in L^1(0,T) $ and $ \glee\circ u\in 
L^2(0,T)$ we have that
\begin{align}
&
\label{absol-continuity}
\text{the map } t \mapsto \Ve(u(t)) \quad \text{is absolutely continuous on } [0,T];
\\
&
\label{ch-rul-form}
 \left| \frac{\dd }{\dd t } \Ve(u(t)) \right| \leq 
\frac 12 \glee^2(u(t))+\frac 12 |u'|^2(t) \qquad \foraa\, t \in (0,T).
\end{align}
In particular $\glee$ is 
an $L^1$-moderated upper gradient  
of $V_\eps$.
\end{theorem}
\begin{proof}
Let $u\in  \AC^2([0,T];X)$ be fixed
according to the assumptions of the theorem.
Since $\phi=\frac\eps 2 G_\eps^2+V_\eps$   by the definition of $G_\eps$, 
we get $\phi\circ u\in L^1(0,T)$. 
Setting $z(t):=\mathrm{e}^{-t/\eps}V_\eps(u(t))$,
 $ \RIC H_\eps \EEE (t):=\int_0^t (  \tfrac12  |u'|^2+ \tfrac1{\eps} \RIC \phi_+ \EEE (u) ) 
\dd r$, 
\eqref{eq:48} yields
\begin{displaymath}
  z(r)-z(s)\le H_\eps(s)-H_\eps(r)\quad
  \text{if }0\le r\le s\le T.
\end{displaymath}
It follows that 
the map $t\mapsto z(t)-H_\eps(t)$  is nondecreasing; 
since it is integrable, $z$ and
$V_\eps\circ u$ are locally bounded in $(0,T)$:
let us set 
$S(I):=\frac 1\eps \sup_{I}   |V_\eps \circ u| $
where $I\subset (0,T)$ is a compact interval.
Multiplying  inequality \eqref{eq:48} by $\rme^{r/\eps}$ we obtain 
\[
  V_\eps(u(r))-V_\eps(u(s))\le 
 (\rme^{-(s{-}r)/\eps} {-} 1) V_\eps(u(s))
 \RIC + H_\eps(s)-H_\eps(r)\,. \EEE
  \]
  We  then estimate the first term on the right-hand side by resorting to the  
elementary inequality $0\le 1-\mathrm{e}^{-x}\le x$ for $x\ge0$, which yields \EEE
\begin{equation}
  \label{eq:48tris}
  V_\eps(u(r))-V_\eps(u(s))\le 
  S(I) (s-r) 
  \RIC + H_\eps(s)-H_\eps(r) \EEE\quad
  r\le s, \ r,s\in I. 
\end{equation}
 Multiplying \eqref{eq:48bis} by  $\rme^{-s/\eps}$ and arguing in the very same way we obtain 
\eqref{eq:48tris} with 
 the order of $r$ and $s$ interchanged.   We thus get
\begin{equation}
  \label{eq:48quater}
 | V_\eps(u(r))-V_\eps(u(s))|\le 
 S(I) |s-r| 
  + |H(s)-H(r)|\quad
  \ r,s\in I,
\end{equation}
which shows that $V_\eps\circ u$ is locally absolutely continuous.

Let us fix now a time $r\in (0,T)$ 
which is a differentiability point for $V_\eps\circ u$
and a Lebesgue point for the 
integrand in \eqref{eq:48bis}.
Dividing this inequality by $s-r$ and passing to the limit
as $s\downarrow0$ we obtain
\begin{equation}
  \label{eq:51}
  \mathrm{e}^{r/\eps} (V_\eps\circ u)'(r)+
  \frac 1\eps \mathrm{e}^{r/\eps} V_\eps(u(r))\le 
  \Big(\frac 12 |u'|^2(r)+\frac 1\eps \phi(u(r))\Big)
  \mathrm{e}^{r/\eps}
\end{equation}
which shows
\begin{equation}
  \label{eq:49}
  (V_\eps\circ u)'(r)\le \frac 12 |u'|^2(r)+
  \frac12 G_\eps(u_\eps(r)).
\end{equation}
A similar argument applied to \eqref{eq:48} 
yields the opposite inequality, thus leading to
\eqref{ch-rul-form} and the absolute continuity
of $V_\eps\circ u$ in $(0,T)$.

Since $V_\eps\circ u$ is also lower semicontinuous,
passing to the limit as $s\down0$
in \eqref{eq:48bis} written \UUU for \EEE $r=0$ yields the continuity
of $V_\eps\circ u$ at $r=0$. \UUU A \EEE similar argument 
applied to \eqref{eq:48} at $s=T$
yields the continuity of $V_\eps\circ u$ at $T$.

We conclude the proof 
 that $G_\eps$ is an $L^1$-moderated upper gradient 
by 
integrating \eqref{ch-rul-form} from $0$ to $T$ 
and applying Corollary \ref{cor:criterium}.
\end{proof}
\begin{corollary}
  \label{cor:WED-GF}
  Every $\ue \in \mathcal{M}_\eps(\ini)$  is a curve of maximal slope 
  for $\Ve$ with respect to the  ($L^1$-moderated) 
  upper gradient $G_\eps$.
\end{corollary}
\section{Passage to the limit as $\eps \to 0$  and proof of Theorem \ref{th:4.1}}
\label{sez:passlim} \noindent
 The proof of \UUU Theorem \ref{th:4.1} is \EEE carried out  in \UUU
 Section \ref{ss:6.2} and \EEE relies on
a series of intermediate results on the asymptotic properties of  the functionals $(\Ve)_\eps$ as $\eps\down 0$,
proved in Sec.\ \ref{ss:6.1}.
 As usual, we will always assume that 
the functional $\phi$ satisfies the 
standard \UUU LSCC Property \EEE  \ref{basic-ass}.
\subsection{Comparison and asymptotic properties of
the functionals $(\Ve)_\eps$ as $\eps\down 0$}
\label{ss:6.1}
\begin{lemma}
\label{lemma:convV-phi} 
Let us suppose that the standard \UUU LSCC Property
\ref{basic-ass} holds. \EEE Then,
\begin{enumerate}
\item For every $\ini\in X$ the map $\eps\mapsto 
  V_\eps(\ini)$ is non increasing, i.e.~
\begin{equation}
\label{monoV} V_{\epsi_1}(\ini)\,\le\,V_{\epsi_0}(\ini)\qquad
\text{for all } \ini \in X \ \  \text{and all}  \ \epsi_1 \geq \epsi _0;
\end{equation}
\item
For every $\bar u\in X$ there holds
\begin{equation}
\Ve(\bar u) \uparrow \phi(\bar u)\,\,\,\,\,\hbox{ as }\,\,\eps\down
0; \label{convV-phi}
\end{equation}
\item Every family $(\bar u_\eps)_{\eps>0} \subset X$
  satisfies the $\Gamma$-$\liminf$ inequality
\begin{equation}
\label{enhanced-lsc}  
\inie \weaksigma \ini,\quad
\limsup_{\eps \downarrow0} \sfd(\bar u_\eps,\bar u)<\infty
\
 \ \Rightarrow \ \
\phi(\ini) \leq \liminf_{\eps \down 0}\Ve(\inie).
\end{equation}
\end{enumerate}
\end{lemma}
\begin{proof}
The monotonicity property 
\eqref{monoV}
is a consequence of the equivalent representation 
of $V_\eps$ as
\begin{equation}
\label{change-variables} 
\Ve(\ini) = \min_{u \in \mathscr C_\eps (\ini)}
\int_0^\infty \left( \frac1{2\eps^2} |\ue'|^2 (t) +
\phi(u(t))\right) \rme^{-t}\dd t\,.
\end{equation}
Convergence  \eqref{convV-phi}   immediately follows from  \eqref{monoV} and 
\eqref{enhanced-lsc}.
In order to prove the latter property, 
it is not restrictive to assume that 
$V_\eps(\bar u_\eps)\le V<\infty$ for sufficiently small $\eps$:  then, Problem \ref{prob:main} is feasible, $ \MMM_\epsi(\ini_\eps) \neq \emptyset$
by Thm.\ \ref{exist-minimizers}, 
and 
we  can 
choose   $u_\eps \in \MMM_\epsi(\ini_\eps)$ 
as in \eqref{eq:29} we set $\psi_\eps(t):=\phi(u_\eps(t))+
\frac 1{16 \eps}\Big(\int_0^t | u_\eps'|\dd r\Big)^2+\sfQ$,
where  $\sfQ> \limsup_{\eps\downarrow0} \sfQ(\bar u_\eps)$,
so that 
%
\begin{equation}
\label{convV-phi1}   \Ve(\bar
u_\eps)  \,\ge\,\int_{0}^{ \infty }\frac{\mathrm{e}^{-t/\eps}}{\eps}\psi_\eps(t) \dd t
-\sfQ=
\int_{0}^{\infty}\mathrm{e}^{-s}\psi_\eps(\eps
s)\dd s-\sfQ.
\end{equation}
Observe now that 
the uniform estimate \eqref{est-phi-up-1} 
yields 
\begin{equation}
\label{e:quoted-later}
\begin{gathered}
\sfd(u_\eps(\eps s),\bar u_\eps)\,\le\,
\Big(\eps s \int_0^{\eps s}| u'_\eps|^2 \dd t\Big)^{1/2} \le 
\sqrt {2\eps s}  \Big(V(\ini_\eps)+\sfQ(\ini_\eps)\Big)^{1/2}\rme^{\sfB \eps s},  
\\
\text{so that } \quad
\lim_{\eps\down0} \sfd(u_\eps(\eps s),\ini_\eps)=0,
\end{gathered}
%
\end{equation}
so that for every $s>0$
\begin{displaymath}
  \lim_{\eps\downarrow 0}u_\eps(\eps s)=\ini
  \quad \text{in the $\sigma$-topology},\quad
\liminf_{\eps\downarrow0}
\psi_\eps(\eps s)\ge \phi(\bar u)+\sfQ.
\end{displaymath}
\UUU Eventually, an \EEE application of 
Fatou's lemma to \eqref{convV-phi1} yields 
 $ \liminf_{\eps\down 0}\Ve(\bar u_\eps)\,\ge\,\phi(\bar
u)$. 
%
%
\end{proof}
The next result provides a lower bound of $V_\eps$ in terms of the
Yosida regularization of $\phi$, defined as
\begin{equation}
  \label{eq:52}
  \phi_t(x):=\inf_{y\in X} \left( \frac1{2t} \sfd^2(y,x)+\phi(y) \right) \quad x\in X,\ t>0.
\end{equation}
Notice that 
\begin{equation}
  \label{eq:65}
  \phi_t(x)\ge -\sfQ(x)\quad\text{if}\quad \frac 1{2t}\ge \sfB,
\end{equation}
and $\phi_t$ is uniformly bounded from below \UUU if \EEE  $\phi$ is bounded from below.
 Let us mention in advance that the upcoming \eqref{eq:53bis} will be used for establishing a key inequality between  the local slope  $|\partial\phi|$
and $\limsup_{\eps\down 0} G_\eps$.  
\begin{theorem}
  \label{thm:V-vs-Y}
  For every $x\in X$ and $T>0$ such that $\frac 1{4T}\ge \sfB$ we have
  \begin{equation}
    \label{eq:53bis}
    V_\eps(x)\ge \int_0^T \phi_t(x)\dd\mu_\eps(t)
    -2\sfQ(x)\rme^{-T/\eps}\quad
    \text{for every }T>0;
  \end{equation}
  in particular, there holds 
    \begin{equation}
    \label{eq:53}
    V_\eps(x)\ge \int_0^\infty \phi_t(x)\dd\mu_\eps(t).
  \end{equation}
\end{theorem}
\begin{proof}
  For every $u\in \mathscr C_\eps(x)$ we introduce the energy
  functional
  \begin{equation}
    \label{eq:54}
    \rmE(t):=\int_0^t |u'|^2(s)\dd s.
  \end{equation}
 Formula \eqref{eq:55} yields
  \begin{equation}
    \label{eq:60}
    \int_0^T \frac\eps 2|u'|^2(t)\dd\mu_\eps(t)=
    \int_0^T \frac 12 \rmE(t)\dd\mu_\eps(t)+\frac {\rme^{-T/\eps}}2\rmE(T),
  \end{equation}
  so that
  \begin{equation}
    \label{eq:61}
    \calI[u]\ge\int_0^T \Big(\frac 12
    \rmE(t)+\phi(u(t))\Big)\dd\mu_\eps(t)+
    \Big(\frac 12\rmE(T)+V_\eps(u(T))\Big)\rme^{-T/\eps}.
  \end{equation}
  On the other hand 
  \begin{equation}
    \label{eq:62}
    \rmE(t)\ge \frac 1t\Big(\int_0^t |u'|(s)\dd s\Big)^2\ge
    \frac{\sfd^2(u(t),x)}t,\quad
    V_\eps(u(T))\ge -2\sfQ(x)-2\sfB\sfd^2(u(T),x)
  \end{equation}
  so that,  taking into account that $\mathsf{B} \leq \frac1{4T}$, we find 
    \begin{equation}
    \label{eq:61-bis}
    \calI[u]\ge\int_0^T \Big(\frac 1{2t}
    \sfd^2(u(t),x)+\phi(u(t))\Big)\dd\mu_\eps(t)
    -2\sfQ(x)\rme^{-T/\eps}.
  \end{equation}
  \eqref{eq:53} immediately follows  from  \eqref{eq:53bis}. 
\end{proof}
\subsection{The WED slope and its relaxation}

Let us now introduce the functional
\begin{equation}
\label{notation-wed-slope}
\wslo(x):= \limsup_{\eps\down 0}  \glee(x) =
\limsup_{\eps\down 0}  \sqrt{2\frac{\phi(x)-\Ve(x)}{\eps} }\quad \text{for all } x\in \SFD(\phi),
\end{equation}
which shall be  referred to as the \emph{WED slope} of $\phi$; as
usual we set $\wslo(x)=\infty$ if $x\not\in\SFD(\phi)$.
We also introduce its lower semicontinuous relaxation
with respect to the $\sigma$-topology,
along $\sfd$-bounded sequences with bounded energy, viz.\
\begin{gather}
 \rwslo(x) :=\inf\left\{\liminf_{n\up\infty} \wslo(x_n):
   x_n \weaksigma x,  \  \sup_n (\sfd(x_n,x),\,\phi(x_n)) <\infty \right\}.
 \label{notation-rwslo}
\end{gather}
We shall refer to $ \rwslo$ as the \emph{relaxed WED slope} of $\phi$.

In Proposition   \ref{l:5.1} below we prove that $ \wslo$
is dominated by the local slope
of $\phi$. 
\begin{proposition}
 \label{l:5.1}
 If 
 $\phi$ satisfies \eqref{basic-ass-2bis}.
 then 
\begin{equation}
\label{wslope-below-bound}
\wslo (x)
\leq  \ls(x) \qquad \text{for every } x \in \SFD(\phi).
\end{equation}
 \end{proposition}
 \begin{proof}
   We recall the duality formula for the local slope \cite[Lemma 3.1.5]{Ambrosio-Gigli-Savare08}
   \begin{equation}
     \label{eq:63}
     \frac
     12\ls^2(x)=\limsup_{t\downarrow0}\frac{\phi(x)-\phi_t(x)}t\quad
     \text{for every }x\in \SFD(\phi).
   \end{equation}
   It is not restrictive to suppose $\ls(x)<\infty$ so that   by
   \eqref{eq:65}   there exists a constant $C\ge \ls(x)$ such that
   \begin{equation}
     \label{eq:67}
     0\le \frac{\phi(x)-\phi_t(x)}t\le C\quad \text{if }0<t<\frac 1{2\sfB}.
   \end{equation}
   Choosing $T$ so that $0<T\le \frac 1{4\sfB}$, 
   by \eqref{eq:53bis} we get
   \begin{align*}
     \frac{\phi(x)-V_\eps(x)}\eps&\le 
     \int_0^T \frac{\phi(x)-\phi_t(x)}\eps\dd\mu_\eps(t)
     +\frac{\rme^{-T/\eps}}\eps\Big(\phi(x)+2\sfQ(x)\Big)\\
       &=\int_0^{T/\eps} \frac{\phi(x)-\phi_{\eps t}(x)}{\eps
       t}\,t\rme^{-t}\dd
     t+\frac{\rme^{-T/\eps}}\eps\Big(\phi(x)+2\sfQ(x)\Big)
     \\&
     \le \int_0^\infty \left(  C{\land} \frac{\phi(x)-\phi_{\eps t}(x)}{\eps
       t} \right)  \,t\rme^{-t} \dd
     t+\frac{\rme^{-T/\eps}}\eps\Big(\phi(x)+2\sfQ(x)\Big),
   \end{align*}
    where the last inequality follows from \eqref{eq:67}. 
   Since the last integrand is uniformly bounded, Fatou's Lemma
   yields
   \begin{align*}
     \frac12\wslo^2(x)
     &\le 
     \int_0^\infty \limsup_{\eps\down0} \left( C{\land} \Big(\frac{\phi(x)-\phi_{\eps t}(x)}{\eps
       t}\Big)\right)  \,t\rme^{-t}\dd t
    \\ 
&\le 
     \int_0^\infty \frac 12\ls^2(x)\,t\rme^{-t}\dd t=
     \frac 12\ls^2(x).\qedhere
   \end{align*}
\end{proof}
 With our next result we provide the converse estimate of \eqref{wslope-below-bound}, cf.\ \eqref{enhanced-slope-bis},   in terms of the \emph{relaxed} slopes $ \rls$ and $ \rwslo $. Indeed, we shall derive it  from   estimate \eqref{enhanced-slope},
which will play a key role in the proof of Theorem \ref{th:4.1}. It
 involves  $ \rls$  and the lower semicontinuous relaxation of $G_\eps$ itself,  with respect to the $\sigma$-topology,
along $\sfd$-bounded sequences with  bounded energy, \emph{and along vanishing sequences $(\eps_n)_n$}, i.e.\
\begin{equation}
\label{Gelinf}
\Gelinf(x): = \inf\left\{\liminf_{n\up\infty} G_{\eps_n}(x_n)\, : \ \eps_n \down 0, \ 
   x_n \weaksigma x,  \  \sup_n (\sfd(x_n,x),\,\phi(x_n)) <\infty \right\}.
\end{equation}
 \begin{proposition}
\label{l:4.1} Assume 
Property \ref{basic-ass}. 
Then, 
for every $\ini  \in \SFD(\phi) $
there holds
\begin{align}
\label{enhanced-slope}
&
\Gelinf(\ini) \geq \rls (\ini), 
\\
&
\label{enhanced-slope-bis}
\rwslo (\ini) \geq \rls(\ini)\,.
\end{align}
\end{proposition} 
\begin{proof}
Let us fix  a vanishing sequence $(\eps_n)_n$ and a sequence $
    \ini_n  \weaksigma \ini  \text{  with }
\ \sup_{n} \left( \sfd(\ini_n,\ini), \, \phi(\ini_n ) \right) \leq C<\infty$. 
From the definition of $V_{\eps_n}$ we have
\begin{equation}
\label{first-steppo}
\begin{aligned}
\frac1{\eps_n} \left( \phi( \ini_n ) {-} \Ven( \ini_n ) \right) \geq
 \frac1{\eps_n} \int_0^{ \infty }
(\phi( \ini_n ){-}\phi( w_{\eps_n}(t)))\,
 \mathrm{d} \mu_{\eps_n}(t)   -\frac1{\eps_n}  \int_0^{ \infty } \frac{\rme^{-t/\eps_n}}{2}
| w'_{\eps_n}|^2(t)\, \mathrm{d}t
\end{aligned}
\end{equation} for every  $ w_{\eps_n}  \in   \ACini{\ini_n}{\eps_n}$. 
In order to show 
that 
\begin{equation}
\label{2PROVE}
\liminf_{n\to\infty}  \frac1{\eps_n} \left( \phi( \ini_n ) {-} V_{\eps_n}( \ini_n ) \right)  \geq \frac12\rls^2 (\ini),
\end{equation}
we pick $w_{\eps_n}$ such that, additionally, it fulfills for every $n \in \N $
\begin{equation}
\label{GMM_eps}
\int_0^t \left( \frac12 |w_{\eps_n}'|^2(s) + \frac12 \rls^{2} (w_{\eps_n}(s))\right) \dd s + \phi(w_{\eps_n}(t)) \leq   \phi(w_{\eps_n}(0))= \phi(\ini_{n}) \quad \text{for all } t >0
\end{equation}
and such that
\begin{equation}
\label{stimgflow} \sup_{n\in \N,\, t \in [0,\infty)}
\left(\phi( w_{\eps_n} (t)) + \int_0^t \mathcal{H}_{\eps_n}(s) \dd s \right) \leq
 \phi( \ini_{n} ), 
\end{equation}
where we have used
the place holder $\mathcal{H}_{\eps_n} (s) := \frac12 |w_{\eps_n}'|^2(s) + \frac12 \rls^{2} (w_{\eps_n}(s))$.
In fact it has been shown in  \cite[Thm.\ 2.3.1, Lemma 3.2.2]{Ambrosio-Gigli-Savare08} that,
under Property \ref{basic-ass},  for every $n\in N$ there exists
 $w_{\eps_n} \in   \ACini{\ini_n}{\eps_n}  $  complying with
\eqref{GMM_eps}--\eqref{stimgflow}.
\par
 In the following lines, we derive some finer estimates for the sequence $(w_{\eps_n})_n$. 
Indeed, for almost all
$t\in (0,\infty)$
\[
\frac{\dd }{\dd t} \frac12 \sfd^2 (w_{\eps_n}(t),w_{\eps_n}(0)) \leq  \sfd (w_\eps(t),w_{\eps_n}(0)) |w_{\eps_n}'|(t) \leq \frac{\delta}2
|w_{\eps_n}'|^2(t) +\frac1{2\delta} \sfd^2 (w_{\eps_n}(t), \ini_{n})
\]
for every $\delta>0$.  Hence, upon integrating along the interval
$(0,t)$ we find
\[
\begin{aligned}
\frac12 \sfd^2 (w_{\eps_n}(t),w_{\eps_n}(0)) &  \leq \delta \left( \phi(
w_{\eps_n}(0) ) - \phi(w_{\eps_n}(t)) \right) + \frac1\delta \int_0^t
\frac12 \sfd^2 (w_{\eps_n}(s),w_{\eps_n}(0)) \dd s \\
& \leq C +
 \mathsf{B}\delta \sfd^2(w_{\eps_n}(t), \ini_{n}) +\delta  \mathsf{Q}(\ini_{n}) 
+ \frac1\delta \int_0^t \frac12 \sfd^2
(w_{\eps_n}(s),\ini_{n}) \dd s
\end{aligned}
\]
where the first inequality follows from estimate  \eqref{stimgflow}, and the
last one from the coercivity condition   \eqref{basic-ass-2bis}  for
$\phi$. Choosing  $\delta = 1/(8 \mathsf{B})$ and taking into account the bounds on the sequence $(\ini_n)_n$,  we then conclude
\[
\sfd^2 (w_{\eps_n}(t),\ini_{n}) \leq  C \left(1+ \int_0^t \sfd^2
(w_{\eps_n}(s),\ini_n) \dd s \right) , 
\]
whence, by the Gronwall Lemma,
\[
\sup_{n\in \N,\, t \in [0,\infty)} \sfd^2 (w_{\eps_n}(t),\ini_n) \leq  C.
\]
Combining this estimate  with \eqref{basic-ass-2} we infer
\[
\exists\, C>0 \ \ \forall\, n\in \N \ \forall\, t \in [0,\infty)\, :
\quad  |\phi( w_{\eps_n} (t))| \leq C\,,
\]
whence by \eqref{stimgflow}
\begin{equation}
\label{quotable-later}
 \int_0^t \mathcal{H}_{\eps_n}(s) \dd s
\leq C \quad \text{for all } t \in [0,\infty), \ n\in \N\,.
\end{equation}

Therefore,   we have
\begin{equation}
\label{iparts}
\begin{aligned}
\ \frac1{\eps_n} \int_0^{ \infty }
(\phi( \ini_n ){-}\phi( w_{\eps_n}(t)))\,
 \mathrm{d} \mu_{\eps_n}(t) 
 & \geq   \frac1{\eps_n} \int_0^{ \infty }
\left(  \int_0^t  \mathcal{H}_{\eps_n} (s) \dd s \right) \,  \mathrm{d} \mu_{\eps_n}(t) 
\\ &
= \frac1{\eps_n} \int_0^{ \infty } \rme^{-t/\eps_n}   \mathcal{H}_{\eps_n} (t) \dd t
\end{aligned}
\end{equation}
where the second equality follows from the integration by parts
formula \eqref{eq:55bis}, taking into account \eqref{quotable-later}. 
 Plugging \eqref{iparts}
into \eqref{first-steppo}, the term $\frac1{\eps_n}  \int_0^{ \infty }
\frac{\rme^{-t/\eps_n}}{2} | w_{\eps_n}'|^2(t)\, \mathrm{d}t$ cancels out,
and we conclude \UUU that \EEE
%
\begin{equation}
\label{second-step}
\begin{aligned}
\frac1{\eps_n} \left( \phi( \ini_n ) {-} \Ven( \ini_n ) \right) \geq 
\frac12 \int_0^{ \infty }
 |\partial \phi|^2( w_{\eps_n} (t))\dd \mu_{\eps_n}( t)  
  = \frac12 \int_0^{ \infty } \rme^{-s}  |\partial \phi|^2( w_{\eps_n} (\eps_n s))\dd s,  
\end{aligned}
\end{equation}
where again we have used the change of variables in \eqref{change-variables}.
Since for all $s \in [0,\infty)$ \UUU we have \EEE
\begin{equation}
\label{d-weps}
\sfd( w_{\eps_n} (\eps_n s),  w_{\eps_n} (0)) \leq (\eps_n s)^{1/2} \sup_{t \in [0,\infty)}\|  w_{\eps_n} '\|_{L^2 (0,t)}
\leq C (\eps_n s)^{1/2},
\end{equation}
(the latter estimate due to \eqref{stimgflow})
and $  w_{\eps_n} (0)=  \ini_n  \weaksigma \ini$ as $n \to \infty$, we conclude that
  $ w_{\eps_n} (\eps_n s) \weaksigma \ini$ as  $n\to\infty$ for all $s \in [0,\infty)$.
  Also, observe that  for all $s \in [0,\infty)$
  $\sup_{n\in\N} \sfd( w_{\eps_n} (\eps_n s),\ini) \leq C$ due to the bounds on $(\ini_n)_n$  and \eqref{d-weps}, and that $\sup_{\eps}\phi(w_{\eps_n}(\eps_n s))
  \leq C$ by \eqref{stimgflow}.
  Therefore,
  \[
  \liminf_{n\to\infty}   |\partial \phi|^2( w_{\eps_n} (\eps_n s)) \geq \rls^2 (\ini) \quad \text{for all } s \in [0,\infty).
  \]
  Ultimately,
 from
 \eqref{second-step} and Fatou's Lemma we find
 \[
 \begin{aligned}
\liminf_{n\to\infty} \frac1{\eps_n} \left( \phi( \ini_n ) {-} \Ven( \ini_n ) \right)  & \geq 
\frac12 \int_0^{ \infty } \rme^{-s} \liminf_{n\to\infty}   |\partial \phi|^2( w_{\eps_n} (\eps_n s))\dd s
\\ &  \geq \frac12 \int_0^{ \infty } \rme^{-s} \rls^2(\ini) \dd s= \frac12  \rls^2(\ini),
 \end{aligned}
 \]
  whence \eqref{2PROVE}. 
 Since the sequences $(\ini_n)_n$  and $(\eps_n)_n $ are  arbitrary, we conclude \eqref{enhanced-slope}.
 \par
 Finally, let us  check that 
 \begin{equation}
 \label{speriamo-in-bene}
 \rwslo (\ini)  \geq \Gelinf(\ini) \qquad \text{for all } \ini \in \mathrm{D}(\phi),
\end{equation}
whence \eqref{enhanced-slope-bis}  immediately follows. 
With this aim, let us fix
$\eta>0$ and pick a sequence 
$(\ini_n)_n$ with
$  \ini_n \weaksigma \ini,  $ and   $\sup_n (\sfd(\ini_n,\ini),\,\phi(\ini_n)) <\infty$ such that 
\[
\liminf_{n\to\infty} \limsup_{\eps \down 0} G_\eps(\ini_n) \leq  \rwslo (\ini) +\eta.
\]
Up to an extraction, we may replace $\liminf_{n\to\infty} $ by $\lim_{n\to\infty}$. Hence,
\[
\exists\, \bar{n}\in \N \ \forall\, n \geq \bar n \, : \quad  \limsup_{\eps \down 0} G_\eps(\ini_n)  = \inf_{r>0}
\sup_{\eps \in (0,r)}G_\eps(\ini_n) \leq   \rwslo (\ini) +2\eta. 
\] 
Therefore, there exists a vanishing sequence $(r_n)_n$ such that, for $n$ sufficiently big, 
$G_{r_n}(\ini_n) \leq   \rwslo (\ini) +2\eta.$ This ensures that 
\[
 \Gelinf(\ini)  \leq \liminf_{n\to\infty} G_{r_n}(\ini_n) \leq  \rwslo (\ini) +3\eta,
\]
which concludes the proof of \eqref{speriamo-in-bene}, since $\eta>0$ is arbitrary. 
\end{proof}
Combining Propositions \ref{l:5.1} and \ref{l:4.1} we conclude  the following result,  specifying in which sense the quantities $(G_\eps)_\eps$ approximate the relaxed slope $\rls$.  
\begin{corollary}
\label{cor:convergence-to-relaxed}
Under \UUU the LSCC \EEE Property \ref{basic-ass} there holds
\begin{equation}
\label{sono-uguali}
 \rwslo (\ini)  =
\rls (\ini) \qquad \text{for every } \ini  \in \SFD(\phi).
\end{equation}
In particular, if  the local slope $\ls$ is $\sigma$-lower semicontinuous along $\sfd$-bounded sequences with bounded energy, then
$ \rwslo (\ini) =  \ls(\ini)$ for all $\ini \in \SFD(\phi)$.
\end{corollary}

\subsection{Proof of Theorem \ref{th:4.1}}
\label{ss:6.2}
It follows from Corollary  \ref{lemma:est-phi-up}
and the fact that $\sup_\eps \phi(\ini_\eps)\leq C$
 that
\[
\exists\, C \geq 0 \ \ \forall\, \eps>0 \ \forall\, t \in [0,\infty)\, : \qquad
\begin{cases}
\int_0^t |\ue'|^2(s) \dd s \leq C,
\\
\int_0^t \phi(\ue(s)) \dd s \leq C\,.
\end{cases}
 \]
 Moreover,
observing that $\sfd(\ue(t),\inie) \leq \int_0^t |\ue'|(s) \dd s $ and
taking into account \eqref{converg-init-data}, 
 we conclude that for every $T>0$
\begin{equation}
\label{d-bounded}
\exists\, C= C(T) >0 \ \ \forall\, \eps>0 \  \ \forall\, t \in [0,T]\, : \quad \sfd(\ue(t),\ini) \leq C(T)\,.
\end{equation}
\par
 We now apply Theorem \ref{thm:main-compactness} and 
 conclude that, for every vanishing $(\eps_k)_k$ there exist a (not relabeled)  subsequence $(\uek)_k$ and $u \in \AC_\loc([0,\infty);X)$ such that the pointwise convergence \eqref{pointiwse-conv} holds, as well as \eqref{eq:3} and \eqref{eq:14}. 
\par
We are now in a position to pass to the limit as $\eps_k \down 0$ in
identity \eqref{intermediate-relation}, which we integrate on any
interval $(0,t) \subset (0,\infty)$:
\begin{equation}
\label{enid-ve} \frac12  \int_0^t |\uek'|^2(s)\, \mathrm{d}s +
\int_0^t \frac1{\eps_k} \left( \phi(\uek(s)) {-} V_{\eps_k}(\uek(s))\right)\,
\mathrm{d}s + V_{\eps_k}(\uek(t)) = V_{\eps_k} (\ini_{\eps_k}).
\end{equation}
 Assumption \eqref{converg-init-data}, estimate \eqref{stime-V}, and the $\liminf$-inequality \eqref{enhanced-lsc}
 yield that
\begin{equation}
\label{convergence-epsk-1}
 \phi(\ini)  \leq \liminf_{k\to\infty} V_{\eps_k} (\ini_{\eps_k}) \leq \limsup_{k\to\infty} V_{\eps_k} (\ini_{\eps_k}) \leq  \limsup_{k\to\infty} \phi(\ini_{\eps_k}) =  \phi(\ini).  
\end{equation} 
As for the left-hand side of \eqref{enid-ve}, we observe that
\begin{equation}
\label{convergence-epsk-2}
\liminf_{\eps_k \downarrow 0}   V_{\eps_k}(\uek(t)) \geq \phi(u(t))) \quad \text{for all }  t\in [0,\infty)
\end{equation} thanks to \eqref{enhanced-lsc}  and \eqref{pointiwse-conv},
 and  
\begin{equation}
\label{convergence-epsk-2-bis}
\liminf_{\eps_k \downarrow 0}   \int_0^t |\uek'|^2(s)\, \mathrm{d}s
\geq \int_0^t |u'|^2(s) \dd s.
\end{equation}
by \eqref{eq:3}. 
In order to conclude \eqref{lsc-gflow}, it remains to show that
\begin{equation}
\label{convergence-epsk-3}
\liminf_{\eps_k \downarrow 0}
\int_0^t \frac1{\eps_k} \left( \phi(\uek(s)) {-} V_{\eps_k}(\uek(s))\right)\,
\mathrm{d}s
\geq \int_0^t   \frac12 \rls^2(u(s))  \dd s\,.
\end{equation}
\par
With this aim, we use that, for any $\delta>0$
\[
\begin{aligned}
\liminf_{\eps_k \downarrow 0}
\int_0^t \frac1{\eps_k} \left( \phi(\uek(s)) {-} V_{\eps_k}(\uek(s))\right)\,
\mathrm{d}s  & \geq \liminf_{\eps_k \downarrow 0}
\int_0^t \left( \frac{\phi(\uek(s)) {-} V_{\eps_k}(\uek(s))}{\eps_k} +\delta  \phi(\uek(s))  \right)\,
\mathrm{d}s
\\ & \quad
+ \liminf_{\eps_k \downarrow 0}  \left(-\delta \int_0^t  \phi(\uek(s))\dd s   \right) =: I_1+I_2\,.
\end{aligned}
\]
Now,
\[
I_2 =-\delta  \limsup_{\eps_k \downarrow 0} \int_0^t  \phi(\uek(s))\dd s\geq -\delta C
\]
for a constant independent of $\eps_k$,
where the latter inequality ensues from estimate \eqref{est-phi-up-2}
and condition \eqref{converg-init-data}.
As for $I_1$, we may apply  the Fatou Lemma since
 the function $$s\mapsto \frac{\phi(\uek(s)) {-} V_{\eps_k}(\uek(s))}{\eps_k} +\delta  \phi(\uek(s))$$
is bounded from below
by a constant independent of $\eps_k$: indeed,
 the first summand is positive, and the second one is bounded from below in view of
 the coercivity
  \eqref{basic-ass-2bis} and  estimate \eqref{d-bounded} above.
  Therefore,
\[
I_1 \geq \int_0^t \liminf_{\eps_k \down 0}\left( \frac{\phi(\uek(s)) {-} V_{\eps_k}(\uek(s))}{\eps_k} +\delta  \phi(\uek(s))  \right) \dd s
\]
Now, for any fixed $s \in (0,t)$ out of a negligible set,
let us extract a further  subsequence $(\eps_k')$, possibly depending on $s$, such that
\[
\liminf_{\eps_k \down 0}\left( \frac{\phi(\uek(s)) {-} V_{\eps_k}(\uek(s))}{\eps_k} +\delta  \phi(\uek(s))  \right)
= \lim_{\eps_k' \down 0}\left( \frac{\phi(u_{\eps_k'}(s)) {-} V_{\eps_k'}(u_{\eps_k'}(s))}{\eps_k'} +\delta  \phi(u_{\eps_k'}(s))  \right)\,.
\]
Observe that, along this subsequence there holds $\sup_k \phi(u_{\eps_k'}(s)) <\infty$, as well as estimate \eqref{d-bounded} and convergence \eqref{pointiwse-conv}. Therefore, we are in the position to apply
 the $\Gamma$-$\liminf$ inequality \eqref{enhanced-slope} from 
Lemma \ref{l:4.1}.  We ultimately conclude that
\[
\liminf_{\eps_k \down 0}\left( \frac{\phi(\uek(s)) {-} V_{\eps_k}(\uek(s))}{\eps_k} +\delta  \phi(\uek(s))  \right) \geq \frac12 \rls^2(u(s)) +\delta \phi(u(s)) \quad \foraa\, s \in (0,t).
\]
All in all, we deduce that
\[
\liminf_{\eps_k \downarrow 0}
\int_0^t \frac1{\eps_k} \left( \phi(\uek(s)) {-} V_{\eps_k}(\uek(s))\right)\,
\mathrm{d}s \geq \int_0^t \frac12 \rls^2(u(s)) \dd s +\delta \int_0^t \phi(u(s)) \dd s -C\delta\,.
\]
Since
$\delta $ is arbitrary, we infer \eqref{convergence-epsk-3}.
\par
  Combining 
 \eqref{convergence-epsk-1}--\eqref{convergence-epsk-3} we pass to the limit in \eqref{enid-ve} and thus conclude the proof 
 of the integral inequality \eqref{lsc-gflow}.  \hfill $\square$

\section{Finer results for $\lambda$-geodesically convex energies}
\label{s:aprio}
Throughout this
section, we shall further assume that 
\begin{equation}
\label{phi-conv} 
\text{$\phi$ is $\lambda$-geodesically
convex on $X$ for some $\lambda \in \R$,}
\end{equation}
cf.\ 
\eqref{def:l-geod-convex}.
Under this condition, first of all 
 we shall prove the continuity  of the value function
 with respect to the metric $\sfd$. 
The following result complements
 Lemma   \ref{lemma1}, where we showed the sequential $\sigma$-lower semicontinuity of $\Ve$ 
 on $\sfd$-bounded sets, as well as Theorem \ref{cont-wrt-dphi} in Appendix A ahead. 
\begin{lemma}
\label{l:3.1} Assume 
Property \ref{basic-ass}  and \eqref{phi-conv}.  Then,  $\Ve$ is continuous on sublevels of the energy
$\phi$, namely
\begin{equation}
\label{contV}
\left( \ini_n \to \ini \text{ and } \sup_n \phi(\ini_n) <\infty \right) \, \Rightarrow \, \Ve(\ini_n) \to \Ve(\ini).
\end{equation}
\end{lemma}
\begin{proof}
 Let 
%
 $\umin \in \ACini {\ini}\eps$  be a  minimizer  for $\calI_\eps$ (observe that it exists since $\ini \in \mathrm{D}(\phi)$).   We construct a sequence  of  curves
  $(u_n)_n $ with $u_n  \in 
 \ACini {\ini_n}\eps$ 
 for every $n\in \N$, 
 fulfilling 
 \begin{equation}
 \label{gamma-conv-argum} 
 \limsup_{n\to\infty}\Ve(\ini_n)\leq \limsup_{n\to\infty}\calI_\eps[u_n] \leq \calI_\eps[u_\eps] = \Ve(\ini)
 \end{equation}
and combine this with the previously proved lower semicontinuity of $\Ve$ with respect to
the topology
 $\sigma$, cf.\ Lemma \ref{lemma1}. 
To construct $(u_n)_n$, we argue in this way: for every $n \in \N$ we
set $\tau_n:= \sfd(\ini_n,\ini)$, and consider the constant-speed
geodesic $\gamma_n : [0,\tau_n] \to X$ connecting $\ini_n$ to
$\ini$, such that
\begin{equation}
\label{speed-1}
\frac{\sfd(\gamma_n(t),\gamma_n(s))}{t-s}=1 \quad \text{for all } s,t \in [0,\tau_n].
\end{equation}
Hence $|\gamma_n'|(t)=1$ for almost all $t \in (0,\tau_n)$. We define $u_n: [0,\infty) \to X$
setting
\[
u_n(t) := \begin{cases}
\gamma_n(t) & t \in [0,\tau_n],
\\
\umin(t)  & t \in [\tau_n,\infty).
\end{cases}
\]
Then,
\[
 \mathcal{I}_\eps[u_n] = \int_0^{\tau_n} \ell_\eps (t,\gamma_n(t),|\gamma_n'|(t)) \dd t +
 \int_{\tau_n}^{\infty} \ell_\eps (t,\umin(t),|\umin'|(t)) \dd t  =:
 I_1+I_2
\]
Since $ I_2$ converges to
$\mathcal{I}_\eps[\umin] $ as $n \to \infty$,
to conclude  \eqref{gamma-conv-argum}
it remains to show that $\lim_{n \to \infty} I_1=0$.
Now, by
\eqref{speed-1} and the $\lambda$-convexity
 \eqref{phi-conv} we have
\[
\begin{aligned}
I_1  & = \int_0^{\tau_n} \rme^{-t/\eps} \left(\frac12
+\frac1{\eps}\phi(\gamma_n(t))\right)\dd t
\\  &
 \leq
\eps (1- \rme^{-\tau_n/\eps}) + \max\{\phi(\ini_n), \phi(\ini)\} (1-
\rme^{-\tau_n/\eps}) -\frac\lambda 2 \sfd^2 (\ini,\ini_n)\int_0^{\tau_n}
\rme^{-t/\eps} \frac{(\tau_n-t)t}{\tau_n^2} \dd t,
\end{aligned}
\]
and we refer to the last integral as $I_3$. We have
 $\lim_{n\to\infty} I_{3}=0$, hence the
  right-hand side in the above inequality converges to $0$ as $n \to \infty$, which concludes the proof.
\end{proof}

In the following two sections
we are going to provide a series of finer properties, and estimates, for the family $(u_\eps)_\eps$ of WED-minimizers.
We shall prove them under the   $\lambda$-convexity condition \eqref{phi-conv},
distinguishing the cases
$\lambda =0$, handled  in the upcoming Section  \ref{ss:aprio-2}, and $\lambda<0$, see Sec.\ \ref{ss:6.2-bis}. 
The starting point for all calculations  will be the following relation
\begin{equation}
\label{starting-point}
\! \!\!\!\!\!\!
-\eps\frac{\dd^2}{\dd t^2}\left( \frac{1}{2}\upeq(t) \right)
 + \frac{\dd}{\dd t}\frac{1}{2}\vert u_\eps'\vert^2(t)\,=\,-
\frac{\dd^2}{\dd t^2}\phi(u_\eps(t))-\frac{\dd}{\dd t}\frac{1}{2}\vert u_\eps'\vert^2(t)\,\le\, - \lambda \vert u_\eps'\vert^2(t) \hbox{ in } \mathcal{D}'(0,\infty),
\end{equation}
holding for all $\lambda \leq 0$. 
\begin{remark}
\label{rmk:formal-derivation}
\upshape
In the Euclidean case $X=\R^n$, for $\phi$
smooth,
we can formally derive \eqref{starting-point} by testing  by $u_\eps'$ the Euler-Lagrange equation satisfied by WED-minimizers, i.e.\
$-\eps u_{\eps}{''}  +u_\eps' + \rmD \phi(u_\eps) =0 $,  and differentiating the relation thus obtained.
Therefore,
\[
\begin{aligned}
    -\frac{\eps}2 \frac{\dd}{\dd t^2} |u_\eps'(t)|^2 +\frac{\dd}{\dd t} |u_\eps'(t)|^2  
   &  = -\frac{\dd }{\dd t } (\langle \rmD \phi(u_\eps(t)), u_\eps'(t) \rangle)
   \\
 & = - \langle  \rmD \phi(u_\eps(t)), u_\eps{''}(t)\rangle - \langle \mathrm{D}^2 \phi(u_\eps(t))(u_\eps'(t)), u_\eps'(t) \rangle
 \\ & \leq  - \eps | u_{\eps}{''}(t)|^2 +\frac12\frac{\dd}{\dd t} |u_\eps'(t)|^2   - \lambda  |u_\eps'(t)|^2
 \\ & \leq \frac12\frac{\dd}{\dd t} |u_\eps'(t)|^2  -\lambda  |u_\eps'(t)|^2
 \end{aligned}
\]
where for the first inequality we have used that $ -\rmD \phi(u_\eps) =-\eps u_{\eps}{''}  +u_\eps' $ by the Euler-Lagrange equation,
and that $\phi$ $\lambda$-convex implies $\mathrm{D}^2\phi \geq \lambda$. Therefore we conclude \eqref{starting-point}.
\end{remark}
In both Section \ref{ss:aprio-2} and Sec.\  \ref{ss:6.2-bis}, we will devote some effort to the proof of 
inequality
\eqref{starting-point} in the present metric context, where the
above arguments are not available. Then, from \eqref{starting-point} we shall deduce the additional properties of the  WED-minimizers  $(u_\eps)_\eps$.

The basic result underlying \eqref{starting-point} is the following Lemma, which holds both for $\lambda =0$ and $\lambda <0$ and is thus anticipated here.
\begin{lemma}
\label{lemma:lambda-1-lambda}
Assume 
Property \ref{basic-ass} and the $\lambda$-convexity  \eqref{phi-conv} with $\lambda \in \R$.
Set
\begin{equation}
\label{calUe}
\mathcal{U}_\epsi(t)\,:=\,\int_0^{t}\frac{1}{2}\vert
u'_\epsi\vert^2(s) \dd s
\end{equation}
 and for every $0\,\le\, a\,<\,
b\,<\infty$ consider the family of linear functions
 \begin{equation}
 \label{linear-a-b}
 \mathsf{l}_{a,b}(t):=\frac{t-a}{b-a}.
 \end{equation}
  Then, for every $[a,b]\subset [0,\infty)$ we have
\begin{equation}
\begin{aligned}
\label{lambda1-lambda}
\int_{a}^b \big(\phi(u_\eps(t))
+\mathcal{U}_\eps(t)  \big) \dd\mu_\eps(t)  \le&\big( \phi(u_\eps(a))
+\mathcal{U}_\eps(a) \big)(i_{a,b}-\theta_{a,b}) + \big(
\phi(u_\eps(b)) +
\mathcal{U}_\eps(b) \big) \theta_{a,b}\\
&- \frac{\lambda}{2}\sfd^2(u_\epsi(a),u_\epsi(b)) \Gamma_{a,b},
\end{aligned}
\end{equation}
where
$\theta_{a,b}\,:=\,\int_{a}^b \mathsf{l}_{a,b}(t)\dd\mu_\eps(t)  $,
$i_{a,b}\,:=\mu_\eps([a,b])   $ and
$\Gamma_{a,b}\,:=\,\int_{a}^b
\mathsf{l}_{a,b}(t)(1-\mathsf{l}_{a,b}(t))\dd\mu_\eps(t).$
\end{lemma}
\begin{proof}
Let
us take
$a\,<\,b\,$ in $(0,\infty)$ such that $\phi(u_\eps(a))\,<\,\infty$
and $\phi(u_\eps(b))\,<\,\infty$ and consider the geodesic
$\gamma:[a,b]\longrightarrow X$ connecting $u_\eps(a)$ and
$u_\eps(b)$ with constant speed $\vert
\gamma'\vert(t)\,=\,\frac{\sfd(u_\eps(a),u_\eps(b))}{b-a}$. Let us
consider the curve $\tilde v$ defined by
\[
\tilde{v}(t) = \begin{cases} u_\eps(t) & t \in (0,a) \hbox{ or }
t\in (b,\infty),
\\
\gamma(t) & t \in [a,b].
\end{cases}
\]
By construction, $\tilde v \in \ACini{\ini}\eps$, hence 
$ \mathcal{I}_\eps[u_\eps]\,\le\,\mathcal{I}_\eps[\tilde{v}],$ 
which implies
\begin{equation}
\label{competitor-v-ab-lambda}
\int_{a}^{b}\left(
\frac{\eps}{2}\upeq(t)+\phi(u_\eps(t))\right) \dd \mu_\eps(t)
\,\le\,\int_{a}^{b} \left(
\frac{\eps}{2}\vert \gamma'\vert^2(t)+\phi(\gamma(t))\right) \dd \mu_\eps(t).
\end{equation}
Now, since $\phi$ is geodesically convex, there holds that
\begin{equation}
\label{phi-convex-ab-lambda}
\phi(\gamma(t))\,\le\,(1-\mathsf{l}_{a,b}(t))\phi(u_\eps(a)) +
\mathsf{l}_{a,b}(t)\phi(u_\eps(b)) - \frac{\lambda}2
(1-\mathsf{l}_{a,b}(t))  \mathsf{l}_{a,b}(t) \sfd^2(u_\eps(a),
u_\eps(b))  \text{ for all  } t\in[a,b].
\end{equation}
Moreover, we can estimate the speed of the geodesic $\gamma$  by
\begin{equation}
\label{citata-dopo-lambda} \vert
\gamma'\vert^2(t)\,=\,\frac{\sfd^2(u_\eps(a),u_\eps(b))}{(b-a)^2}\,\le\,
\frac{1}{b-a}\int_{a}^{b}\upeq(t) \dd t
\,\le\,2\frac{\mathcal{U}_\eps(b)-\mathcal{U}(a)}{b-a}.
\end{equation}
 Now, we
introduce the function
\[
\tilde{\mathcal{U}}_\eps^{a,b}(t) = \begin{cases}
\mathcal{U}_\eps(t) & t \in (0,a) \hbox{ or } t\in (b,\infty),
\\
(1-\mathsf{l}_{a,b}(t))\mathcal{U}_\eps(a) +
\mathsf{l}_{a,b}(t)\mathcal{U}_\eps(b) & t \in [a,b],
\end{cases}
\]
which coincides with $\mathcal{U}_\eps$ when $t=a,b$ and satisfies
$\frac{1}{2}\vert \UUU \gamma'\EEE\vert^2(t)\,\le \frac{\dd}{\dd
t}\tilde{\mathcal{U}}_\eps^{a,b}(t)$ for all $t\in\,(a,b)$. We have
\begin{align}
& \label{neglected-1-lambda}
 \eps\int_{a}^b \frac{1}{2}\upeq(t) \dd \mu_\eps(t)
\,=\,\int_{a}^b \mathcal{U}_\eps(t) \dd  \mu_\eps(t)  +
\left[\rme^{-t/\eps}
\mathcal{U}_\eps(t)\right]_{a}^{b},\\
& \label{neglected-2-lambda}
\eps\int_{a}^b \frac{1}{2}\vert\gamma'\vert^2(t)
\dd  \mu_\eps(t) \,\le\,\eps\int_{a}^b \frac{\dd}{\dd
t}{\tilde{\mathcal{U}}_{\eps}^{a,b}}(t)\dd
 \mu_\eps(t) \,=\,\int_{a}^b
{\tilde{\mathcal{U}}_{\eps}^{a,b}}(t)\dd  \mu_\eps(t)+ \left[\rme^{-t/\eps}
\tilde{\mathcal{U}_{\eps}}^{a,b}\right]_{a}^{b},
\end{align}
 by the integration by parts formula \eqref{eq:55}, 
where the  inequality in \eqref{neglected-2-lambda} is due to
\eqref{citata-dopo-lambda}. Thus, recalling that
$\tilde{\mathcal{U}}_\eps^{a,b}(t)=\mathcal{U}_\eps(t)$ for $t=a,b$,
 and combining  \eqref{neglected-1-lambda}--\eqref{neglected-2-lambda} with \eqref{competitor-v-ab-lambda}
and \eqref{phi-convex-ab-lambda}, we deduce \eqref{lambda1-lambda}.
\end{proof}
 We conclude this section  by fixing an identity that can be checked with direct calculations, and that will have a crucial role in the following proofs:
\begin{equation}
\label{cruc-29.09}
 t^2 = a^2  (1-\mathsf{l}_{a,b}(t) ) +  b^2  \mathsf{l}_{a,b}(t)  -(b-a)^2(1-\mathsf{l}_{a,b}(t) )  \mathsf{l}_{a,b}(t)   \text{ for all } t \in [a,b] \text{ and all } 0\leq a <b.
\end{equation}


\subsection{Finer 
 properties of WED minimizers  
in the $\lambda$-convex case,  $\lambda=0$}
\label{ss:aprio-2}
The main result of this section, Theorem \ref{teo:beppe1} below, shows
 that, for every fixed  $\eps> 0$,
along WED minimizers $u_\eps$ 
the energy $\phi$ is \emph{nonincreasing} and \emph{convex}. 
 Moreover, we also prove that the map $t\mapsto \phi(u_\eps(t))$ is continuous on $[0,\infty)$, i.e.\ it enjoys the same continuity  as $u_\eps$. 
Observe that these are the properties of 
 $\phi$ along  a curve of maximal slope, i.e.\ a solution of the gradient flow in the limit as $\eps \down 0$,
  cf.\ \cite[Thm.\ 3.2, page 57]{Brezis73} for the Hilbertian case, and \cite[Thm.\ 2.4.15]{Ambrosio-Gigli-Savare08}
   in the metric context. 
  Interestingly, and somewhat surprisingly, these properties hold also at the level $\eps>0$, provided  that the energy
  is geodesically convex.
\begin{theorem}
\label{teo:beppe1}
Assume 
Property \ref{basic-ass}, \eqref{phi-conv} with $\lambda =0$,  and let $u_\eps\in
\MMM_\epsi(\ini)$.  Then,
\begin{enumerate}
\item   $t\mapsto \frac12 |u_\eps'|^2(t)$ admits a locally Lipschitz   continuous pointwise representative on $(0,\infty)$; 
\item  $t \mapsto \vert u_\eps'\vert(t)$ \,\,\,and\,\,\,$t\,\mapsto\,\phi(u_\eps(t))$ \,\,are nonincreasing,
\item $t \mapsto \phi(u_\eps(t))$\,\,\,\,is convex,
\item  \eqref{starting-point} holds with $\lambda=0$, i.e.\ 
\begin{equation}
\label{starting-0}
-\eps\frac{\dd^2}{\dd t^2}\left(\frac{1}{2}\upeq(t)
\right) + \frac{\dd}{\dd t}\frac{1}{2}\vert u_\eps'\vert^2(t)\,=\,-
\frac{\dd^2}{\dd t^2}\phi(u_\eps(t))-\frac{\dd}{\dd t}\frac{1}{2}\vert u_\eps'\vert^2(t)\,\le\,0\,\,\,\,\hbox{ in }\,\,\mathcal{D}'(0,\infty).
\end{equation}
\end{enumerate}
Hence, the function $t \mapsto \phi(\ue(t))$ is continuous on $(0,\infty)$ and
  right-continuous at $t=0$. 
\end{theorem}

For the proof of Theorem \ref{teo:beppe1} we need  a series of
auxiliary results.
The first one will allow us to  deduce from
 estimate \eqref{lambda1-lambda} (with $\lambda=0$) in  Lemma \ref{lemma:lambda-1-lambda}   that the function
$t\mapsto \phi(u_\eps(t))
+\mathcal{U}_\eps(t)$, with $\mathcal{U}_\eps$  from
\eqref{calUe},
is convex.
\begin{lemma}
\label{lemma:beppeconvex}
Let $\zeta\in C^1([0,\infty))$ be strictly positive and let $\psi$ be
\UUU  
 lower semicontinuous \EEE in $(0,\infty)$. If
\begin{equation}
\label{cond-convex-int}
\int_{a}^b\psi(t)\zeta(t) \dd t \,\le\,\psi(a)(i_{a,b}-\theta_{a,b}) +
\psi(b)\theta_{a,b}
\quad \text{with }
\begin{cases}
\theta_{a,b}:=\int_{a}^b\zeta(t)
\mathsf{l}_{a,b}(t) \dd t,\\  i_{a,b}:=\int_{a}^b\zeta(t) \dd t,
\end{cases}
\end{equation}
then $\psi$ is convex.
\end{lemma}
\begin{proof}
We preliminarily prove that for any $\bar t\,\in\,(0,\infty)$
\begin{equation}
\label{liminf-equality} \psi(\bar t)= 
\liminf_{t\to \bar t}\psi(t).
\end{equation}
 Indeed, for  any fixed $\bar t$ the lower semicontinuity of $\psi$ gives
$ \psi(\bar t) \,\le\,\liminf_{t\down \bar
t}\psi(t)  =:  L $. 
Then, consider a sequence $t_n\down \bar t$ for which
$\psi(t_n)\rightarrow L$ as
 $n\to\infty$.
 Now, denoting by $\eta$ the measure $\eta\,:=\,\zeta(t)\mathcal{L}^1$, we have
\begin{align*}
&\disp\liminf_{n\to \infty}    \frac{1}{\eta([\bar t,t_n])}\int_{\bar t}^{t_n}\psi(t) \dd\eta(t)\,\ge\,L\,\,\,\,\hbox{ and }\nonumber\\
&\disp\lim_{n\to \infty}\frac{1}{\eta([\bar t,t_n])}\int_{\bar
t}^{t_n}\mathsf{l}_{\bar t,t_n}(t) \dd\eta(t)\,=\,\lim_{n\to
\infty}\frac{1}{\eta([\bar t,t_n])}\int_{\bar
t}^{t_n}(1-\mathsf{l}_{\bar t,t_n}(t)) \dd\eta(t)\,=\,\frac{1}{2}
\end{align*}
 (recall the notation $\mathsf{l}_{\bar t,t_n}(t)= \frac{t-\bar t}{t_n-\bar t}$). 
Thus, dividing both sides of \eqref{cond-convex-int} (written on the interval $(\bar t,t_n)$),
 by $\mu([\bar t,t_n])$ and letting
$n\to \infty$, we get
$$ L\le \frac{1}{2}\psi(\bar t) + \frac{1}{2}L $$
which, together with  $\psi(\bar t) \leq L$,  implies $L\,=\,\psi(\bar t)$. The same argument works with a sequence $t_n\uparrow \bar t$, and we conclude \eqref{liminf-equality}. 

Now,  in order   to conclude the proof we argue by contradiction. Thus, assume that $\psi$ is not convex. Then, there exist $\alpha\,<\,\bar t\,<\,\beta$ such that
\begin{equation}
\label{no-convex} \psi(\bar t)\,>\,(1-\mathsf{l}_{\alpha,\beta}(\bar
t))\psi(\alpha) + \mathsf{l}_{\alpha,\beta}(\bar t)\psi(\beta).
\end{equation}
Denote by $A_{\alpha,\beta}$ the open set defined as $A_{\alpha,\beta}\,:=\,\left\{t: \eqref{no-convex}\,\,
\hbox{ holds for } \alpha,\, \beta\right\}$. Let $(a,b)$ be the connected component
of $A_{\alpha,\beta}$
 containing $\bar t$. Thanks to \eqref{liminf-equality} and to the lower semicontinuity of $\psi$, we have
\begin{equation*}
\psi(a)\,=\,\liminf_{t\rightarrow
a}\psi(t)\,\ge\,(1-\mathsf{l}_{\alpha,\beta}(a))\psi(\alpha) +
\mathsf{l}_{\alpha,\beta}(a)\psi(\beta)\,\ge\,\psi(a),
\end{equation*}
which gives
$\psi(a)\,=\,(1-\mathsf{l}_{\alpha,\beta}(a))\psi(\alpha) +
\mathsf{l}_{\alpha,\beta}(a)\psi(\beta)$. The same argument also gives 
$\psi(b)\,=\,(1-\mathsf{l}_{\alpha,\beta}(b))\psi(\alpha) +
\mathsf{l}_{\alpha,\beta}(b)\psi(\beta)$. Therefore, we can conclude that 
\begin{equation*}
\psi(t)\,>\,(1-\mathsf{l}_{a,b}( t))\psi(a) +
\mathsf{l}_{a,b}( t)\psi(b)\,\,\,\hbox{ for all }
t\in(a,b).
\end{equation*}
Now, integrating the above inequality with respect to the measure $\mu$, we  contradict  \eqref{cond-convex-int}.
\end{proof}

We now derive a bound on the energy $\phi$ evaluated along $u_\eps$ in terms of the initial energy.
\begin{lemma}
\label{lemma:phi-decreases}
Let $u_\eps  \in \MMM_\epsi(\ini)$, then
\begin{equation}
\label{eq:phi-decreases}
\phi(u_\eps(t))\,\le\,\phi(\bar u) \quad \text{for all $t\,\ge\,0.$}
\end{equation}
\end{lemma}
\begin{proof}
By contradiction, assume that there exists a point $\bar t$ for which
$\phi(u_\eps(\bar t))\,>\,\phi(\bar u)$. Since $\phi$ is
lower semicontinuous, the set $A\,:=\,\left\{t\,\in\,(0,\infty): \phi(u_\eps(t))>\phi(\bar u)\right\}$ is open.
 Let $(a,b)$ denote the connected component of $A$ containing $\bar t$: then, $(a,b)$ is a bounded (open)
  interval of $(0,\infty)$. First of all $a\ge 0$. Moreover, $b$ is finite. In fact, assuming the opposite,
  we would have $\phi(u_\eps(t))\,>\,\phi(\bar u)\,\ge\,\phi(a)$
  for all $t \in (a,\infty)$,
  and thus, setting
 \[
\tilde u(t) = \begin{cases} u_\eps(t) & t \in (0,a) \hbox{ or } t\in (b,\infty),
\\
u_\eps(a) & t \in [a,\infty),
\end{cases}
\]
we would have $\mathcal{I}_\eps[\ue] > \mathcal{I}_\eps[\tilde{u}]$,
against the fact that $\ue$ is a minimizer for $\mathcal{I}_\eps$.
Thus, $b<\infty$ and $\phi(\ue(b)) \leq \phi(\ini)$. 
\par
Now,
let us
consider a geodesic $\tilde\gamma:[0,1]\longrightarrow X$ connecting
 $u_\eps(a)$ with $u_\eps(b)$ with unit speed. The convexity of $\phi$ implies that
\begin{equation*}
\phi(\tilde\gamma (s))\,\le\,\max\left\{\phi(u_\eps(a)),\phi(u_\eps(b))\right\}\,\le\,\phi(\bar u),
\end{equation*}
where we  have also used that $\phi(u_\eps(b))\,\le\,\phi(\bar u)$. We
reparametrize the geodesic $\tilde\gamma$ on $[a,b]$ to a curve
$\gamma$, with $\gamma(t)\,:=\,\tilde\gamma(s(t))$, fulfilling
 $\vert \gamma'\vert(t)\,=\,\vert u_\eps '\vert(t)$ on $[a,b]$.
  To obtain this, we consider the parametrization $t\mapsto s(t)$ such that
\begin{equation*}
\UUU s'\EEE(t)\,=\,\vert u_\eps '\vert(t)\,\,\,\,\hbox{ and }\,\,\,\,s(t)\,:=\,\min\left\{\int_{a}^{t}\vert u_\eps '\vert(r)dr,1\right\}.
\end{equation*}
As a consequence, the curve
 \[
 v(t): = \begin{cases} u_\eps(t) & t \in (0,a) \hbox{ or } t\in
(b,\infty),
\\
\gamma(t) & t \in (a,b),
\end{cases}
\]
satisfies \[\begin{aligned} 
\int_{a}^b\left(\frac{\eps}{2}\vert
v'\vert^2(t) + \phi(v(t))\right)\dd \mu_\eps(t)
 & =
\int_{a}^b \left(\frac{\eps}{2}\vert\gamma'\vert^2(t)
+ \phi(v(t))\right) \dd \mu_\eps(t)  \\ & <
\int_{a}^b\left(\frac{\eps}{2}\upeq(t) +
\phi(u_\eps(t))\right)\dd \mu_\eps(t),
\end{aligned}
\]
which contradicts the minimality of $u_\eps$ for $\calI_\eps$. 
 Hence, \eqref{eq:phi-decreases} holds.
\end{proof}
We now have all the ingredients for  checking Theorem \ref{teo:beppe1}.
\paragraph{\bf Proof of Theorem \ref{teo:beppe1}.}
We split the proof in some steps.
\\
$\vartriangleright$ \eqref{starting-0}:
First of all, Lemma \ref{lemma:beppeconvex} and 
inequality
\eqref{lambda1-lambda}
give that $t\mapsto\phi(u_\eps(t)) + \mathcal{U}_\eps(t)$ is convex and thus
\begin{equation}
\label{corte}
 \frac{\dd^2}{\dd t^2}(\phi(u_\eps(t)) +
\mathcal{U}_\eps(t))=\frac{\dd^2}{\dd t^2}\phi(u_\eps(t)) + \frac{\dd}{\dd t}\frac{1}{2}\upeq(t)\ge 0  \quad \text{ in $\calD'(0,\infty)$.} 
\end{equation}
Then, we rewrite the  metric inner variation  equation  \eqref{euler-lagrange-eq-infinite}  as
\begin{equation}
\label{in-the-form}
\frac{\dd}{\dd t}\phi(u_\eps(t)) + \frac{1}{2}\upeq(t) = \eps\frac{\dd}{\dd t}\frac12 \upeq(t) - \frac{1}{2}\upeq(t)
  \quad \text{ in $\calD'(0,\infty)$} 
\end{equation}
which, together with  \eqref{corte}, gives \eqref{starting-0}.  
\\
$\vartriangleright$ $t\mapsto \upeq(t)$ is nonincreasing: To this end, we set   $w_\eps(t) :=  \frac{1}{2}\upeq(t)$. The above discussion immediately gives that $w_\eps$ verifies
\begin{equation}
\label{compare}
-\eps {w}_\eps'' + {w}_\eps'\,\le\,0 \qquad \text{in } \mathcal{D}'(0,\infty),
\end{equation}
which we rewrite as 
\begin{equation}
\label{crucial-starting-pt}
 -\eps \rme^{t/\eps}\frac{\dd}{\dd t}(\rme^{-t/\eps}{w}_\eps')\,\le\,0 \qquad \text{ in $ \mathcal{D}'(0,\infty)$.}
 \end{equation}
\par
 In fact, it follows from \eqref{crucial-starting-pt} that  the distributional derivative ${w}_\eps'$ of $w_\eps$ is locally bounded, so that $w_\eps$ admits a locally Lipschitz pointwise representative, which will be still denoted by the same symbol. Moreover, the second  distributional derivative ${w}_\eps''$ is also locally bounded from above, so that $w_\eps$ is semiconcave, and thus admits left and right derivatives at every point. We will use  the right derivative  $({w}_\eps)_+'$
in the following argument to show that  $w_\eps$ is nonincreasing.
\par
  Indeed, suppose by contradiction that for some $\bar t$ we had $({w}_\eps)_+'(\bar t)> \bar{c}>0$.  Since
  $t\mapsto \rme^{-t/\eps}({w}_\eps)_+'$ is a nondecreasing function by  \eqref{crucial-starting-pt},
 for all $t\ge  \bar t $ we would have $\rme^{-t/\eps}({w}_\eps)_+' (t) \geq \rme^{-\bar{t}/\eps}({w}_\eps)_+' (\bar{t}) \geq \rme^{-\bar{t}/\eps} \bar{c}$,  which would imply
 $w_\eps(t)\,\ge\,w_\eps(\bar t) + \eps \bar{c}(\rme^{(t-\bar t)/\eps}-1)$.
 This clearly contradicts the integrability of $w_\eps(t)=\frac{1}{2}\upeq(t)$ on $(0,\infty)$
  (cf.\ \eqref{est-phi-up-1}). Thus, we have obtained that $t\mapsto \upeq(t)$ is nonincreasing.
\\
$\vartriangleright$ $t\mapsto  \phi(u_\eps(t))$ is convex:  It follows from   \eqref{starting-0} and from the monotonicity of  $t\mapsto \upeq(t)$
that $-\frac{\dd^2}{\dd t^2}\phi(u_\eps(t))\,\le \,0$ in $\calD'(0,\infty)$, which yields the thesis. 
\\
$\vartriangleright$ continuity of $t\mapsto \phi(u_\eps(t))$: it follows from the previously proved convexity that  
 $t \mapsto \phi(u_\epsi(t))$ is continuous on $(0,\infty)$.
  In order to check that $\phi \circ u_\eps$ is right-continuous at $t=0$,
 we observe that, by \eqref{eq:phi-decreases},
\begin{equation}
\label{crucial-limsup} 
\limsup_{t \down 0} \phi(\ue(t)) \leq \phi(\ue(0)) = \phi(\ini).
\end{equation}
Since by   \eqref{basic-ass-1} we also have $\liminf_{t \down 0}
\phi(\ue(t)) \geq \phi(\ue(0))$, we  conclude that
$\lim_{t \down 0} \phi(\ue(t)) = \phi(\ini)$.
\\
$\vartriangleright$ 
   $t\mapsto \phi(u_\eps(t))$ is nonincreasing: This follows from Lemma \ref{lemma:phi-decreases} and
   from the convexity of $t\mapsto \phi(u_\eps(t))$.
\fin

\noindent

\subsection{Finer  properties of WED minimizers   in the $\lambda$-convex case,  $\lambda<0$}
\label{ss:6.2-bis}
The main result of this section is the  analogue of Theorem  \ref{teo:beppe1} for $\lambda <0$. 
Observe that, along Hilbert and metric gradient flows 
(cf.\ the aforementioned \cite[Thm.\ 3.2, page 57]{Brezis73}, \cite[Thm.\ 2.4.15]{Ambrosio-Gigli-Savare08}),  the map  $\phi \circ u$ is nonincreasing and, 
if the energy $\phi$ is $\lambda$(-geodesically) convex,
$t\mapsto \rme^{-2\lambda^- t} \phi(u(t))$ is convex  ($\lambda^-$ denoting the negative part of $\lambda$), and
$t \mapsto \rme^{2\lambda t} |u'|^2(t)$ is nonincreasing.
Likewise,  in Theorem \ref{teo:prop-lambda} below we show that, at the level $\eps>0$, along any WED minimizer $u_\eps$ the functions
$\phi \circ u_\eps$ and $|u_\eps'|$ have these properties,  with  suitable correction terms. \EEE
 \begin{theorem}
\label{teo:prop-lambda}
Assume 
\UUU the LSCC \EEE Property \ref{basic-ass}, \eqref{phi-conv} with $\lambda <0$,  and let $u_\eps\in
\MMM_\epsi(\ini)$.  Then,
\begin{enumerate}
\item  the function $t \mapsto \phi(\ue(t))$ is locally Lipschitz  on $(0,\infty)$ and  right-continuous at $t=0$;
\item
$t\mapsto \frac12 |u_\eps'|^2(t)$ admits a locally Lipschitz representative on $(0,\infty)$,
\item
 \eqref{starting-point} holds;
 \item $t\mapsto\phi(u_\eps(t)) $ is nonincreasing;
 \item for every $[a,b]\subset (0,\infty) $  there exists $C_{a,b}>0$ such that 
 $t\mapsto \phi(u_\eps(t)) + \calU_\eps(t)$
 (with $\calU_\eps$ from \eqref{calUe})
  is  $\lambda C_{a,b} $-convex on $[a,b]$. 
\end{enumerate}
Moreover,
 for every $\lambda'<\lambda$ there exists $\eps'>0$ such that for all $0<\eps<\eps'$  
\begin{enumerate}
   \setcounter{enumi}{5}
\item the function    $t\mapsto \rme^{2\lambda' t}\vert u_\eps'\vert^2 (t) $   is nonincreasing.
\end{enumerate}
\end{theorem}

We split the proof in several results, 
 and start by
checking
the continuity  of $t\mapsto \phi(\ue(t))$ on $(0,\infty) $,
in Corollary \ref{cor:29.09} ahead, as a consequence of the result below, which establishes  a suitable convexity-type property of the mapping $\phi \circ u_\eps$.
\begin{lemma}
Assume 
\UUU the LSCC \EEE Property \ref{basic-ass}, \eqref{phi-conv} with $\lambda <0$,  and let $u_\eps\in
\MMM_\epsi(\ini)$. Let us  introduce the function
\[
L(t): =  \int_0^t |u_\eps'|(r) \dd r.
\]
Then, there holds
\begin{equation}
\label{e:29.09}
\phi(u_\eps(t)) \leq (1-\mathsf{L}(a,b;t))  \phi(u_\eps(a)) + \mathsf{L}(a,b;t)  \phi(u_\eps(b))  - \frac{\lambda}2 (1{-}\mathsf{L}(a,b;t))  \mathsf{L}(a,b;t)   (L(b){-} L(a))^2 
\end{equation}
 for all $t \in [a,b] $ and all $[a.b] \subset (0,\infty)$, 
where we have used the short-hand notation, cf.\ \eqref{linear-a-b}, 
\[
 \mathsf{L}(a,b;t)  : = \mathsf{l}_{L(a),L(b)}(L(t) )
 = \frac{L(t) - L(a)}{L(b) - L(a)}.
 \] 
 Therefore, the function $t\mapsto \phi(u_\eps(t)) - \frac{\lambda}{2} L^2(t)$ is convex on $(0,\infty)$.
\end{lemma}
\begin{proof}
Preliminarily, we  introduce the polynomial function
\[
P(s) = \phi(u_\eps(a)) + (s{-}L(a)) \frac{\phi(u_\eps(b)){-} \phi(u_\eps(a))  }{L(b){-} L(a)}
+\frac{\lambda}{2}  (s{-}L(a))( L(b){-}s),
\]
which satisfies
$P(L(a))  =  \phi(u_\eps(a))$ and $P(L(b))  =  \phi(u_\eps(b))$
and $P''(s) \equiv -\lambda$. 
 A direct calculation shows that 
\begin{equation}
\label{identity-29.09}
\begin{aligned}
P(L(t)) =  & (1-\mathsf{L}(\alpha,\beta;t))  P(L(\alpha)) + \mathsf{L}(\alpha,\beta;t)  P(L(\beta)) \\ & \quad  - \frac{\lambda}2 (1{-}\mathsf{L}(\alpha,\beta;t))  \mathsf{L}(\alpha,\beta;t)   (L(\beta){-} L(\alpha))^2 \quad \text{for all }  [\alpha,\beta] \subset (0,\infty).
\end{aligned}
\end{equation}
In particular, 
\[
P(L(t))  =  (1-\mathsf{L}(a,b;t))  \phi(u_\eps(a)) + \mathsf{L}(a,b;t)  \phi(u_\eps(b))  - \frac{\lambda}2 (1{-}\mathsf{L}(a,b;t))  \mathsf{L}(a,b;t)   (L(b){-} L(a))^2, 
\]
so that 
  \eqref{e:29.09} reads
\begin{equation}
\label{29.09-2show}
\zeta(t): = \phi(u_\eps(t))   -  P(L(t))  \leq 0 \quad \text{ for all } t \in [a,b] \quad \text{for all } [a,b]\subset (0,\infty)
\end{equation}
\par
Let us prove \eqref{29.09-2show} by contradiction. Suppose that there exist $[a,b]\subset (0,\infty)$ and $ \bar{t } \in (a,b) $ (we may suppose that $\bar t $ is in the interior of $[a,b]$  by a lower  semicontinuity argument),   such that 
$\zeta(\bar t) >0$.  Denote by $A_{a,b}$
the subset of  $ [a,b]$ where $\zeta $ is strictly positive, and by $[\alpha,\beta]$
 the  connected component of $A_{a,b}$ containing $\bar t$, so that 
 \begin{equation}
 \label{used-later29.09}
 \zeta(t)>0 \quad\text{for all } t \in (\alpha,\beta), \ \zeta(\alpha) = \zeta(\beta) =0.
 \end{equation}
 Arguing as in the proof of 
 Lemma \ref{lemma:lambda-1-lambda}, we now choose a suitable competitor for the curve   $u_\eps\in
\MMM_\epsi(\ini)$: we consider  the geodesic $\gamma : [L(\alpha), L(\beta) ] \to X$ connecting $u_\eps(\alpha)$ to $u_\eps(\beta)$ with unit speed, and 
 define the curve
\[
\tilde{v}(t): = \begin{cases}
u_\eps(t) & \text{ if } t \in (0,\alpha) \text{ or } t \in (\beta,\infty),
\\
\gamma(L(t)) & \text{ if } t \in [\alpha,\beta],
\end{cases}
\]
so that 
\[
|\tilde{v}'|(t) =| \gamma'|(L(t)) L'(t) = |u_\eps'|(t) \qquad \foraa\, t \in (\alpha,\beta).
\]
From $\calI_{\eps} [u_\eps] \leq \calI_\eps[\tilde v]$ we then conclude (cf.\ \eqref{competitor-v-ab-lambda}) that
\begin{equation}
\label{ad-cont1}
\int_\alpha^\beta 
\phi(u_\eps(t)) \dd \mu_\eps(t)  \leq 
\int_\alpha^\beta 
\phi(\gamma(L(t)) ) \dd \mu_\eps(t). 
\end{equation}
In turn,
\begin{equation}
\label{ad-cont2}
\begin{aligned}
& 
 \int_\alpha^\beta  \phi(\gamma(L(t)) )  \dd \mu_\eps(t)  
 \\
 &  \stackrel{(1)}{\leq}    \int_\alpha^\beta    (1-\mathsf{L}(\alpha,\beta;t))  \phi(u_\eps(\alpha)) + \mathsf{L}(\alpha,\beta;t)  \phi(u_\eps(\beta))  - \frac{\lambda}2 (1{-}\mathsf{L}(\alpha,\beta;t))  \mathsf{L}(\alpha,\beta;t)   (L(\beta){-} L(\alpha))^2  \dd \mu_\eps(t) 
\\
 &  
   \stackrel{(2)}{=}    \int_\alpha^\beta   P(L(t))  \dd \mu_\eps(t)  
\end{aligned}
\end{equation}
 where (1) follows from 
 the $\lambda$-geodesic convexity of $\phi$ and the fact that $\gamma(L(\alpha)) = u_\eps(\alpha)$ and 
  $\gamma(L(\beta)) = u_\eps(\beta)$, while (2) ensues from
  the fact that $ \phi(u_\eps(\alpha)) = P(L(\alpha))$ and $ \phi(u_\eps(\beta)) = P(L(\beta))$, cf.\ \eqref{used-later29.09}, 
  combined with 
   \eqref{identity-29.09}. 
   From \eqref{ad-cont1} and \eqref{ad-cont2} we thus arrive at a contradiction with \eqref{used-later29.09}. This concludes the proof of \eqref{e:29.09}. 
   \par
   The final assertion follows by combining \eqref{e:29.09} with the identity 
   \[
   L^2(t) =  (1-\mathsf{L}(a,b;t))  L^2(a) + \mathsf{L}(a,b;t) L^2(b)
    -  (1{-}\mathsf{L}(a,b;t))  \mathsf{L}(a,b;t)   (L(b){-} L(a))^2 \quad \text{for all } t \in [a,b],
   \]
   which follows from \eqref{cruc-29.09}.
\end{proof}

\begin{corollary}
\label{cor:29.09}
Assume 
\UUU the LSCC \EEE Property \ref{basic-ass}, \eqref{phi-conv} with $\lambda <0$,  and let $u_\eps\in
\MMM_\epsi(\ini)$.
Then,  the functions $t \mapsto \phi(\ue(t))$
 is continuous  on $(0,\infty)$, while $t\mapsto |u_\eps'|(t)$ is locally bounded. 
\end{corollary}
\begin{proof}
The assertion for $\phi \circ u_\eps$ follows from the continuity of the function $t\mapsto \phi(u_\eps(t)) - \frac{\lambda}{2} L^2(t)$. Since the mapping $t\mapsto \phi(u_\eps(t))  -\frac\eps2|u_\eps'|^2(t)$ is in $W^{1,1}_\loc([0,\infty))$, we conclude that $t\mapsto |u_\eps'|^2(t)$ has a continuous representative, whence the thesis.
\end{proof}
We are now in a position to prove inequality \eqref{starting-point}, along with 
some of the  other claims in the statement of 
\UUU Theorem \EEE \ref{teo:prop-lambda},
 in the case $\lambda<0$.
\begin{lemma}
\label{l:intermediate-tony}
Assume 
\UUU the LSCC \EEE Property \ref{basic-ass}, \eqref{phi-conv} with $\lambda <0$,  and let $u_\eps\in
\MMM_\epsi(\ini)$.  Then,
 $u_\eps$ enjoys
properties (2), (3) and (5) from the  statement of \UUU Theorem \EEE \ref{teo:prop-lambda}. In particular,
the function $t\mapsto \phi(u_\eps(t))$ is locally Lipschitz on $(0,\infty)$. 
\end{lemma}
\begin{proof}
\emph{\textbf{Claim $1$:}   For every  $[a,b]\subset [0,\infty)$, let  
$C_{a,b} = \sup_{t\in (a,b)}|u_\eps'|^2(t)$. 
 Then,}
\begin{equation}
\label{eq:lambda1}
 t\mapsto \phi(u_\epsi(t)) + \mathcal{U}_\epsi(t) \hbox{ is }  \lambda C_{a,b}-\hbox{convex on } [a,b]. 
\end{equation}
Indeed, from inequality \eqref{lambda1-lambda} we deduce that 
\[
\begin{aligned}
&
\int_{a}^b \big(\phi(u_\eps(t))
+\mathcal{U}_\eps(t)  \big) \dd\mu_\eps(t)  \\ & \le\big( \phi(u_\eps(a))
+\mathcal{U}_\eps(a) \big) \int_{a}^b (1-\mathsf{l}_{a,b}(t)) \dd \mu_\eps(t) + \big(
\phi(u_\eps(b)) +
\mathcal{U}_\eps(b) \big) \int_a^b \mathsf{l}_{a,b}(t) \dd \mu_\eps(t) \\ & \quad  -
 \frac{\lambda C_{a,b}}{2} (b-a)^2 \left( \int_{a}^b (1-\mathsf{l}_{a,b}(t)) \dd \mu_\eps(t)  \right)  \int_a^b \mathsf{l}_{a,b}(t) \dd \mu_\eps(t).
\end{aligned}
\]
We combine this with  identity  \eqref{cruc-29.09} to conclude that 
\[
\begin{aligned}
&
\int_{a}^b \big(\phi(u_\eps(t))
+\mathcal{U}_\eps(t)   - \frac{\lambda C_{a,b}}{2} t^2 \big) \dd\mu_\eps(t) \\ &  \leq 
\big( \phi(u_\eps(a))
+\mathcal{U}_\eps(a)  -\frac{\lambda C_{a,b}}{2} a^2 \big) \int_{a}^b (1-\mathsf{l}_{a,b}(t)) \dd \mu_\eps(t) + \big(
\phi(u_\eps(b)) +
\mathcal{U}_\eps(b) -\frac{\lambda C_{a,b}}{2} b^2 \big) \int_a^b \mathsf{l}_{a,b}(t) \dd \mu_\eps(t).
\end{aligned}
\]
Therefore, applying Lemma \ref{lemma:beppeconvex} we conclude that the function
$\psi(t): = \phi(u_\eps(t))
+\mathcal{U}_\eps(t)   - \frac{\lambda C_{a,b}}{2} t^2  $ is convex, whence  the desired \eqref{eq:lambda1}. 
\par
\noindent
\emph{\textbf{Claim $2$:} there holds}
\begin{equation}
  \begin{aligned} -\eps\frac{\dd^2}{\dd t^2}\left(\frac{1}{2}\upeq(t)
\right) + \frac{\dd}{\dd t}\frac{1}{2}\vert u_\eps'\vert^2(t) 
=-\frac{\dd^2}{\dd t^2} \phi(u_\eps(t))-\frac{\dd}{\dd
t}\frac{1}{2}\vert u_\eps'\vert^2(t)  \le -\lambda
 C_{a,b}\,\hbox{ in }\, \mathcal{D}'(a,b). \end{aligned}
\label{prop-la1}
\end{equation}
It follows from \eqref{eq:lambda1}  that 
\begin{equation}
\label{later29.09}
\frac{\dd^2}{\dd
t^2}\Big(\phi(u_\epsi(t)) +
\mathcal{U}_\epsi(t)\Big)\,\ge\,\lambda  C_{a,b}  \qquad \text{in  } \mathcal{D}'(a,b).
\end{equation}
 and thus,
rewriting  the  metric inner variation   equation
\eqref{euler-lagrange-eq-infinite}  in the form \eqref{in-the-form}
and rearranging  the terms,  we have \eqref{prop-la1}.
\par
\noindent
\emph{\textbf{Claim $3$:}
 the function $t\mapsto \frac12 |u_\eps'|^2(t)$ admits a locally Lipschitz  representative.}
\\
    We again use the notation  $ w_\eps
(t)\,=\,\frac{1}{2}\vert u'\vert^2(t)$. From \eqref{prop-la1} we
deduce that $ w_\eps $ fulfills $-\epsi   \UUU  w_\eps''\EEE  +\UUU  w_\eps'\EEE \le
-\lambda  C_{a,b}$  in $\calD'(a,b)$  which, setting $\nu:=-\lambda
 C_{a,b}$,  \EEE we rewrite as
\begin{equation}
\label{stg-pt-2}
 -\eps \rme^{t/\eps}\frac{\dd}{\dd t}\Big(\rme^{-t/\eps}( \UUU  w_\eps'\EEE  - \nu)\Big)\,\le \,0   \text{ in }  \calD'(a,b). 
\end{equation}
Now, let us set $\mathcal L(t)= w_\eps (t)- \nu t$.  It follows from \eqref{stg-pt-2} that the distributional derivative of $\mathcal{L}$ is locally bounded, so that $\mathcal{L}$ admits a locally Lipschitz representative, whence the claim for $w_\eps$.  From now on, we will identify $t\mapsto \frac12 |u_\eps'|^2(t)$ with its locally Lipschitz representative.
\par
\noindent 
\emph{\textbf{Claim $4$:}
the function $t\mapsto \phi(u_\eps(t))$ is locally Lipschitz   on $(0,\infty)$. }
Since the function  $t\mapsto \frac12 |u_\eps'|^2(t)$ is locally bounded, its primitive $\mathcal{U}_\eps$ is locally Lipschitz on $(0,\infty)$. In turn, by Claim $1$ the mapping  $t\mapsto \phi(u_\eps(t)) + \calU_\eps(t)$ is locally Lipschitz on $[a,b]$ for every $[a,b]\subset (0,\infty)$. 
The claim follows. 
\par
\noindent
\emph{\textbf{Claim $5$:   \eqref{starting-point} holds.}}
 Let $\mathsf{f}$  be the density of the distributional derivative 
$-\frac{\dd^2}{\dd
t^2}\Big(\phi(u_\epsi(t)) +
\mathcal{U}_\epsi(t)\Big)$. It follows from  
\eqref{later29.09} that 
\[
\mathsf{f}(t) \leq - \lambda \sup_{s\in [a,b]}{|u_\eps'|^2(s)} \qquad \foraa\, t \in (a,b), \ \text{for all }  [a,b]\subset(0,\infty). 
\]
Let $t\in (0,\infty)$, out of a negligible set, be a Lebesgue point for $\mathsf{f}$.
Then,
\[
\mathsf{f}(t) = \lim_{r\down 0}\frac1r \int_t^{t+r} \mathsf{f}(s) \dd s \leq   \lim_{r\down 0} \left( - \lambda \sup_{s\in [t,t+r]}{|u_\eps'|^2(s)} \right) = -\lambda |u_\eps'|^2(t),
\]
whence 
 \eqref{starting-point}. 
 \par
  This concludes the proof. 
\end{proof}
\par
We are now in a  position to conclude the
\begin{proof}[Proof of Theorem \ref{teo:prop-lambda}.]
In view of Lemma \ref{l:intermediate-tony}, it remains to prove properties (4) and (6), as well the 
right-continuity of $t\mapsto \phi(u_\eps(t))$ at $t=0$. \EEE
We split the proof in several claims.

\noindent
\emph{\textbf{Claim $1$:}} \emph{there exists a family  $ (x_2^\eps)_\eps \subset (0,\infty)$ such that
$x_2^\eps\down-2\lambda$ as $\eps \down 0$ and}
\begin{equation}
\label{claim1}
\rme^{-x_2^\eps t} |u_\eps'|^2 \text{ is  nonincreasing. }
\end{equation}
 Then, (1) in \UUU Theorem \EEE \ref{teo:prop-lambda} follows  upon choosing, for every prescribed $\lambda'<\lambda$, $\eps'>0$ such that for all $\eps \in (0,\eps')$ there holds 
with $-\frac12 x_2^\eps >\lambda'$.

\noindent
We combine \eqref{starting-point} with the metric inner variation equation, cf.\ \eqref{in-the-form}, and deduce that
(cf.\ \eqref{compare})
\begin{equation}
\label{conseq}
-\eps \UUU  w_\eps''\EEE+\UUU  w_\eps'\EEE+2\lambda w_\eps \leq 0 \qquad \text{in } \mathcal{D}'(0,\infty),
\end{equation}
where again we have used the place-holder $w_\eps := \frac12 |u_\eps'|^2 $.
Let us introduce  the \emph{negative}   distributions   $\nu_\eps$ and
$h_\eps $ 
by
\[
h_\eps:= \frac{\nu_\eps}\eps  \qquad \text{and }
\nu_\eps:= -\eps \UUU  w_\eps''\EEE +\UUU  w_\eps'\EEE+2\lambda w_\eps. 
\]
Hence,
\begin{equation}
\label{eq:eqw}
\UUU  w_\eps''\EEE -\frac{1}{\eps}\UUU  w_\eps'\EEE- \frac{2\lambda}{\eps}w_\eps = -h_\eps \qquad \text{in }\mathcal{D}'(0,\infty).
\end{equation}
The general solution of \eqref{eq:eqw} has the form
\begin{equation*}
w_\eps(t) = A \rme^{x_1^\eps t} + B\rme^{x_2^\eps t} + \int_{-\infty}^{\infty}E(t-s)(-h_\eps(s))\dd s,
\end{equation*}
where $x_1^\eps$ and $x_2^\eps$ are the two (real) roots of the characteristic
equation and $E$ is the fundamental solution with support in $(-\infty,0]$.
We have that
\begin{equation}
\label{eq:roots}
x_1^\eps = \frac{1 + \sqrt{1 + 8\lambda\eps}}{2\eps},\,\,\,\,\,x_2^\eps=  \frac{1 - \sqrt{1 + 8\lambda\eps}}{2\eps}.
\end{equation}
Note that (at least for sufficiently small $\eps$)
$x_1^\eps$ and $x_2^\eps$ are positive.
Consequently,
\begin{equation}
\label{eq:w}
w_\eps(t) = \int_{-\infty}^{\infty}E(t-s)(-h_\eps(s))\dd s,
\end{equation}
since $w_\eps$ must be integrable on $(0,\infty)$.
The function $E$ is the fundamental solution with support in $(-\infty,0)$ and
can be found by  solving the following  Cauchy   problem
\begin{equation}
\label{eq:fund_probl}
\begin{cases}
\UUU  v''\EEE(t) -\frac{1}{\eps}\UUU  v'\EEE(t) - \frac{2\lambda}{\eps}v(t)= 0,\\
v(0) = 0\\
v'(0) = -1.
\end{cases}
\end{equation}
Denoting with $H(\cdot)$ the Heaviside function, we have that
\begin{equation}
\label{eq:green_funct}
E(t) = -\frac{1}{\theta}{\rm e}^{{t}/{(2\eps)}}\sinh(\theta(t)) H(-t) \qquad \text{with the place-holder } \theta:=\frac{\sqrt{1 + 8\lambda\eps}}{2\eps}.
\end{equation}
 Therefore, from \eqref{eq:w} we gather that 
$
w_\eps(t) = \frac{1}{\theta}\int_{t}^{\infty}h_\eps(s)
\rme^{{(t-s)}/{(2\eps)}}\sinh(\theta(t-s)) \dd s,
$
which we rewrite as
\begin{equation}
\label{eq:w2}
w_\eps(t) = \frac{\rme^{x_2^\eps t}}{\theta}\int_{t}^{\infty}
h_\eps(s) \rme^{-s/2\eps}f(t,s) \dd s
\end{equation}
where $ft,s):=\frac{1}{2} (\rme^{2\theta t}\rme^{-\theta s} - \rme^{\theta s})$.
Now, differentiating with respect to $t$ we find
$$
\UUU  w_\eps'\EEE(t) = x_2^\eps w_\eps(t) + \frac{\rme^{x_2^\eps t}}{\theta}\left(
-h_\eps(t) \rme^{-t/2\eps}f(t,t) + \int_{t}^{\infty}
h_\eps(s) \rme^{-s/2\eps}\partial_t f(t,s)\dd s \right).
$$
Thus,
\begin{equation}
\label{piu-sotto}
\UUU  w_\eps'\EEE(t) \le x_2^\eps w_\eps(t)  \qquad \text{in } \mathcal{D}'(0,\infty)\,
\end{equation}
since 
$f(t,t) = 0$ and
$h_\eps\le 0$ by construction  while, in turn, $\partial_t f(t,s) \geq 0$. 
 As a consequence, we have
\begin{equation}
\label{eq:first_cons}
\frac{\mathrm{d}}{\mathrm{d} t} (\rme^{-x_2^\eps t }w_\eps(t)) \le 0 \qquad \text{in } \mathcal{D}'(0,\infty)\,
\end{equation}
whence \eqref{claim1}.

\noindent
\emph{\textbf{Claim $2$:}} \emph{the function $\phi \circ u_\eps$ is nonincreasing}.
\\
Indeed, from \eqref{in-the-form} we gather that
\begin{equation}
\label{in-the-form-2}
\frac{\dd}{\dd t} (\phi \circ u_\eps) = \eps w_\eps' -2 w_\eps \leq (\eps x_2^\eps -2) w_\eps \leq 0 \qquad \text{in } \mathcal{D}'(0,\infty)
\end{equation}
where the first inequality holds in view of \eqref{piu-sotto}.
The second one is true
  for a sufficiently small $\eps$, since $x_2^\eps$ converges to $-2\lambda$. 
\par
Finally,  The continuity of $t\mapsto \phi(u_\eps(t))$ 
 at $t=0$  follows from its previously
  proved monotonicity, arguing in the very same way as in the proof of Thm.\ \ref{teo:beppe1}, cf.\ \eqref{crucial-limsup}.
  \par
This concludes the proof of Thm.\ \ref{teo:prop-lambda}. 
\end{proof}
\EEE

\section{The metric  Hamilton-Jacobi equation and the gradient flow of $\Ve$}
\label{sez:hamilton-jabobi}
In this section we get further insight into the interpretation of WED minimizers
as curves of maximal slope with respect to $\Ve$.
Our starting point will again be the fundamental
identity
\eqref{intermediate-relation} satisfied by any WED minimizer $\ue$, viz.\
\[
 -\frac{\dd}{\dd t} \Ve(\ue(t))  = \frac12 |\ue'|^2(t) +\frac1\eps
\phi(\ue(t))-\frac1\eps \Ve (\ue (t))\qquad \foraa\, t \in (0,\infty),
\]
but
we shall adopt a different viewpoint in comparison to Theorem \ref{thm:upper-gradient}.
 Indeed, 
here we will combine \eqref{intermediate-relation}   with 
the \emph{metric analogue} of identity \eqref{HJ_Veps}, which in turn derived from the 
 \emph{Hamilton-Jacobi equation} \eqref{HJ-intr}. That is why, we may refer to \eqref{cruc-rel} below, 
 relating the functional $\frac{1}\eps (\phi -\Ve)$ with  (a suitable version of ) the local slope of $\Ve$, 
 cf.\ \eqref{llssvi} ahead,
   as a \emph{(metric) Hamilton-Jacobi identity}. 
 Let us mention in advance that the proof  of  \eqref{cruc-rel}    relies on the $\lambda$-geodesic convexity
 of $\phi$, for some $\lambda \in \R$.
 From this we will deduce in Corollary \ref{V-gflow} that $\ue$ is a curve of maximal slope for
 $\Ve$, albeit with respect to this suitably modified notion of slope.
 \par
We set 
\begin{equation}
\label{llssvi}
 \llsvi(u):= \limsup_{v \to u, \, \phi(v) \to \phi(u)}\frac{(\Ve(u) -\Ve(v))^+}{\sfd(u,v)} \quad \text{for $u \in X$}
\end{equation}
 and refer to $|\tilde{\partial}\Ve|$ as the   \emph{$\phi$-conditioned (local) slope} of $\Ve$ at $u$, 
 to highlight that, in its definition we restrict to sequences converging to $u$ with converging $\phi$-energy.
 Clearly, we have
 \[
  \llsvi(u)\leq   \lsvi(u) \quad \text{for all } u \in X.
 \]
\par
 With  the  main result of this section we establish the Hamilton-Jacobi identity for the value functional. 
\begin{theorem}
\label{prop:HJ}
Under 
\UUU the LSCC \EEE Property \ref{basic-ass} and   \eqref{phi-conv} for some $\lambda \in \R$, there holds
\begin{equation*}
\label{cruc-rel}  \sqrt{2\frac{\phi(u) - \Ve(u)}{\eps}} = G_\eps(u) = \llsvi(u) \quad \text{for all }  u  \in  \SFD(\phi).  
\end{equation*}
\end{theorem}
\noindent We split the proof in the two following lemmas.
\begin{lemma}
\label{lem:proof-ineq1}
Under 
\UUU the LSCC \EEE Property \ref{basic-ass} and  \eqref{phi-conv}, there holds 
 \begin{equation}
\label{cruc-rel1-enhanced}
 G_\eps(\ini) \leq  \llsvi(\ini) \quad \text{for all } \ini \in  \SFD(\phi) 
\end{equation}
holds. 
\end{lemma}
\begin{proof}
Let $u_\eps \in \mathcal{M}_\eps(\ini) $ and $\delta>0$.  We have 
\[
\begin{aligned}
\frac{\left(\Ve(u) -
\Ve(\umin(\delta))\right)^+}{\sfd(\umin(\delta),u)} \geq  \frac{\left(\Ve(u) -
\Ve(\umin(\delta))\right)^+}{\int_0^\delta |\umin'|(s) \dd s } &  \geq \frac{\left(\Ve(u) -
\Ve(\umin(\delta))\right)^+}{\delta^{1/2} \left( \int_0^\delta |\umin'|^2(s) \dd s\right)^{1/2}}
\\
=   \frac{\int_0^\delta G_\eps^2(u_\eps(s)) \dd s}{\delta^{1/2} \left( \int_0^\delta G_\eps^2(u_\eps(s)) \dd s\right)^{1/2}},
\end{aligned}
\]
where the latter identity  follows from  the fact that
$\umin$ is a curve of maximal slope  for $V_\eps$ w.r.t.\ $G_\eps$, cf.\
 Corollary \ref{cor:WED-GF}, so that $|\umin'| =G_\eps(u_\eps)$ a.e.\ in $(0,\infty)$.  
 Therefore,  
 \[
 \begin{aligned}
 G_\eps(u)  \stackrel{(1)}\leq   \liminf_{\delta \to 0} \left( \frac1\delta \int_0^\delta G_\eps^2(u_\eps(s)) \dd s \right)^{1/2} 
    &  \leq   \limsup_{\delta \to 0}
\frac{\left(\Ve(u) -
\Ve(\umin(\delta))\right)^+}{\sfd(\umin(\delta),u)}
\\
 & \stackrel{(2)}{\leq} \limsup_{v \to u \, \phi(v) \to \phi(u)} \frac{\left(\Ve(u) -
\Ve(v)\right)^+}{\sfd(v,u)} = \llsvi(u).
\end{aligned}
 \]
Indeed, (1) follows from the lower semicontinuity of the mapping $t\mapsto G_\eps^2(u_\eps(t))$ on $[0,\infty)$,
which is in turn guaranteed by the lower semicontinuity of $t\mapsto \phi(u_\eps(t))$ and the continuity of $t\mapsto \Ve(u_\eps(t))$ thanks to Lemma \ref{l:3.1}: the latter result  applies since $\sup_{t\in [0,\infty)}\phi(u_\eps(t)) \leq \phi(\ini)$ by
Theorems  \ref{teo:beppe1} and  \ref{teo:prop-lambda}. Further, (2)  is 
 due to the fact that $\phi(\umin(\delta))
\to \phi(u)$ as $\delta \to 0$, since $\phi \circ \ue$ is
right-continuous  at $t=0$ (cf.\ Theorems \ref{teo:beppe1} and \ref{teo:prop-lambda}).
\end{proof}

\noindent In fact, in the proof of Lemma \ref{lem:proof-ineq1} the $\lambda$-geodesic convexity
\eqref{phi-conv} has been used only in that it guarantees that the
map $t \mapsto \phi( \ue(t))$
 is bounded on $[0,\infty)$ and 
  right-continuous  at $t=0$.  Instead,
in the proof of Lemma \ref{lem:proof-ineq1-bis} below, \eqref{phi-conv} is
used more explicitly.
\begin{lemma}
\label{lem:proof-ineq1-bis}
Under 
 Property \ref{basic-ass} and \eqref{phi-conv}, there holds  $\SFD(\phi) \subset \SFD(|\tilde \partial V_\epsi|) $ and
\begin{equation}
\label{cruc-rel-inv}
 G_\eps(u) \geq  \llsvi(u) \quad \text{for all } u \in  \SFD(\phi). 
\end{equation}
\end{lemma}
\begin{proof}
Let us fix $v \in
\SFD(\phi)$ and  fix $r>0$. We set  $ \delta = \frac{\sfd(u,v)}{r}. $ 
 Let us denote by $\gamma$ the constant-speed geodesic
\[
\gamma: [0,\delta] \to X \text{ such that } \gamma(0)=u, \
\gamma(\delta) = v, \quad |\gamma'|(t)=\frac{\sfd(u,v)}{\delta}=r
\quad \foraa\,t \in (0,\delta).
\]
Furthermore, let $\vmin \in \MMM_\epsi(v) $ and, finally, let us consider the curve $\zeta
:[\delta,\infty) \to X $ given by $\zeta(t) = \vmin(t-\delta)$.
Hence $\zeta(\delta)= \vmin(0)=v$, and the curve $\tilde{u}:
[0,\infty) \to X$ defined by
\[
\tilde{u}(t) = \begin{cases} \gamma(t) & t \in [0,\delta],
\\
\zeta(t) & t \in [\delta,\infty)
\end{cases}
\]
is absolutely continuous, fulfils $\tilde{u}(0)=u$, and can thus
be chosen as a competitor in the minimum problem which defines $\Ve$.
Hence,
\[
\begin{aligned} \Ve(u)   & \leq \int_0^{\infty}
\ell_\eps(t,\tilde{u}(t),|\tilde{u}'|(t))\dd t \\ & =
\int_0^{\delta} \ell_\eps(t,\gamma(t),|\gamma'|(t))\dd t +
\int_{\delta}^{ \infty } \ell_\eps(t,\zeta(t),|\zeta'|(t))\dd t
\\ & =\int_0^{\delta}
\ell_\eps(t,\gamma(t),|\gamma'|(t))\dd t +
\rme^{-\delta/\eps}\int_{0}^{ \infty }
\ell_\eps(t,\vmin(t),|\vmin'|(t))\dd t,
\end{aligned}
\]
where the last integral equals $\Ve(v)$. Therefore,
\begin{equation*}
\label{crucial-point} \Ve(u) -\rme^{-\delta/\eps}\Ve(v) \leq
\int_0^{\delta} \rme^{-t/\eps}\left(\frac12\frac{\sfd^2(u,v)}{\delta^2}
+ \frac1\eps \phi(\gamma(t))\right) \dd t
\end{equation*}
Using that $\phi$ is $\lambda$-geodesically-convex, we conclude that
\[
\begin{aligned}
\Ve(u) -\rme^{-\delta/\eps}\Ve(v)  & \leq
\frac12\frac{\sfd^2(u,v)}{\delta^2} \int_0^{\delta} \rme^{-t/\eps} \dd
t +\frac1\eps \max\{\phi(u),\phi(v) \} \int_0^{\delta}
\rme^{-t/\eps} \dd t \nonumber\\
&\quad -\frac{\lambda  \sfd^2(u,v)
}{2\delta^2} \int_0^{\delta} \rme^{-t/\eps}  (\delta-t) t \dd t
 \\ & = \eps (1- \rme^{-\delta/\eps})
\left(\frac{\sfd^2(u,v)}{2\delta^2}+ \frac1\eps
\max\{\phi(u),\phi(v) \} \right) -\frac{\lambda
\sfd^2(u, v)
}{2\delta^2}
\int_0^{\delta} \rme^{-t/\eps}  (\delta-t) t \dd t  .
\end{aligned}
\]
We now  add to both sides of the equality the term $\Ve(v)$,  and
 divide by $\sfd(u,v)$, thus obtaining
\[
\begin{aligned}
 & \frac{\Ve(u) - \Ve(v)}{\sfd(u,v)}
\\
  & \leq
\frac12\frac{\sfd^2(u,v)}{\delta^2}  \frac{\delta}{\sfd(u,v)}
\frac{1- \rme^{-\delta/\eps}}{\frac\delta\eps}  +
\frac{\delta}{\sfd(u,v)} \frac{1- \rme^{-\delta/\eps}}{\delta}
\max\{\phi(u),\phi(v) \} \\ &  \quad- \frac{\delta}{\sfd(u,v)}
\frac{1- \rme^{-\delta/\eps}}{\delta} \Ve(v) -\frac{\lambda
\sfd^2(u, v)
}{2\delta^2} \int_0^{\delta} \rme^{-t/\eps} (\delta-t) t \dd t  =:
\Lambda_1 + \Lambda_2+ \Lambda_3+ \Lambda_4.
\end{aligned}
\]
Then, we take the
$\limsup$ as $v \to u$, with
$\phi(v) \to \phi(u)$, of the above inequality. Notice that, we may suppose that
 $\Ve(v) \leq \Ve(u)$. As $v \to u$, we have that $\delta \to
 0$, and
\[
\begin{aligned}
& \limsup_{v \to u} (\Lambda_1) \leq \frac12 r
\\
& \limsup_{v \to u} (\Lambda_2) \leq \frac{1}\eps \frac1r \phi(u)
\\
& \limsup_{v \to u} (\Lambda_3) = - \liminf_{v \to
u}\frac{\delta}{\sfd(u,v)} \frac{1- \rme^{-\delta/\eps}}{\delta} V(v)
\leq -\frac1\eps\frac1r V(u),
\end{aligned}
\]
where the second limit follows from the fact that
 $\phi(v) \to \phi(u)$, and for the third limit we have used that $\Ve$
 is lower
semicontinuous.   We also    have $\limsup_{v \to u} (\Lambda_4)=0$.
In conclusion, we find
\[
 \llsvi(u)   = \limsup_{v \to u, \ \phi(v)\to \phi(u), \ \Ve(v) \leq \Ve(u)}
\frac{\Ve(u) - \Ve(v)}{\sfd(u,v)} \leq \frac12 r +\frac1r \left(
\frac1\eps \phi(u) - \frac1\eps \Ve(u)\right)  \quad \text{for all $ r >0.$}
\]
Therefore,
\[
\frac12 \llsvi^2(u)= \sup_{r>0}\left(\llsvi(u) r - \frac12 r^2
\right) \leq \frac1\eps \phi(u) - \frac1\eps \Ve(u),
\]
whence \eqref{cruc-rel-inv}.
\end{proof}

\noindent As a straightforward consequence of the \emph{(metric)
Hamilton-Jacobi} equation \eqref{cruc-rel}
 and of Corollary \ref{cor:WED-GF}, 
 we have
\begin{corollary}
\label{V-gflow} Assume 
Property \ref{basic-ass} and \eqref{phi-conv}. Then, for every
fixed $\eps>0$ the curve $\ue$ fulfils
\begin{equation}
\label{gflow-V} \frac{\dd}{\dd t} \Ve(\ue(t))= - \frac12 |\ue'|^2(t)
- \frac12 \llsvi^2 (\ue(t)) \quad \foraa\, t \in (0,T),
\end{equation}
i.e. $\ue$ is  a curve of maximal slope for $\Ve$, with respect to   the ($L^1$-moderated) upper gradient 
 $\llsvi$. 
\end{corollary}

We conclude this Section getting further insight into the
relationship between  $|\tilde\partial\Ve|$ and $|\partial\phi|$.  The following result is an immediate corollary of Prop.\ \ref{l:5.1}, Cor.\ \ref{cor:convergence-to-relaxed}, and Thm.\ \ref{prop:HJ}. 
\begin{corollary}
\label{nice-cor} Assume  Property \ref{basic-ass} and
\eqref{phi-conv}. 
 Then, for all  $u \in \mathrm{D}(\phi)$ 
\begin{align}
 \label{e:surprise2}
\rls(u) \leq \liminf_{\eps\down 0} \llsvi(u) \leq
\limsup_{\eps\down 0} \llsvi(u) \leq |\partial \phi|(u).
\end{align}
\end{corollary}
\section{Applications}
\label{s:appl}

\UUU
The aim of this section is to present some application of the abstract
theory. In particular, we comment on the framing of our
main result Theorem \ref{th:4.1} in two different variational
settings, namely in Banach spaces (Sec.\ \ref{ss:appl.1})
and in Wasserstein spaces of probability measures (Sec.\
\ref{ss:appl.2}). 

\subsection{Application to gradient flows in Banach spaces}
\label{ss:appl.1}
\RIC We take as ambient space $X$ a reflexive 
\GGG and separable 
\EEE
Banach space $\mathcal{B}$, \UUU  with norm $\| \cdot\|$ 
\GGG and corresponding duality mapping $J: \Banach \rightrightarrows
\Banach^*$, defined by 
\begin{equation}\label{eq:76}
  \xi\in J(v)\quad\text{if and only if}\quad
  \langle \xi,v\rangle_\Banach=\|v\|^2=\|\xi\|_*^2.
\end{equation}
Given an energy functional $\phi: \mathcal{B}\to (-\infty,\infty]$
we are interested in trajectories
\GGG  $u: [0,\infty)\to
\mathrm{D}(\phi)$ solving
\begin{equation}
\label{gflow-Banach1}
J(u'(t)) + \partial^\circ \phi(u(t)) \ni 0 \quad \text{ in } \Banach^* \quad
\foraa\, t \in (0,\infty).
\end{equation}
Here, $\partial^\circ\phi$ denotes the sets of elements of minimal norm 
(the \emph{minimal section}) in the 
Fr\'echet subdifferential of $\phi$, defined at $u\in D(\phi)$ by 
\begin{equation}
\label{Frechet-subdif}
\xi \in \partial\phi(u) \text{ if and only if } \ \phi(v) -\phi(u) \geq \langle \xi, v-u\rangle_\Banach + o(\| v-u\|) \quad \text{as } v\to u,
\end{equation}
so that $ \partial\phi$ coincides with the subdifferential of convex
analysis of $\phi$ (and is thus denoted by the same symbol) as soon as
$\phi$ is convex. Since for every $u\in D(\partial\phi)$ 
the set $\partial\phi(u)$ is convex and weakly$^*$-closed in
$\Banach^*$, $\partial^\circ \phi(u)$ is well defined and satisfies
\begin{equation}
  \label{eq:77}
  |\partial\phi|(u)\le \|\xi\|_*\quad
  \forall\,\xi\in \partial\phi(u).
\end{equation}
The next result (see \cite{Ambrosio-Gigli-Savare08}) provides a
connection
between \eqref{gflow-Banach1} and curves of maximal slope.
\begin{proposition}
\label{prop:app1}
Assume Property \ref{basic-ass} with respect to the strong topology of
$\Banach$ and suppose that the graph of the Fr\'echet subdifferential of $\phi$ is
strongly-weakly closed, i.e.
\begin{equation}
\label{eq:78}
\left.
\begin{aligned}
\displaystyle 
&u_n \in \Banach,\xi_n\in\Banach^* \
\text{with }
 \xi_n\in
\partial \phi(u_n) \ \text{for all $n\in \N$}\\
\displaystyle
&  u_n\to u,\
  \xi_n \weakto \xi \ \text{as $n\to\infty$}, 
  \ \sup_n\phi(u_n)<\infty
  \end{aligned}
\right\}
\quad\Rightarrow\quad
\xi\in \partial\phi(u).
\end{equation}
There holds
  \begin{equation}
  |\partial^-\phi|(u) =|\partial\phi|(u)=
  \|\xi\|_*\quad\text{for every
  }\xi\in \partial^\circ\phi(u).\label{eq:80}
\end{equation}
Furthermore,  if $|\partial\phi|$ is a $L^\infty$-moderated upper gradient for
$\phi$, then $u$ is a 
curve of maximal slope w.r.t.\  $|\partial\phi|$ if and only if it
is a solution of the gradient flow \eqref{gflow-Banach1}.
\end{proposition}
\paragraph{\bf Functionals of the Calculus of Variations.}

In our abstract framework we can consider the   so-called  \emph{Functionals of the Calculus of Variations}, cf.\  e.g.\ \cite[Chap.\ 2.5]{Lions69}.  
We limit our analysis to one of the simplest examples, 
in the Banach space $\Banach=L^\alpha(\Omega)$, 
$1<\alpha<\infty$, where
$\Omega$ is a bounded Lipschitz open subset of $\R^d$.
We consider a \OOO Carath\'eodory \RRR integrand $f:\Omega\times\R^d\times \R\to \R$ 
$f(x,\boldsymbol z,u)$,  
which is strictly convex with respect to $\boldsymbol z$ for every
$x\in \Omega $ and $ u\in \R$, of class $\mathrm C^1$ w.r.t.~$(\boldsymbol
z,u)$ for a.e.~$x\in \Omega$, 
and satisfies the following
coercivity and growth conditions for some $p\in [\alpha,\infty)$ (for
simplicity), with 
$q=p(1-1/\alpha)$ 
and suitable positive constants $M_i$:
\begin{equation}
\label{ass:f}
\begin{cases}
  f(x,\boldsymbol z,u)\ge M_1 \vert \boldsymbol z\vert^{p} - M_2\\
  \vert \nabla_{\boldsymbol z}f(x,\boldsymbol z,u)\vert \le M_3(1 + \vert \boldsymbol
  z\vert^{p-1}+ |u|^{p-1})\\
  |f_u(x,\boldsymbol z,u)| \le M_4(1 + \vert \boldsymbol
  z\vert^{q}+ |u|^{q})
\end{cases}
\qquad
\forall
  (x,\boldsymbol z,u) \in \Omega\times \mathbb{R}^d\times \R
\end{equation}
(with $f_u$ the derivative of $f$ w.r.t.\ $u$ and $\nabla_{\boldsymbol z} f$ its gradient w.r.t.\ $\boldsymbol z$). 
We consider 
the integral functional $\phi: L^\alpha(\Omega)\to
(-\infty,\infty]$ defined via
\begin{equation}
\label{eq:calcvarfunct}
\UUU
\phi(u): = 
\begin{cases}
\displaystyle\int_{\Omega} f(x,\nabla u(x),u(x)) \dd x\,\,\,\hbox{ if }
\,u\in W^{1,p}_{0}(\Omega),
\\ 
\infty\,\,\,\,\,\,\,\,\,\,\,\,\,\,\,\,\,\,\,\,\hbox{ otherwise. }
\end{cases}\EEE
\end{equation}
\begin{lemma}
  \label{le:app1}
  The Fr\'echet subdifferential $\partial \phi(u)$ of $\phi$, with
  respect to the topology of $\Banach=L^\alpha(\Omega)$, 
  is non-empty
  if and only if $A(u):=-\mathrm{div}(\nabla_{\boldsymbol z}f(\cdot,\nabla
  u,u))+f_u(\cdot,\nabla u,u) \in L^{\alpha'}(\Omega)$, and in
    that  case  
it is given by
\[
\partial \phi(u) = \left\{A(u)
\right\}.
\]
Furthermore, $\partial \phi$  is strongly/weakly closed in $L^\alpha(\Omega)\times L^{\alpha'}(\Omega)$ along sequences with 
bounded energy, and $|\partial\phi|$ is an $L^\infty$-moderated upper gradient for $\phi$.
\end{lemma}
\begin{proof}
  In order to clarify the calculations, let us denote by $V$ the
  Banach space $W^{1,p}_0(\Omega)$; we have $V\hookrightarrow \Banach$
  and $\Banach'\hookrightarrow V'$ with compact inclusions. 

  Let us first notice that the restriction of the functional $\phi$ to
  $V$ is continuous;
  moreover, due to the strict convexity with respect to the gradient
  variable and to the coercivity assumption, 
  if $u_n\weakto u$ in $V$ and $\phi(u_n)\to\phi(u)$ 
  we deduce that $u_n\to u$ strongly in $V$.

  We consider the functional
  $\Phi:V\times V\to \R$
  \begin{displaymath}
    \Phi(v,u):=\int_\Omega f(x,\nabla v,u)\,\dd x\quad
    \text{with}\quad
    \phi(u)=\Phi(u,u),
  \end{displaymath}
  and the corresponding continuous operators $A_1:V\times V\to V'$,
  $A_2:V\times V\to \Banach'$
  \begin{displaymath}
    A_1(v,u):= -\hbox{div}(\nabla_{\boldsymbol z}f(x,\nabla
    v,u)),
    \quad
    A_2(v,u):=f_u(x,\nabla v,u),\quad
    A(u)=A_1(u,u)+A_2(u,u).
  \end{displaymath}
  By taking directional derivatives, it is immediate to check that
  $\partial\phi(u)$ may contain just the element $A(u)$. Let us first
  check that if $A(u)\in \Banach'$ then it is the Fr\'echet
  subdifferential of $\phi$. Notice that by the convexity of $\Phi$
  with respect to its first variable
  \begin{align*}
    \phi(v)-\phi(u)&-
    \langle A(u),v-u\rangle
      =\Phi(v,v)-\Phi(u,u)-
      \langle A(u),v-u\rangle \\&=
      \Phi(v,u)-\Phi(u,u)-
      \langle A_1(u,u),v-u\rangle +
      \Phi(v,v)-\Phi(v,u)-\langle A_2(u,u),v-u\rangle
    \\&\ge
      \int_0^1 \int_\Omega\OOO \left(
        A_2(v,u+\lambda(v-u))-A_2(u,u)\right)\RRR (v-u)\,\dd
        x\,\dd\lambda
    \\&\ge-\|v-u\|_{L^\alpha}
        \sigma(v,u),\quad
        \sigma(v,u):=\int_0^1 \|A_2(v,u+\lambda(v-u))-A_2(u,u)\|_{L^{\alpha'}}\dd\lambda
  \end{align*}
  In order to check that $A$ is the Fr\'echet subdifferential of
  $\phi$ it is not restrictive to assume 
  that $v\to u$ in $\Banach$ and
  $\phi(v)\to\phi(u)$, so that $v\to u$ in $V$. We then obtain that
  $\sigma(v,u)\to0$ since
  $A_2$ is continuous from $V\times V$ to $L^{\alpha'}(\Omega)$.
  
 The closedness property \eqref{eq:78} is a consequence of the fact that 
  the operator $A$ is pseudo-monotone from
  $V$ to $V'$ \cite[Prop.~2.6]{Lions69}
  and that  the inclusion $L^{\alpha'}(\Omega)\hookrightarrow
  W^{-1,p'}(\Omega)$ is compact. 
  
  Finally, in order to check that $|\partial\phi|$ (which can now   be
  identified with $\|A(\cdot)\|_{L^{\alpha'}}$ )
  is an $L^\infty$-moderated upper gradient, 
  it is sufficient to observe that the quantity $\sigma$ defined in
  the previous calculations is uniformly bounded on the sublevels of
  $\phi$. On each sublevel there exists a constant $C>0$ such that
  for every choice of $v,u$ in the sublevel there holds
  \begin{displaymath}
    \phi(v)-\phi(u)
    \ge -\big(|\partial\phi|(u)+C\big)\|v-u\|_\Banach
  \end{displaymath}
  We can then argue as in \cite[Thm.~1.2.5, Lemma 1.2.6]{Ambrosio-Gigli-Savare08}.
  \end{proof}
%

\OOO Lemma \ref{prop:app1} allows us to apply \RRR
 Proposition \ref{prop:app1}  and Lemma \ref{le:app1}
\OOO and obtain that, \RRR
given $u_0\in L^\alpha(\Omega)$, the WED minimizers converge to a
solution $u\in H^1_{\rm loc}([0,\infty);L^\alpha(\Omega))$ 
of
\begin{equation}
\label{eq:BGF}
\begin{cases}
  J_\alpha(\partial_t u) -\hbox{div}(\nabla_{\boldsymbol z}f(x,\nabla u,u)) +
  f_u(x,\nabla u,u) = 0&
  \hbox{ in } \Omega\times (0,\infty),\\
  u = 0 &
  \hbox{ on } \partial\Omega\times (0,\infty),\\
  u(\cdot, 0) = u_0(\cdot)
  &\hbox{ in } \,\Omega,
\end{cases}
\end{equation}
where $J_\alpha:L^\alpha(\Omega)\to L^{\alpha'}(\Omega)$ is the
duality mapping $J_\alpha (v)=\|v\|_{L^\alpha}^{2-\alpha}|v|^{\alpha-2}v$.
\par
It is interesting to compare this result with those obtained in \cite{BoDuMa14}
for an energy of the type of $\phi$ in $L^2(\Omega)$.
In \cite{BoDuMa14} the convexity of $f$ is assumed.
Our approach allows for some nonconvexity of $f$ and for 
doubly nonlinear equations,
at the price of imposing suitable regularity, coercivity and growth
conditions.
Of course, when $\alpha=2$, and there exists $-\lambda\ge0$ such that 
$(\boldsymbol z,u)\mapsto f(x,\boldsymbol z,u)-\frac\lambda2 u^2$ is
convex,
growth conditions could be avoided also in our setting
since  then the functional $\phi$ would be $\lambda$-convex in $L^2(\Omega)$.

\paragraph{\bf Limiting subdifferential}

When the Fr\'echet subdifferential does not satisfy the closedness
property
\eqref{eq:78}, one can consider a relaxed formulation of 
\eqref{gflow-Banach1} obtained by substituting
$\partial\phi$ with the 
\emph{limiting
  subdifferential}  of $\phi$, 
cf., e.g., \cite{Mordu06}. 
At $u \in \mathrm{D}(\phi) $ the limiting subdifferential
$\lmsbd \phi(u)  $ is defined by 
\UUU
\begin{equation}
\label{def:lmbd} \xi \in \lmsbd \phi(u)   \Leftrightarrow
\begin{cases}
\displaystyle \exists\, u_n \in \Banach,\xi_n\in\Banach^* \
\text{with }
 \xi_n\in
\partial \phi(u_n) \ \text{for all $n\in \N$,}\\
\displaystyle
  u_n\weakto \GGG u\EEE,\
  \xi_n \weakto \xi \ \text{as $n\to\infty$}, \   \ \GGG \sup_n
  \phi(u_n) \EEE
  <\infty\,.
  \end{cases}
\end{equation}
\GGG
Since $\lmsbd \phi(u)$ is not necessarily weakly$^*$ closed, 
we introduce the following notions
  \begin{equation} 
  \label{e:not-min-sel}
  \begin{aligned}
    \bar\partial_\ell\phi(u):=&\text{weak$^*$ closure of $\lmsbd \phi(u)$},\\
    \|\lmsbd^\circ \phi \|_*(u) :=&\min\left\{\| \xi \|_* \, : \ \xi \in
      \bar\partial_\ell \phi(u) \right\}, \\
    \lmsbd^\circ \phi (u):=&\left\{\xi\in \bar\partial_\ell\phi(u),
      \quad \| \xi \|_*= \|\lmsbd^\circ \phi \|_*(u)
    \right\},
  \end{aligned}
\end{equation}
and we are interested in trajectories $u: [0,\infty)\to
\mathrm{D}(\phi)$ solving
\begin{equation}
\label{gflow-Banach2}
J(u'(t)) + \lmsbd^\circ \phi(u(t)) \ni 0 \quad \text{ in } \Banach^* \quad
\foraa\, t \in (0,T).
\end{equation}
for a functional $\phi$ satisfying the LSCC Property \ref{basic-ass} of Section \ref{s:3}.

\EEE
To start with, let us record that curves of maximal slope (w.r.t.\ the
relaxed slope $|\partial^-\phi|$ of $\phi$) indeed solve equation
\eqref{gflow-Banach2}. 
This is ensured by the following result, which
is a slight adaptation of \cite[Prop.\ 6.1, 6.2]{RSS11} 
\GGG 
and 
\cite[Sect.~3.2]{Rossi-Savare06}, \EEE
cf.\ also \cite[Prop.\ 1.4.1]{Ambrosio-Gigli-Savare08}. 
\begin{proposition}
\label{prop:fromRSS11}
\RIC Assume Property \ref{basic-ass}. \UUU
There holds
\[
  \GGG
  \|\lmsbd^\circ \phi\|_* (u) \leq 
  \inf \Big\{\|\xi\|_*:\xi\in \lmsbd\phi(u)\OOO \Big\}  \RRR\le 
  |\partial^-\phi|(u)  \quad \text{for all } u \in D(|\partial^-\phi|)\,.
\]
Furthermore,  if $\phi$ fulfills the chain rule
w.r.t.\ $\lmsbd \phi$, namely:
\begin{align}
  \label{eq:69}
  & \RIC u \in H^1(0,T,\mathcal{B}), \ \xi \in
  L^{2}([0,T],\mathcal{B}^*), \UUU \ \xi \in \lmsbd\phi(u) \
  \text{a.e. in} \ (0,T) \nonumber\\
& \Rightarrow \ \phi\circ u \in AC([0,T]), \ \frac{\rm d}{{\rm d}t}
  (\phi \circ u) = \langle \xi,u'\rangle \ \text{a.e. in} \ (0,T),
\end{align}
then
\begin{enumerate}
\item
 the relaxed slope $|\partial^-\phi|$ is a strong upper gradient for $\phi$;
\item
every curve of maximal slope $u$ w.r.t.\  $|\partial^-\phi|$  is a solution of the gradient flow \eqref{gflow-Banach2}
and fulfills the minimal section principle
\begin{equation}
\label{min-sect}
-\lmsbd^\circ \phi (u(t)) \subset J(u'(t)),\quad
\GGG
\|\lmsbd^\circ \phi\|_* (u(t))=|\partial^-\phi|(u(t))
\EEE
\qquad \foraa\, t \in (0,T).
\end{equation}
\end{enumerate}
\end{proposition}

This proposition paves the way to applying  Theorem \ref{th:4.1} and
we have the following.
\begin{corollary}
\label{corWED}
Let $\phi: \Banach \to (-\infty,\infty]$ comply with the LSCC Property
\ref{basic-ass} and, in addition,  with the condition of Proposition
\ref{prop:fromRSS11}. Then, WED minimizers converge up to subsequences
to a solution $u\in H^1(0,T;\Banach)$ of \eqref{min-sect}. 
\end{corollary}
Let us now present a classe of functionals to which Corollary
\ref{corWED} applies. 
\medskip

\paragraph{\bf Dominated concave perturbations of convex functionals in Hilbert spaces.}
{
 Let us now focus on the case in which  $\Banach$ is a Hilbert space. 
 In this context,
 a class  of energies that complies with the hypotheses of Proposition \ref{prop:fromRSS11}
is given by  functionals $\phi:D(\phi)\to (-\infty, \infty]$, $D(\phi)\subset \Banach$,
admitting the decomposition 
\begin{subequations}
\begin{align}
&
\phi = \psi_1 - \psi_2 \,\,\hbox{ in } D(\phi),\\
&
\psi_1:D(\phi)\to \mathbb{R} \,\,\,\hbox{convex and l.s.c.},\\
&
\psi_2: D(\phi)\to \mathbb{R}\,\,\hbox{ convex and l.s.c. in }
  D(\phi),\,\, D(\partial\psi_1)
\subset D(\partial \psi_2).
\end{align}
\end{subequations}
 It has been shown in  \cite[Theorem 4]{Rossi-Savare06} that, 
 if, in addition, \RIC  $-\psi_2$  is a \emph{dominated} concave perturbation of $\psi_1$, namely  
\begin{equation}
\label{hyp:concave_domination}
 \begin{gathered}
\forall M\ge0\,\,\,\,\exists \rho\,<\,1, \gamma\ge 0\,\,\,\text{ such that }\,\,
\forall u\in D(\partial\psi_1)\,\,\,\text{with}\,\, \max\left\{\phi(u),\| u\|\right\}\le M
\\
\sup_{\xi_2\in \partial\psi_2(u)}\| \xi_2\|\le \rho\| \partial^\circ\psi_1(u)\| + \gamma,
\end{gathered}
\end{equation}
then $\phi$ satisfies the chain rule \eqref{eq:69}  with respect to the \emph{limiting} subdifferential $\lmsbd \phi$. 
Consequently, Proposition \ref{prop:fromRSS11}  applies. 
\par
If, in addition, 
$\phi$ complies with the LSCC properties, then 
 Corollary \ref{corWED} ensures  
that  WED minimizers converge to a solution of the 
nonconvex gradient flow \eqref{min-sect}. 
We refer to \cite{Rossi-Savare06} for further examples.

}

\subsection{A class of gradient flows in Wasserstein spaces}
\label{ss:appl.2}
 Gradient flows in metric spaces classically
arise as variational formulations of nonlinear parabolic PDEs is the
space of probability measures. By referring to
\cite{Ambrosio-Gigli-Savare08} for all details, we consider here 
  the nonlocal drift-diffusion
equation
\begin{equation}
  \label{eq:drift}
  \partial_t \rho - {\rm div} \left( \rho \nabla V +
    \frac{\nabla L_F(\rho)}{\rho} + \nabla W \ast \rho \right) =0 \quad \text{in} \
  \Rz^d \times (0,\infty)
\end{equation}
where
$V:
\Rz^d \to \Rz$ acts as confinement potential, $L_F$ is defined as $L_F(r) = rF'(r)-F(r)$
where $F:[0,\infty) \to \Rz$ is the internal-energy density, $W:
\Rz^d\times \Rz^d \to \Rz$ is the interaction potential, and the
symbol $*$ stands for classical convolution in $\Rz^d$.

Equation \eqref{eq:drift}
can be formulated as  a gradient flow in $$X= \mathscr{P}_2(\Rz^d)=\Big\{\mu
\in \mathscr{P}(\Rz^d)\ : \ \int_{\Rz^d}|x|^2{\rm d}\mu(x)<\infty\Big\}$$ endowed with the 
Wasserstein metric $\sfd=W_2$. The latter is classically given via
\begin{equation}
W_2(\mu_1,\mu_2) = \min_{\gamma \in \Gamma(\mu_1,\mu_2)}\int_{\Rz^d \times \Rz^d} |x-y|^2 {\rm
  d} \gamma(x,y)\label{eq:81}
\end{equation}

where $ \Gamma(\mu_1,\mu_2) = \{ \gamma \in \mathscr{P}(\Rz^d \times \Rz^d)\, : \ 
\pi_\#^i\gamma = \mu_i, \ i=1,2 \} $
and $\pi^i_\#$ stands for the push-forward of the measure through
the projection on the $i$-th component.

Let \RIC us consider the functional $\phi: \mathscr{P}_2(\Rz^d) \to (-\infty,\infty]$
 given by  the sum of the potential, the
internal, and the interaction energy, namely \EEE
\begin{equation}
\phi(\mu) = 
\left\{
  \begin{array}{ll}
\displaystyle\int_{\Rz^d} V(x){\rm d} \mu(x) + \int_{\Rz^d}
F(\rho(x)) {\rm d} x + \frac12 \int_{\Rz^d\times \Rz^d}W(x,y){\rm d}
(\mu \otimes \mu)(x,t)\quad&\text{if} \ \ {\rm d}\mu= \rho {\rm d} x\\
\infty&\text{else}
  \end{array}
\right.
\label{summa}
\end{equation}
\RIC (i.e., $\phi(\mu) = \infty$ is $\mu$ is not  
 absolutely continuous  w.r.t.\ the Lebesgue measure). \UUU Define $\sigma$
to be the   topology of narrow convergence in $\mathscr{P}_2(\Rz^d)$, namely $\mu_n \weak
\mu$ iff $$\int_{\Rz^d} f(x){\rm d}\mu_n(x) \to \int_{\Rz^d} f(x){\rm
  d}\mu(x)  \quad \forall f \in C_b(\Rz^d) \ \text{(continuous and
  bounded)}.$$
While referring to \cite{Ambrosio-Gigli-Savare08} for a full discussion on this
setting, we assume here that
\begin{align}
&V:\Rz^d \to \Rz \ \text{is $\lambda$-convex},  \ \ \lim_{|x| \to
  \infty}{|x|^{-2}}{V(x)}=\infty, \label{VV}\\
&F:[0,\infty) \to \Rz \ \text{is convex and differentiable with $F(0)=0$},\nonumber\\
&\lim_{r\to \infty} F(r)/r= \infty, \ \ F(r)/r^\alpha \ \
\text{is bounded below as $r\to 0+$ for some} \ \alpha >d/(d+2),\nonumber\\
&\exists C_F \geq 0 \  \forall x,
\, y \geq 0 \, : \ F(x+y)\leq C_F(1 + F(x)+F(y)), \nonumber\\
& r \mapsto r^d F(r^{-d}) \ \ \text{is convex and nonincreasing}\label{FF},\\
& W:\Rz^d\times \Rz^d \to \Rz \ \ \text{is even, \GGG
  $\lambda$-convex,
  \EEE differentiable, and}\nonumber\\
&\exists C_W \geq 0 \  \forall x,
\, y \in \Rz^d \,: \ W(x+y)\leq C_W(1 + W(x)+W(y)) \label{WW}
\end{align}
\GGG for some $\lambda\le 0$. \EEE
Let us record that these assumptions include the two classical choices
$F(r) = r \ln r$ and $F(r) = r^m$ for $m>1$. By combining
\cite[Thm. 11.1.3]{Ambrosio-Gigli-Savare08} and
\cite[Prop. 7.4]{RSS11} one finds that  
\GGG
\begin{enumerate}
\item the functional $\phi$ is $\lambda$-convex
  in $\mathscr P_2(\R^d)$;
\item
equation \eqref{eq:drift} can be equivalently written as \GGG
\begin{equation}
v_t +\partial_{\rm W} \phi(\mu_t) \ni 0\ \ \text{for a.a.} \ t \in (0,\infty)\label{abs}
\end{equation}
where the Borel velocity field $v_t=v(\cdot,t) : \Rz^d \to \Rz^d$, $t\in (0,\infty)$, is
associated with $\mu_t=\mu(\cdot,t)$ by letting 
$v_t\in L^2(\Rz^d,\mu_t)$, $|\mu_t'| = \|
v(\cdot,t)\|^2_{L^2(\Rz^d,\mu_t)}$ and requiring the continuity equation
$$\partial_t \mu + {\rm div} (v\mu) =0$$
to hold in the sense of
distributions in $\Rz^d \times (0,\infty)$. The Wasserstein
subdifferential $\partial_{\rm W}
\phi(\mu)$ in relation \eqref{abs} at a measure $\mu\in D(\phi)$ is
given by 
\begin{align*}
  &-v \in \partial_{\rm W} \phi(\mu) \ \ \Leftrightarrow \ \ -v \in
  L^2(\Rz^d,\mu) \ \ \text{and} \\ 
&\phi(\nu) - \phi(\mu) \geq \int_{\Rz^d \times
  \Rz^d} (-v(x))\cdot (y-x) {\rm d} \gamma_o(x,y) +\frac\lambda2
  W^2_2(\mu,\nu) \quad \text{for every} \ \nu\ \text{in}  \ \mathscr{P}_2(\Rz^d)
\end{align*}
where $\gamma_o\in \Gamma(\mu,\nu)$ is the unique optimal plan attaining the minimum in \eqref{eq:81}.
\item  \RIC curves of maximal slope for $\phi$ from  \eqref{summa}
  solve  \eqref{abs} 
\GGG and are unique.
\end{enumerate}
\EEE
We are now in the position of presenting the application of our
abstract Theorem \ref{th:4.1} in this setting.
\begin{corollary}\label{prope}
  Assume \eqref{VV}-\eqref{WW}. Then, the functional
  $\phi$ from \eqref{summa} fulfills the LSCC Property \ref{basic-ass}
  and is $\lambda$-geodesically convex on $\mathscr{P}_2$. In
  particular, the \RIC local slope $|\partial \phi|$ coincides with its relaxation $|\partial^-\phi|$ and it is  an upper gradient.
  Thus, \UUU
  minimizers of the WED functionals converge to
  curves of maximal slope for $\phi$ w.r.t. $|\partial\phi|$.
\end{corollary}
\GGG Notice that in the previous statement  the \emph{whole} family of WED
minimizers converge and there is no need to extract 
subsequences: this fact depends on the uniqueness of curves of maximal
slope for $\phi$.
\RIC Without entering into the  details of the proof of Corollary \ref{prope}, \UUU  let us mention here that 
\OOO the \RRR LSCC Property \ref{basic-ass} and the geodesic convexity of $\phi$
have been checked in
\cite{Ambrosio-Gigli-Savare08}.
 In particular, the slope
$|\partial\phi|$ is a ($L^\infty$-moderated) upper gradient
\cite[Cor. 2.4.10]{Ambrosio-Gigli-Savare08} and Theorem \ref{th:4.1} applies.
\EEE


\appendix
\section{A Finsler distance induced by the energy}
\label{ss:2.5}
In this section 
we will briefly introduce Finsler extended distances
induced by a general \UUU function  $\frf:X\to [1,\infty]$ via \EEE
\begin{definition}
  The Finsler extended distance $\sfd_\frf$ associated with  $\frf$   is defined by  
  \begin{equation}
    \label{Finsler-distance}
    \sfd_\frf(u_0,u_1): = \inf\left \{ \int_a^b
      \frf(\teta(t)) |\teta'|(t) \dd  t \, : \ 
      \teta \in \mathrm{AC}([a,b];X),\quad
      \vartheta(a)=u_0,\ \vartheta(b)=u_1\right\}
  \end{equation}
  where $\sfd_\frf(u_0,u_1) =\infty$ if
  there are no absolutely continuous curves connecting $u_0$ to $u_1$.
\end{definition}
It is easy to check that $\sfd_\frf$ is an extended distance
(i.e.~it satisfies all the usual axioms defining a distance but
it may assume the value $\infty$), satisfying
\begin{equation}
  \label{stronger-distance}
  \sfd_\frf(u_0,u_1) \geq \sfd(u_0,u_1) \qquad \text{for all } u_0,\, u_1 \in X.
\end{equation}
By a linear change of variable, it is also possibile to fix
$[a,b]=[0,1]$ in \eqref{Finsler-distance}. 

\UUU In the following we assume that \EEE
\begin{equation}
\frf:X\to [1,\infty]\quad\text{\UUU satisfies \EEE 
  assumptions
\eqref{basic-ass-1} and \eqref{basic-ass-3}.}\label{eq:20}
\end{equation}
Our main motivating examples are provided by the choices
\begin{equation}
  \label{eq:10}
  \frf(x):=\big(\phi(x)\lor 1\big)^{1/2},\quad
  \frf(x):=
  \begin{cases}
    1&\text{if }\phi(x)\le c,\\
    \infty&\text{otherwise.}
  \end{cases}
\end{equation}
 Indeed, working with the distance induced by the latter functional would be helpful to force WED minimizers with bounded energy. 
 On the other hand,  we  will see that the Finsler distance 
\begin{equation}
\label{induced-phi}
\sfd_\phi  \text{ induced via \eqref{Finsler-distance}
 by }   \frf=\big(\phi\lor 1\big)^{1/2}
\end{equation}
 can be thought of as a \emph{natural metric} in the context of WED minimization,
cf.\ Theorem  \ref{cont-wrt-dphi} ahead. 


%
\begin{lemma}
  \label{le:basic}
  Let $(u_{n},v_n)$ be a sequence $\sigma$-converging to $(u,v)$
  and let $\vartheta_n\in \AC([a_n,b_n];X)$ be a sequence of curves
  connecting $u_n$ to $v_n$ and satisfying
  \begin{equation}
    \label{eq:24}
    \liminf_{n\to\infty}\int_{a_n}^{b_n}\frf(\vartheta_n)|\vartheta_n'|\,\dd
    t\le F<\infty.
  \end{equation}
  Then there exists a curve $\vartheta\in \AC([a,b];X)$ 
  connecting $u$ to $v$ and satisfying 
    $\int_{a}^{b}\frf(\vartheta)|\vartheta'|\,\dd
    t\le F.$ 
\end{lemma}
\begin{proof}
  By the reparametrization technique 
  stated in Lemma 
  \ref{le:reparametrization} it is not restrictive to suppose
  that 
  \begin{equation}
    \label{eq:25}
    a_n=0,\ |\vartheta_n'|\equiv 1\quad \text{$\Leb 1$-a.e.~in
    }  (0,b_n),  \quad
    b_n\le \int_0^{b_n}\frf(\vartheta_n)\,\dd s\le F.
  \end{equation}
  Applying the compactness Theorem \ref{thm:main-compactness}
  (with $\frf$ instead of $\phi$) and setting
  $b:=\liminf_{n\to\infty}b_n$, 
  we find a suitable subsequence
  $k\mapsto n_k$ and a limit $1$-Lipschitz curve $\vartheta$ defined
  in $[0,b]$  such that 
  \begin{gather*}
    \vartheta_{n_k}(s)\weaksigma \vartheta(s)\quad\text{for every }s\in
    [0,b),\quad
    \vartheta_{n_k}( b_{n_k} )=v_{n_k}  \weaksigma \vartheta(b)=v,\\
    \int_0^b \frf(\vartheta)|\vartheta'|\,\dd s\le 
    \int_0^b \frf(\vartheta)\,\dd s\le
    \liminf_{k\to\infty}\int_0^{b_{n_k}} \frf(\vartheta_{n_k})\,\dd s\le F. \qedhere
  \end{gather*}
\end{proof}
 With our next result we examine the relationships between the properties of absolute continuity w.r.t.\ $\sfd$ and
$\sfd_\frf$, and the related metric derivatives.  \UUU We hence use
\EEE the notation 
$ |u'|_\sfd $ for the metric derivative w.r.t.\ $\sfd$, to distinguish it from  the metric derivative w.r.t.\ $\sfd_\frf$. 
\begin{proposition}
\label{lemma:dphi}
If \eqref{eq:20} holds, 
for every $u_0,u_1\in X$ at finite $\sfd_\frf$-distance
the $\inf$ in \eqref{Finsler-distance} is attained,
$\sfd_\frf$ is lower semicontinuous in $X\times X$ 
with respect to the product
topology induced by $\sigma$, and 
$(X,\sfd_\frf)$ is a complete extended metric space.
%
%
Moreover,
\begin{align}
\label{ac-curves}
\begin{aligned}
&
 u \in \AC([a,b]; (X,\sfd_\frf)) 
 \quad \text{ if and only if } \quad  u \in \AC([a,b]; (X,\sfd)) 
 \text{ and } (\frf\circ u)\, |u'|_\sfd \in L^1(a,b)
  \\
 &
 \text{with, in this case,}
\quad
|u'|_{\sfd_\frf}(t) =\frf(u(t)) |u'|_\sfd(t) \qquad \foraa\, t \in (0,T).
 \end{aligned}
\end{align}
\end{proposition}
\begin{proof}
The $\sigma$-lower semicontinuity of $\sfd_\frf$ and
the existence of minimizers for \eqref{Finsler-distance}
are immediate consequences of the previous Lemma 
\ref{le:basic}.  Then,
$\sfd_\frf$ is  lower semicontinuous also with respect to $\sfd$:
since $(X,\sfd)$ is complete and $\sfd_\frf\ge\sfd$ we 
obtain the completeness of $(X,\sfd_\frf)$.
\par
As for \eqref{ac-curves}, clearly any absolutely continuous curve $u$ on
$[a,b]$ with $(\frf\circ u)  |u'|_\sfd  \in L^1(a,b)$ is
absolutely continuous with respect to the distance $\sfd_\frf$.
Conversely, for every $u \in \AC([a,b];(X,\sfd_\frf))$ there holds
\begin{equation}
\label{from-rms} |u'|_{\sfd_\frf}(t) \geq \frf(u(t))  |u'|_\sfd (t)
\qquad \foraa\, t \in (a,b).
\end{equation}
This can be shown by adapting the argument from the proof of
\cite[Lemma 6.4]{Rossi-Mielke-Savare08}.
From \eqref{from-rms}, we infer that $(\frf\circ
u) |u'|_\sfd  \in L^1(a,b)$. The converse inequality of
\eqref{from-rms} can be trivially checked, and \eqref{ac-curves}
ensues.
 \end{proof}
 \begin{corollary}
   \label{cor:df-compatible}
   If $(X,\sfd,\sigma)$ is a compatible
   metric-topological space   in the sense of \eqref{basic-assu} and $\frf$ complies with  \eqref{eq:20},
   then also $(X,\sfd_\frf,\sigma)$ is a 
   compatible
   metric-topological space.
 \end{corollary}
We conclude this section with
the following equivalent representation  for $\sfd_\frf$. 
\begin{lemma}
\label{l:prel-Lagra}
For every $u_0,u_1\in X$
\begin{equation}
\label{lagrangian}
\begin{aligned}
&
\sfd_\frf(u_0,u_1)= \inf \left\{ \int_{s_0}^{s_1} \left(
    \frac12|\teta'|^2+ \frac12 \frf^2(\teta) \right)  \dd  s :
  \  \teta \in \AC^2([s_0,s_1];X),\ \teta(s_i)=u_i
\right \},
 \end{aligned}
 \end{equation}
and the infimum is attained whenever it is finite.
\end{lemma}
\begin{proof}
By the Cauchy inequality it is immediate to check the inequality $\le$ in \eqref{lagrangian}.
In order to prove the converse inequality let us suppose that $\sfd_\frf(u_0,u_1)<\infty$
and let us choose an optimal curve $\vartheta\in \AC([0,1];(X,\sfd))$
connecting $u_0$ to $u_1$;
replacing $\vartheta$ with the reparametrized
curve $\vartheta_\frf$ provided by \eqref{eq:47} from 
Lemma \ref{le:reparametrization} we conclude.
\end{proof}
\paragraph{\bf Properties of the value function with respect to the Finsler distance induced by the energy}
Our next result shows that  the value function 
enjoys finer properties  with respect to $\sfd_\phi$ from \eqref{induced-phi}. Namely,
$\Ve$ is 
indeed continuous w.r.t.\ $\sfd_\phi$ under the sole \UUU LSCC Property \EEE \ref{basic-ass}, whereas 
we have been able to prove its 
 continuity w.r.t.\ $\sfd$, along sequences with bounded energy, only under the additional $\lambda$-convexity condition 
\eqref{phi-conv}. Furthermore, $G_\eps$ is a \emph{strong} upper gradient, in the \emph{standard} sense of \eqref{e:2.3},  w.r.t.\ $\sfd_\phi$, while $G_\eps$ is just an \emph{$L^1$-moderated} upper gradient (cf.\ \eqref{eq:45}) w.r.t.\ the distance $\sfd$. 
\begin{theorem}
\label{cont-wrt-dphi}
Assume \UUU the LSCC \EEE Property \ref{basic-ass}. Then,
\begin{enumerate}
\item for all $\UUU (x_n)_n \EEE, \, x \in X$ we have that 
   $\sfd_\phi(x_n, x) \to 0$ as $n\to\infty$ implies $\lim_{n\to\infty}\Ve(x_n) = \Ve(x)$; 
\item  for every curve $u\in \AC([a,b]; (X,\sfd_\phi))$ the map $t\mapsto \Ve(u(t))$ is absolutely continuous, and there holds
\begin{equation}
\label{fine..}
\left| \frac{\dd}{ \dd t } \Ve(u(t)) \right| \leq G_\eps(u(t))|u'|(t) \qquad \foraa\, t \in (a,b).
\end{equation}
\end{enumerate}
\end{theorem}
\begin{proof}
$\vartriangleright (1)$: In order to show the continuity of $\Ve$, we
\UUU argue as in \EEE the proof of Lemma \ref{l:3.1}. Namely, for a 
given minimizer $\umin \in \ACini x\eps$ for $\calI_\eps$ (which exists since $x\in \mathrm{D}(\phi)$),
we exhibit a sequence $(u_n)_n$ with $u_n  \in 
 \ACini {x_n}\eps$ 
 for every $n\in \N$, 
 fulfilling 
 \eqref{gamma-conv-argum}: in fact, observe that $\liminf_{n\to\infty}\Ve(x_n)\geq \Ve(x)$ again  follows from Lemma \ref{lemma1} as $\sfd_\phi$-convergence 
 implies $\sigma$-convergence via \eqref{stronger-distance}. 
To construct $(u_n)_n$,  for every $n \in \N$  we consider an optimal curve $\teta_n \in \AC ([0,1];X)$ for
$\sfd_\phi(x_n,x)$.  We exploit \eqref{eq:47} from Lemma \ref{le:reparametrization} to 
reparametrize the curves $(\teta_n)_n$ to curves
$\tilde{\teta}_n : [0,\tau_n] \to X$ such that
\begin{equation}
\label{connecting-teta-n}
\int_0^{\tau_n} |\tilde{\teta}_n'|^2(s) \dd s  = \int_0^{\tau_n} (1\vee \phi(\tilde\teta_n(s))) \dd s = \int_0^1 |\teta_n'|(t) (1\vee \phi(\teta_n(t)))^{1/2} \dd t = \dphi(\ini_n,\ini)\to 0\,.
\end{equation}
Hence, we define $u_n: [0,\infty) \to X$ by
\[
u_n(t) := \begin{cases}
\tilde\teta_n(t) & t \in [0,\tau_n],
\\
\umin(t)  & t \in [\tau_n,\infty),
\end{cases}
\]
so that
\[
 \mathcal{I}_\eps[u_n] = \int_0^{\tau_n} \ell_\eps (t,\tilde\teta_n(t),|\tilde\teta_n'|(t)) \dd t +
 \int_{\tau_n}^{\infty} \ell_\eps (t,\umin(t),|\umin'|(t)) \dd t  =:
 I_1+I_2\,.
\]
Now,
\[
\begin{aligned}
I_1  &   =\int_0^{\tau_n} e^{-t/\eps} \left( \frac12 |\tilde\teta_n'|^2(t) + \frac1{\eps}\phi(\tilde\teta_n(t))\right) \dd  t\\ & 
\leq  C  \int_0^{\tau_n} \left(  |\tilde{\teta}_n'|^2(t) + (1\vee \phi(\tilde\teta_n(t)))  \right) \dd t = C \dphi(\ini_n,\ini) \to 0
\end{aligned}
\]
 as $n \to \infty$.
 On the other hand, we clearly have that $I_2$ converges to
$\mathcal{I}_\eps[\umin] $ as $n \to \infty$ as $\tau_n\down 0$.  All in all, we have shown that $\limsup_{n\to\infty} \calI_\eps[u_n]\leq \calI_\eps[\umin]$ as desired. 
\\
$\vartriangleright$ (2): Let $u\in \AC([a,b]; (X,\sfd_\phi))$: then, $1 \vee (\phi{\circ} u) \in L^1(a,b)$. 
We again resort to   \eqref{eq:47}  and reparametrize $u$ to a curve $\tilde u : [\tilde a,\tilde b]\to X$ fulfilling \eqref{connecting-teta-n}, so that $\tilde u \in \AC^2([\tilde a,\tilde b];X) $ and $\phi \circ \tilde u \in L^1(\tilde a,\tilde b)$. Therefore, $\Ve \circ \tilde u \in  L^1(\tilde a,\tilde b)$ and $ G_\eps \circ \tilde u \in  L^2(\tilde a,\tilde b)$, so that we are in a position to apply Thm.\ \ref{thm:upper-gradient} and conclude that $s\mapsto \Ve(\tilde u(s))$ is absolutely continuous, with
\[
\left| \frac{\dd}{ \dd s } \Ve(\tilde{u}(s)) \right| \leq G_\eps(\tilde u(s))|\tilde{u}'|(s) \qquad \foraa\, s \in (\tilde a,\tilde b),
\]
which gives \eqref{fine..}. This concludes the proof of (2).
\end{proof}
\EEE

\bibliographystyle{amsalpha}

\def\cprime{$'$} \def\cprime{$'$}
\providecommand{\bysame}{\leavevmode\hbox to3em{\hrulefill}\thinspace}
\providecommand{\MR}{\relax\ifhmode\unskip\space\fi MR }

\end{document}